\documentclass{amsart}

\usepackage{amsmath, amsthm, amssymb, amscd}

\usepackage{enumerate}      
\usepackage{stmaryrd}       
\usepackage{xspace}         
\usepackage{verbatim}       
\usepackage{url}            

\usepackage[usenames,dvipsnames]{color}

\usepackage{tikz}
\usepackage{tikz-cd}
\usetikzlibrary{
  arrows.meta,               
  decorations.pathmorphing,  
  decorations.markings,      
  positioning                
}
\usetikzlibrary{decorations.pathreplacing}
\usetikzlibrary{patterns}
\usepackage[colorlinks=true]{hyperref}
\usepackage[active]{srcltx}
\usepackage[all]{xypic}
\SelectTips{cm}{}

{

   \newtheorem{theorem}[subsubsection]{Theorem}
      \newtheorem*{theorem*}{Theorem}
   \newtheorem{proposition}[subsubsection]{Proposition}

   \newtheorem{lemma}[subsubsection]{Lemma}

   \newtheorem{corollary}[subsubsection]{Corollary}
   
   \newtheorem*{conjecture*}{Conjecture}

}
{\theoremstyle{definition}
          \newtheorem*{exercise*}{Exercise}
   
   \newtheorem{example}[subsubsection]{Example}
   \newtheorem*{example*}{Example}
   \newtheorem{definition}[subsubsection]{Definition}
   
   \newtheorem*{definition*}{Definition}
   
   \newtheorem{remark}[subsubsection]{Remark}

}
\newcommand{\RR}{{\mathbb{R}}}

\newcommand{\QQ}{{\mathbb{Q}}}
\newcommand{\NN}{{\mathbb{N}}}

\newcommand{\ZZ}{{\mathbb{Z}}}

\newcommand{\GG}{{\mathbb{G}}}

\renewcommand{\AA}{{\mathbb{A}}}

\newcommand{\m}{{\mathfrak{m}}}

\newcommand{\cA}{{\mathcal A}}

\newcommand{\cC}{{\mathcal C}}
\renewcommand{\cD}{{\mathcal D}}

\newcommand{\cI}{{\mathcal I}}
\newcommand{\cJ}{{\mathcal J}}

\newcommand{\cO}{{\mathcal O}}

\renewcommand{\cR}{{\mathcal R}}

\def\<{\langle}
\def\>{\rangle}
\newcommand{\spa}{{\operatorname{span}}}

\newcommand{\Ver}{{\operatorname{Vert}}}
\newcommand{\Int}{{\operatorname{Int}}}
\newcommand{\ext}{{\operatorname{ext}}}
\newcommand{\inn}{\operatorname{in}}

\newcommand{\Spec}{\operatorname{Spec}}

\newcommand{\cProj}{{{\mathcal P}roj}}

\newcommand{\Sing}{{\operatorname{Sing}}}

\newcommand{\codim}{\operatorname{codim}}

\newcommand{\lcm}{{\operatorname{lcm}}}

\newcommand{\Bl}{{\operatorname{Bl_{rs}}}}

\def\:{{\colon}}
\def\.{{,\dots,}}

\newcommand{\double}{\genfrac..{0pt}1
{\raise -1pt\hbox{$\scriptstyle\longrightarrow$}}{\raise 3pt\hbox
{$\scriptstyle\longrightarrow$}}}

\renewcommand{\setminus}{\smallsetminus}

\def\tototi{\mathbin{\mathop{\otimes}\limits^{\raise-1pt\hbox
{$\scriptscriptstyle {\rm L}$}}}}

\def\indlim{\mathop{\vrule width0pt height7pt depth
4pt\smash{\lim\limits_{\raise 1pt\hbox to 14.5pt
{\rightarrowfill}}}}}
\def\projlim{\mathop{\vrule width0pt height7pt depth
4pt\smash{\lim\limits_{\raise 1pt\hbox to 14.5pt
{\leftarrowfill}}}}}

\newcommand\displaceamount{3pt}

\newcommand{\doubledown}{\ar@<\displaceamount>[d]\ar@<-\displaceamount>[d]}

\newcommand{\doubleup}{\ar@<\displaceamount>[u]\ar@<-\displaceamount>[u]}

\newcommand{\doubleright}{\ar@<\displaceamount>[r]\ar@<-\displaceamount>[r]}

\newcommand{\ord}{{\operatorname{ord}}}

\newcommand{\gr}{{\operatorname{gr}}}

\newcommand{\inv}{{\operatorname{inv}}}
\newcommand{\Inv}{{\operatorname{Inv}}}

\newcommand{\inte}{{\operatorname{int}}}
\newcommand{\maxinv}{{\operatorname{maxinv}}}

\def\supp{{\operatorname{supp}}}

\begin{document}

\title{Functorial resolution by torus actions}

%

\author[J. W{\l}odarczyk] {Jaros{\l}aw W{\l}odarczyk}
\address{Department of Mathematics, Purdue University\\
150 N. University Street,\\ West Lafayette, IN 47907-2067}
\email{wlodarcz@purdue.edu}

\thanks{This research is supported by  BSF grant 2014365 and Simons Foundation grant  MPS-TSM-00008103} 

\date{\today}

\begin{abstract}
We present a simple and fast embedded resolution of varieties and principalization of ideals using torus actions on ambient smooth varieties with simple normal crossings (SNC) divisors. The canonical functorial resolution in characteristic zero is achieved via the newly introduced \emph{cobordant blow-ups} along smooth weighted centers. These centers are defined by a geometric invariant measuring the singularities on smooth schemes with SNC divisors.

The output is a smooth variety with a torus action and an SNC exceptional divisor. Its geometric quotient is birational to the resolved variety, has only abelian quotient singularities, and can be desingularized by purely combinatorial methods. The method is rooted in ideas from the joint work with Abramovich and Temkin \cite{ATW-weighted} and is closely related to McQuillan's resolution via stack-theoretic weighted blow-ups \cite{Marzo-McQuillan}.

As an application, we establish resolution results for certain classes of singularities in positive and mixed characteristic.

This paper is a shortened and revised version of an earlier preprint.
\end{abstract}
\maketitle
\setcounter{tocdepth}{1}
\setcounter{tocdepth}{2}
\tableofcontents
\section{Introduction}
Actions of the multiplicative group \( G_m \) play a central role in birational geometry and Mori theory, as recognized in the work of Reid, Thaddeus, and others (see \cite{Tha1, Tha2, Tha3, Reid92, DH98}). Reid \cite{Reid02} emphasized the use of \( G_m \)-equivariant weighted blow-ups and flips in birational transformations, a viewpoint that became foundational in the proof of the Weak Factorization Theorem through the concept of \emph{birational cobordism} (\cite{W-Cobordism, W-Factor, AKMW}). (See Figures~\ref{F1}--\ref{F5} for key illustrations.)

This paper is a \emph{shortened and revised version} of an earlier arXiv preprint, originally posted in 2022 \cite{W22} and updated in 2023 \cite{W23}. It develops a more streamlined and geometric approach to resolution of singularities using \emph{cobordant blow-ups}-a smooth, torus-equivariant analogue of stack-theoretic constructions such as weighted and Kummer blow-ups.

Our method, formulated entirely within the language of the logarithmic schemes with torus actions, offers an alternative but largely equivalent perspective to those in \cite{Marzo-McQuillan,ATW-weighted}, (see also \cite{ATW-principalization, ATW-relative,Quek} )which uses stack-theoretic weighted blow-ups to implement resolution. This  conceptual shift resolves several key limitations of the classical resolution paradigm. Smooth centers- implicitly assuming uniform weights-are often ill-suited to reflect the intrinsic geometry of singularities. They rarely lead to absolute improvement of singularities and typically require auxiliary logarithmic structures to track relative progress which becomes especially problematic in positive or mixed characteristic settings.

By contrast, weighted centers align more naturally with the geometry of singularities through the use of weighted normal cones. In characteristic zero, they enable simplification of the resolution algorithm by relying on a geometric invariant that guarantees absolute improvement. They also provide a coherent framework in positive characteristic, where smooth centers are known to lead to many additional difficulties related to the logarithmic structure-which can be avoided when using weighted centers. Furthermore, they are indispensable in the resolution of singular foliations, where compatibility with the weighted structure is essential (see \cite{Panazzolo}).

This improvement is not merely conceptual. In practice, the use of weighted centers allows singularities to be approximated more accurately, and their resolutions to be guided more effectively by the associated graded structures.

Cobordant blow-ups at weighted centers further stimplify the resolution process. This approach is compatible with all characteristics and circumvents several difficulties associated with traditional smooth centers (see \cite{Abh67, Moh87, Moh96, Hau96, W08, CP08, CP09, Cut11, BV13, KM16, HP18}). 

In characteristic zero, our algorithm yields an efficient resolution strategy based on \emph{rational Rees algebras}, which define canonical invariants and centers using \(\mathbb{Q}\)-gradings. This contrasts with classical approaches that rely on rescaling, equivalence relations, or homogenization (see \cite{Villamayor, Bierstone-Milman, Wlodarczyk, Encinas-Hauser, Kollar, Bierstone-Milman-funct}). Note that  Rees algebras were already used  in \cite{Encinas-Villamayor,BV}  in the context of the classical resolution by smooth centers as a replacement of {\it idealistic exponents}. On the other hand $\ZZ$-graded Rees algebra approach was also pursued by Quek in \cite{Quek} for the purpose of his stack-theoretic logarithmic resolution.

The algorithm presented here draws on fundamental concepts from the standard Hironaka-style framework-including admissibility, coefficient ideals, and maximal contact -as developed in \cite{Hironaka}, \cite{Villamayor}, \cite{Bierstone-Milman}, \cite{Wlodarczyk}, \cite{Encinas-Hauser}, \cite{Encinas-Villamayor}, and \cite{Kollar}. These notions are reinterpreted through the lens of rational  Rees algebras and implemented via cobordant blow-ups.

 In our framework, the centers are constructed recursively by adding graded maximal contacts  associated with  graded coefficients  of generators until a maximal admissible center is identified. While intermediate steps may depend on auxiliary choices, the final output is automatically canonical and governed by invariant data. This approach substantially reduces the complexity and length of computations, as illustrated in Examples in Section~\ref{ex}.  

Finally, in contrast to earlier treatments \cite{Marzo-McQuillan, ATW-weighted}, our construction of weighted resolution and principalization is embedded directly in an \emph{SNC} (simple normal crossing) setup.

The method extends naturally to certain classes of schemes in positive and mixed characteristic. We introduce a class of \emph{almost homogeneous singularities} (see Section~\ref{free}; Theorems~\ref{Homoge1}, \ref{Homoge3}), defined via their weighted tangent cones, and show that they can be resolved using cobordant blow-ups.

A key motivating example-presented in Section~\ref{ex} and generalized in Section~\ref{homog} to arbitrary characteristic-demonstrates that a Brieskorn singularity of the form
\[
f = a_1x_1^{c_1} + \cdots + a_n x_n^{c_n}
\]
can be resolved in a single step via a canonical cobordant blow-up, at least in characteristic zero and under additional assumptions  also in positive characteristic. This example illustrates the recursive construction of invariants and centers in characteristic zero and serves as a prototype for more general almost homogeneous singularities in positive characteristic. We note that resolving Brieskorn singularities or even quadrics  of the form \[
f = a_1x_1^{2} + \cdots + a_n x_n^{2}
\] in characteristic $2$ is extremely difficult over nonperfect fields, and only partial results are currently known~\cite{CPS}.

The ideas in this paper are developed further in \cite{Wlodarczyk-Cox}, where cobordant blow-ups are extended to arbitrary proper birational morphisms using \emph{Cox rings}, and resolution is performed for blow-ups of locally monomial ideals. The output is a smooth scheme with a torus action whose geometric quotient has abelian quotient singularities. These can be resolved canonically by toroidal methods \cite{Wlodarczyk-functorial-toroidal} or by the destackification algorithm of Bergh and Rydh \cite{Bergh-Rydh}.

In addition to their role in resolving varieties, cobordant blow-ups have proven effective in the resolution of singular \emph{foliations}. Unlike classical or weighted blow-ups, cobordant blow-ups may transform some  singular foliations into  nonsingular ones, as illustrated in Section~\ref{foliations} (see Figure~\ref{F5}). This phenomenon does not occur in standard methods. The use of cobordant blow-ups in this context  appears in joint work with Abramovich, Belotto, and Temkin \cite{ABTW25}, further highlighting their broad applicability in singularity theory.
\subsection*{Acknowledgments.} Submitted to the special issue dedicated to James McKernan on the occasion of his 60\textsuperscript{th} birthday.

The author would like to thank Dan Abramovich, Kenji Matsuki, Mircea Musta\c{t}\u{a}, Michael Temkin, Willem Veys, and many others for helpful discussions and suggestions. 
\subsection{Main Theorems in Characteristic Zero}

The results on logarithmic resolution and principalization presented in this section were  established in the initial version of this paper, posted on arXiv in 2022~\cite{W22} (see also \cite{W23}). A related but distinctive approach to the weighted logarithmic technique was subsequently proposed in a recent work~\cite{ABQTW}.

The theorems below provide a functorial resolution of singularities in characteristic zero via sequences of cobordant blow-ups at smooth weighted centers. These are formulated in terms of torus actions and SNC divisors, extending and refining the stack-theoretic approach of~\cite[Section 1.2]{ATW-weighted}.
\subsubsection{Functorial principalization}

\begin{theorem}\label{principalization}
Let \( X \) be a smooth variety over a field \( k \) of characteristic zero, \( E \) an SNC divisor, and \( \mathcal{I} \subset \mathcal{O}_X \) an ideal sheaf. Then there exists a functorial sequence of cobordant blow-ups
\[
X = X_0 \leftarrow X_1 \leftarrow \cdots \leftarrow X_k = X'
\]
at smooth weighted centers \( V(\mathcal{J}_i) \subseteq V(\mathcal{I}_i) \), such that:
\begin{enumerate}
    \item Each \( X_i \) admits  an action  of torus \( T_i \simeq G_m^i \) with finite stabilizers, and the geometric quotients \( X_i / T_i \) exist.
   \item Set \( E_0 := E \), and for \( i \geq 1 \), let \( E_i \) denote the total transform of \( E_{i-1} \). The divisors \( E_i \) on \( X_i \) are \( T_i \)-stable and have simple normal crossings (SNC).
 \item The centers are \( T_i \)-invariant and are  adapted  to \( E_i \).

       \item The final ideal becomes principal: \( \mathcal{O}_{X'} \cdot \mathcal{I} = \mathcal{O}_{X'}(-D') \) for an SNC divisor \( D' \), with \( X' \setminus D' \simeq (X \setminus V(\mathcal{I})) \times T \).
    \item The sequence descends to sequences of weighted blow-ups on both:
    \begin{itemize}
        \item geometric quotients: \( X/T_0 \leftarrow \cdots \leftarrow X'/T_k \), where the total transform of $\cI$ becomes principal;
        \item stack-theoretic quotients: \( [X/T_0] \leftarrow \cdots \leftarrow [X'/T_k] \), where  the total transform of \( \mathcal{I} \) defines an SNC divisor.
    \end{itemize}
    \item The process is functorial under smooth morphisms, field extensions, and group actions preserving \( (\mathcal{I}, E) \).
\end{enumerate}
\end{theorem}

\subsubsection{Embedded desingularization}

\begin{theorem}\label{embedded}
Let \( Y \subset X \) be a closed subscheme of a smooth scheme over \( k \) of characteristic zero, with \( E \) an SNC divisor on \( X \). Then there exists a functorial sequence of cobordant blow-ups
\[
X = X_0 \leftarrow \cdots \leftarrow X_k = X', \quad Y = Y_0 \leftarrow \cdots \leftarrow Y_k = Y'
\]
such that:
\begin{enumerate}
    \item Each \( X_i \) admits a torus action \( T_i \), and the closed subschemes \( Y_i \subset X_i \) are \( T_i \)-invariant.
    \item The divisors \( E_i \) remain SNC under  total transforms.
    \item Centers lie in the singular or non-transversal locus of \( Y_i \) with respect to \( E_i \).
    \item The final transform \( Y' \) is smooth and transverse to \( E_k \).
    \item The sequence descends to geometric and stack-theoretic quotients where \( Y'/T \) has abelian quotient singularities and \( [Y'/T] \subset [X'/T] \) is smooth.
    \item Functoriality holds for smooth morphisms, field extensions, and group actions preserving \( Y \).
\end{enumerate}
\end{theorem}


\subsubsection{Nonembedded desingularization}

\begin{theorem}\label{nonembedded}
Let \( Y \) be a reduced scheme of finite type over \( k \) of characteristic zero. Then there exists a functorial sequence of cobordant blow-ups
\[
Y = Y_0 \leftarrow \cdots \leftarrow Y_k = Y'
\]
such that:
\begin{enumerate}
    \item \( Y' \) is smooth.
    \item Each \( Y_i \) carries a torus action \( T_i \), and admitts quotient \( Y_i/T_i \).
    \item The exceptional locus in \( Y' \) is a \( T \)-invariant SNC divisor.
    \item Over the nonsingular locus, \( Y' \to Y \) is isomorphic to \( Y^{ns} \times T \).
    \item The resolution descends to the weighted blow-ups of geometric quotients \( Y/T_0 \leftarrow \cdots \leftarrow Y'/T_k \) with abelian quotient singularities of  $ Y'/T_k$.
    \item The stack quotients \( [Y/T_0] \leftarrow \cdots \leftarrow [Y'/T_k] \) yield a smooth stack $ [Y'/T_k].$
    \item Functoriality holds for smooth morphisms, field extensions, and group actions.
\end{enumerate}
\end{theorem}
\section{Geometry of Cobordant Blow-ups}\label{Sec:weighted-blowup}

\subsection{Rational powers and \texorpdfstring{$\QQ$}{Q}-ideals}

In \cite{ATW-weighted}, we introduced the notion of \emph{valuative $\QQ$-ideals} $\cJ$, or simply $\QQ$-ideals. These generalize classical ideals and are closely related to \emph{rational powers of ideals} as studied by Huneke-Swanson \cite[Section 10.5]{HS06}. While $\QQ$-ideals are a compact way to encode centers, in this paper we primarily work with their associated rational Rees algebras.

\begin{definition}[\cite{ATW-weighted}]
Let $X$ be an irreducible noetherian scheme. A \emph{$\QQ$-ideal} is an equivalence class of formal expressions $\cJ^{1/n}$, where $\cJ$ is an ideal and $n \in \NN$. Two such expressions $\cJ^{1/n}$ and $\cI^{1/m}$ are equivalent if $(\cJ^m)^\inte = (\cI^n)^\inte$.
\end{definition}

\begin{remark}
If $\cJ^{1/n} \sim \cI^{1/m}$, then $\cI$ and $\cJ$ are said to be \emph{projectively equivalent} with ratio $m/n$; see \cite{Rush07}.
\end{remark}

Each $\QQ$-ideal $\cJ = \cI^{1/n}$ defines a unique integrally closed Rees algebra:
\[
\cA_\cJ := (\cO_X[\cI t^n])^\inte \subset \cO_X[t],
\]
where the superscript \( \texttt{int} \) denotes the {\it integral closure} of the graded algebra in \( \mathcal{O}_X[t] \). 
Conversely, any integrally closed $\ZZ$-graded Rees algebra arises in this way. This correspondence was emphasized by Quek:

\begin{proposition}[\cite{Quek}, Theorem 2.2.5]
There is a bijective correspondence between $\QQ$-ideals $\cJ = \cI^{1/n}$ and integrally closed $\ZZ$-graded Rees algebras, given by
\[
\cJ \mapsto \cA_\cJ = (\cO_X[\cI t^n])^\inte.
\]
\end{proposition}

Associated with a $\QQ$-ideal $\cJ = \cI^{1/n}$ is its \emph{ideal of sections}:
\[
\cJ_X := \{f \in \cO_X \mid f^n \in \cI^\inte\},
\]
which equals the $t$-gradation $(\cA_\cJ)_1$. For ordinary ideals, this recovers the integral closure.

\begin{definition}[\cite{HS06}]
For an ideal $\cI$ and $m/n \in \QQ_{>0}$, the ideal of sections $(\cI^{m/n})_X$ is called the \emph{rational power} of $\cI$.
\end{definition}

Given $\QQ$-ideals $\cJ_i = \cI_i^{a_i/n_i}$, their sum corresponds to the Rees algebra
\[
\cO_X[\cI_1^{a_1} t^{n_1}, \ldots, \cI_k^{a_k} t^{n_k}]^\inte,
\]
which defines another $\QQ$-ideal. One has natural operations:
\[
\cI^{1/n} + \cJ^{1/n} = (\cI + \cJ)^{1/n}, \quad
\cI^{1/n} \cdot \cJ^{1/m} = (\cI^m \cJ^n)^{1/mn}.
\]
This extends the usual sum and product of ideals.

We write $\cJ_1 \subseteq \cJ_2$ if, for sufficiently divisible $N$, one has
\[
(\cI_1^{N/n_1})^\inte \subseteq (\cI_2^{N/n_2})^\inte.
\]

\subsubsection{Graded algebras of \texorpdfstring{$\QQ$}{Q}-ideals}

For any $\QQ$-ideal $\cJ = \cI^{1/n}$, define its graded $\QQ$-ideal algebra as
\[
\cO_X[\cJ t] := \bigoplus_{i \geq 0} \cJ^i t^i.
\]
The associated Rees algebra of ideals on $X$ is given by
\[
\cA_\cJ := (\cO_X[\cJ t])_X := \bigoplus_{i \geq 0} (\cJ^i)_X t^i.
\]
\subsubsection{Functoriality of $\QQ$-ideals}
Let $f\colon X'\to X$ be a morphism of integral schemes, and let $\cJ = \cI^{1/n}$ be a $\QQ$-ideal on $X$. Then $f^*\cJ := (\cO_{X'}\cdot \cI)^{1/n}$ defines a natural pullback $\QQ$-ideal on $X'$.

\subsubsection{Monomial valuations}
\begin{definition}
Let $X$ be a regular irreducible scheme and $u_1,\ldots,u_k$ be part of a local parameter system at $p\in X$. A valuation $\nu$ is \emph{monomial at $p$} with respect to weights $w_i\in \ZZ_{\geq 0}$ if
\[
\cI_{\nu,a,p} := \{f \in \cO_{X,p} \mid \nu(f) \geq a\} = (u^\alpha \mid \sum a_i w_i \geq a).
\]
We say $\nu$ is a \emph{monomial valuation on $X$} if it is monomial at all $p \in V(u_1,\ldots,u_k)$.
\end{definition}

\begin{lemma}
Let $X$ be regular and $V(u_1,\ldots,u_k)$ irreducible. Assigning weights $w_i$ to $u_i$ defines:
\begin{enumerate}
\item A monomial valuation $\nu$ on $\Spec(\cO_{X,p})$ for any $p \in V(u_1,\ldots,u_k)$.
\item A unique monomial valuation $\nu$ on $X$, centered at $V(u_1,\ldots,u_k)$.
\end{enumerate}\qed
\end{lemma}

\subsubsection{Regular weighted centers}
A \emph{regular weighted center} on a regular scheme $X$ is a $\QQ$-ideal locally of the form $(u_1^{a_1},\ldots,u_k^{a_k})$, where $a_i \in \QQ_{>0}$ and $u_i$ are part of a system of parameters. One can write
\[
(u_1^{a_1},\ldots,u_k^{a_k}) = \left(u_1^{na_1},\ldots,u_k^{na_k}\right)^{1/n}
\]
for sufficiently divisible $n$. The ideal of sections is denoted $(u_1^{a_1},\ldots,u_k^{a_k})_X$.

\begin{lemma}[{\cite{ATW-weighted}, \cite{Quek}}\cite{W22}]\label{comp}
Let $X$ be regular, $V(u_1,\ldots,u_k)$ irreducible, and $w_1,\ldots,w_k$ positive integers. Then the following are equivalent:
\begin{enumerate}
\item The $\QQ$-ideal $\cJ = (u_1^{1/w_1},\ldots,u_k^{1/w_k})$.
\item The Rees algebra $\cA_\cJ = \cO_X[u_1 t^{w_1}, \ldots, u_k t^{w_k}]^\inte=\\
\cO_X[u_1t^{c_1},\ldots, u_kt^{c_k} \mid 0< c_1\leq w_1]$.
\item A monomial valuation $\nu$ with $\nu(u_i) = w_i$, giving ideals $\cI_{\nu,a} = (u^\alpha \mid \nu(u^\alpha)\geq a)$.
\end{enumerate}
Moreover, 
\[
\cA_\cJ = \bigoplus_{a \geq 0} \cI_{\nu,a} t^a = (\cO_X[\cJ t])_X.
\]
\end{lemma}

\begin{lemma} \label{comp2-short}\cite{W22}
With $\cJ = (u_1^{1/w_1}, \ldots, u_k^{1/w_k})$, and monomial valuation $\nu$ as above, for any $a \in \QQ_{>0}$,
\[
(\cJ^a)_X = \cI_{\nu,a} = \{f \in \cO_X \mid \nu(f) \geq a\} = (u^\alpha \mid \sum \alpha_i w_i \geq a).
\]
In particular,
\[
(u_1^{a_1},\ldots,u_k^{a_k})_X = (u^\alpha \mid \sum \alpha_i / a_i \geq 1)\qed.
\]
\end{lemma}

\subsubsection{Blow-ups of $\QQ$-ideals}
A $\QQ$-ideal $\cJ = \cI^{1/n}$ on a normal scheme $X$ defines the (normalized) blow-up:
\[
Y := \cProj(\cA_\cJ) \to X,
\]
which transforms $\cJ$ into a Cartier divisor ideal $(\cO_X(-E))^{1/n}$ for some exceptional divisor $E$.

\subsubsection{Weighted and stack-theoretic blow-ups}
For $\cJ = (x_1^{1/w_1},\ldots,x_k^{1/w_k})$, the weighted blow-up is:
\[
Y = \cProj(\cO_X[x_1 t^{w_1}, \ldots, x_k t^{w_k}]^\inte).
\]
The stack-theoretic weighted blow-up (\cite[Section 3.1]{ATW-weighted} ) is the quotient stack:
\[
\left[\left(\Spec_X(\cO_X[x_1 t^{w_1}, \ldots, x_k t^{w_k}]^\inte) \setminus V(t^{w_1} x_1, \ldots, t^{w_k} x_k)\right)/G_m\right].
\]
This construction allows smooth handling of singularities even when the blow-up space is not regular.

\subsection{Rational Rees algebras}

Let $X$ be a scheme. A \emph{rational Rees algebra} (or simply \emph{Rees algebra}) is a finitely generated $\cO_X$-algebra
\[
R = \bigoplus_{a \in \Gamma} R_a t^a \subset \cO_X[t^{1/w_R}],
\]
where $\Gamma$ is a finitely generated additive subsemigroup of $\QQ_{\geq 0}$, $R_0 = \cO_X$, and $R_a \cdot R_b \subseteq R_{a+b}$. The minimal $w_R \in \QQ_{>0}$ such that $\Gamma \subseteq (1/w_R)\cdot \ZZ_{\geq 0}$ is called the \emph{grading denominator}.

The \emph{extended Rees algebra} is $R^{\ext} := R[t^{-1/w}]$ for any multiple $w$ of $w_R$.

The \emph{integral closure} $R^\inte$ is the integral closure of $R$ in $\cO_X[t^{1/w_R}]$, and $R^\Int$ denotes the integral closure in $\cO_X[t^{1/w}]$ for a certain $w$.

The \emph{vertex} of $R$  (or $R^{\ext}$) is the closed set:
\[
V(R) = V(R^{\ext}):=V\left( \sum_{a > 0} R_a \right)  .
\]

\begin{remark}
The filtration need not satisfy $R_a \subseteq R_b$ for $a \geq b$, unless $R$ is integrally closed.
\end{remark}

\subsubsection{Examples}
The Rees algebra of an ideal $\cI$ is the standard $\ZZ$-graded algebra
\[
\cA_\cI := \cO_X[\cI t] = \bigoplus_{n \geq 0} \cI^n t^n,
\]
with extended version $\cA_\cI^{\ext} = \cO_X[t^{-1}, \cI t]$.


\subsubsection{Rees centers}
A \emph{Rees center} is a Rees algebra locally of the form
\[
\cA = \cO_X[x_1 t^{1/a_1}, \ldots, x_k t^{1/a_k}]^\inte,
\]
where $x_1,\ldots,x_k$ are part of a local parameter system, and $a_i \in \QQ_{>0}$. The integral closure is taken in \( \mathcal{O}_X[t^{1/w_A}] \), where \( w_A := \mathrm{lcm}(a_1, \ldots, a_k) \) denotes the least common multiple of the positive rational numbers \( a_1, \ldots, a_k \).
The \emph{extended center} is
\[
\cA^\ext := \cO_X[t^{-1/w}, x_1 t^{1/a_1}, \ldots, x_k t^{1/a_k}],
\]
with $w$ a multiple of $w_A$ so that all $w/a_i$ are integral.

\subsubsection{Rescaling}
Given $w_0 \in \QQ_{>0}$, the \emph{rescaling} of $R$ is defined by
\[
R^{w_0} := \bigoplus_{a \in \Gamma} R_a t^{w_0 a} \subset \cO_X[t^{w_0 / w_R}].
\]

\begin{lemma}
The map $ft^a \mapsto ft^{w_0 a}$ defines an isomorphism $R \simeq R^{w_0}$. In particular, $R$ is integrally closed iff $R^{w_0}$ is.
\end{lemma}

\subsubsection{Monomial valuations}
If $\cA = \cO_X[x_1 t^{1/a_1}, \ldots, x_k t^{1/a_k}]^\inte$, and $V(\cA)$ is irreducible, there exists a unique monomial valuation $\nu_A$ with $\nu_A(x_i) = 1/a_i$, such that
\[
\cA_a = \{ f \in \cO_X \mid \nu_A(f) \geq a \}.
\]
From Lemma \ref{comp} we get
\begin{lemma}(\cite{W22})\label{comp3}
Let $\cA = \cO_X[ x_1 t^{1/a_1}, \ldots, x_k t^{1/a_k}]^\inte$ be a Rees center and $\cA^\ext = \cO_X[t^{-1/w}, x_1 t^{1/a_1}, \ldots, x_k t^{1/a_k}]$ be its extension. Then the integral closure $\cA^\Int$ of $\cA$ in $\cO_X[t^{1/w}]$ equals the nonnegative part of $\cA^\ext$, i.e., $\cA^\Int = (\cA^\ext)_{\geq 0}$.
\end{lemma}

\subsubsection{Regular centers vs. Rees centers}
A regular $\QQ$-ideal center $\cJ = (u_1^{1/w_1}, \ldots, u_k^{1/w_k})$ with $w_i \in \NN$ corresponds to:
\begin{itemize}
  \item the Rees integral algebra $\cA_\cJ = \cO_X[u_1 t^{w_1}, \ldots, u_k t^{w_k}]^\inte$,
  \item and the extended Rees integral algebra $\cA_\cJ^\ext = \cO_X[t^{-1}, u_1 t^{w_1}, \ldots, u_k t^{w_k}]$.
\end{itemize}

More generally, there is a natural correspondence between:
\begin{itemize}
\item Rees centers $\cA=\cO_X[x_1 t^{1/a_1}, \ldots, x_k t^{1/a_k}]^\inte$, with $a_i\in \QQ_{>0}$.
\item Extended Rees centers  $\cA^{\ext}=\cO_X[t^{-1/w_A},x_1 t^{1/a_1}, \ldots, x_k t^{1/a_k}]^\inte$, where $w_A=\lcm(a_1,\ldots,a_k)$
\item Their associated integral Rees algebras $\cA_\cJ=\cA^w=\cO_X[u_1 t^{w_1}, \ldots, u_k t^{w_k}]^\inte$,
\item  Their associated extended integral Rees algebras $$\cA_\cJ^\ext=(\cA^{\ext})^w= \cO_X[t^{-1}, u_1 t^{w_1}, \ldots, u_k t^{w_k}]$$
\item Their associated $\QQ$-ideal centers $\cJ = (x_1^{1/w_1}, \ldots, x_k^{1/w_k})$ with $w_i \in \NN$,
\end{itemize}

\subsection{Cobordant Blow-ups}

\subsubsection{Good and geometric quotients}

Let $T = \Spec(\ZZ[t_1^{\pm 1}, \ldots, t_k^{\pm 1}])$ act relatively affinely on a scheme $X$ over $\ZZ$. A \emph{good (GIT) quotient} is an affine $T$-invariant morphism $\pi: X \to Y = X \sslash T$ such that $\cO_Y \simeq \pi_*(\cO_X)^T$. The quotient is \emph{geometric} if all geometric fibers of $\pi$ are single $T$-orbits.

\subsubsection{Birational cobordisms}
\cite[Definition 2]{W-Cobordism}

The concept of {\it birational cobordism} was originally introduced over a field. For motivational clarity, we briefly recall the definition in that setting. Given a $G_m$-action on an integral scheme $B$ over a field, define:
\[
B_- := \{ p \in B \mid \lim_{t \to 0} tp \text{ does not exist} \}, \quad
B_+ := \{ p \in B \mid \lim_{t \to \infty} tp \text{ does not exist} \}.
\]
We say $B$ is a \emph{birational cobordism} between $X_1$ and $X_2$ if:
\begin{itemize}
  \item $B_\pm$ are Zariski open and nonempty,
  \item geometric quotients $B_\pm / T \simeq X_{1,2}$ exist,
  \item the birational map $\phi: X_1 \dashrightarrow X_2$ factors via $B_\pm / T$.
\end{itemize}

The notion extends naturally over more general bases such as $\Spec(\ZZ)$, though the definition of limits in that context may differ slightly and will not be used in this paper.
\subsubsection{Example: Weighted blow-up via cobordism}

Let \( T \) act on $$ \mathbb{A}^{n+1} = \Spec(k[x_0, \ldots, x_n]) $$via
\[
t \cdot (x_0, x_1, \ldots, x_n) = (t^{-1} x_0,\, t^{w_1} x_1,\, \ldots,\, t^{w_k} x_k).
\]
Then the open charts are
\[
B_- = \mathbb{A}^{n+1} \setminus V(x_0), \quad B_+ = \mathbb{A}^{n+1} \setminus V(x_1, \ldots, x_n).
\]
Using toric geometry, the quotient map
\[
B_+/T \longrightarrow B_-/T = B // T = \Spec(k[x_0, \ldots, x_n]^T) = \Spec(K[u_1, \ldots, u_k]),
\]
with \( u_i = x_i / x_0^{w_i} \) and \( t^{-1} = x_0 \), is the weighted blow-up at the \(\mathbb{Q}\)-ideal
\[
\mathcal{J} = (u_1^{1/w_1}, \ldots, u_k^{1/w_k}).
\]
The total space is given by
\[
B = \Spec(k[x_0, \ldots, x_n]) = \Spec(k[t^{-1}, t^{w_1} u_1, \ldots, t^{w_k} u_k]),
\]
with
\[
B_+ = B \setminus V(t^{w_1} u_1, \ldots, t^{w_k} u_k), \quad B_- = B \setminus V(t^{-1}).
\]


\subsubsection{Cobordant blow-up: definition}

\begin{definition}
Let \( X \) be a regular scheme, and let
\[
\mathcal{A}^{\mathrm{ext}} = \mathcal{O}_X[t^{-1/w}, x_1 t^{1/a_1}, \ldots, x_k t^{1/a_k}]
\]
be an extended Rees algebra associated with a \(\mathbb{Q}\)-ideal center \( \mathcal{J} = (x_1^{1/w_1}, \ldots, x_k^{1/w_k}) \), where \( w = \mathrm{lcm}(a_1, \ldots, a_k) \) and \( w_i = w / a_i \). Consider the rescaled algebra:
\[
(\mathcal{A}^{\mathrm{ext}})^w = \mathcal{A}^{\mathrm{ext}}_{\mathcal{J}} = \Spec_X\left( \mathcal{O}_X[t^{-1}, t^{w_1} x_1, \ldots, t^{w_k} x_k] \right).
\]
We define the \emph{full cobordant blow-up} of \( \mathcal{A} \) (or equivalently, of the center \( \mathcal{J} \)) as
\[
B := \Spec_X\left( (\mathcal{A}^{\mathrm{ext}})^w \right) = \Spec_X\left( \mathcal{A}^{\mathrm{ext}}_{\mathcal{J}} \right) = \Spec_X\left( \mathcal{O}_X[t^{-1}, t^{w_1} x_1, \ldots, t^{w_k} x_k] \right).
\]

We distinguish the following components of the cobordant blow-up:

\begin{itemize}
    \item The \emph{trivial cobordant blow-up} is the projection
    \[
    \sigma_-: B_- := B \setminus V(t^{-1}) = \Spec_X\left( \mathcal{O}_X[t^{\pm 1}] \right) \longrightarrow X,
    \]
    which corresponds to the product \( X \times \mathbb{G}_m \).

    \item The \emph{cobordant blow-up} is the \( T \)-equivariant morphism
    \[
    \sigma_+: B_+ := B \setminus V(t^{w_1} x_1, \ldots, t^{w_k} x_k) \longrightarrow X,
    \]
    where \( B_+ \) is the complement of the vertex and corresponds to the weighted blow-up determined by \( \mathcal{J} \).

    \item The \emph{vertex} of the cobordant blow-up is the closed subscheme
    \[
    \operatorname{Ver}(B) := V(t^{w_1} x_1, \ldots, t^{w_k} x_k) = V\left( (\mathcal{A}^{\mathrm{ext}})^w \right) = B \setminus B_+,
    \]
    which represents the geometric counterpart of the center $\cJ$ on $X$ inside \( B \).     
\end{itemize}
\end{definition}

In this setup, \( B_+ \) and \( B_- \) define a birational cobordism over \( X \), interpolating between the trivial product \( X \times \mathbb{G}_m \) and the weighted blow-up determined by \( \mathcal{J} \). See Figure~\ref{F1}.
\begin{remark}
The algebras ${\cO_X}[t^{-1}, t^{w_1}x_1, \ldots, t^{w_k}x_k]$ used in the construction of cobordant blow-ups also appeared independently in work by Quek and Rydh, developed in the context of stack-theoretic blow-ups. Their approach, which emphasizes a stack-theoretic interpretation, was made publicly available on their homepage shortly after the first version \cite{W22} of the present paper appeared on the arXiv; see \cite{Rydh-proj}. In contrast, cobordant blow-ups provide a torus-action-based viewpoint that is well-suited for applications in positive and mixed characteristic, as well as in the resolution of foliations.

Both approaches can be traced back to ideas introduced in \cite{ATW-weighted} and are naturally connected to extended Rees algebras studied by Rees, Swanson, and Huneke; see \cite[Def.\ 5.1]{HS06}.

Over a field $k$ of characteristic zero, the stack-theoretic quotient of a cobordant blow-up $[B_+/\GG_m] \to X$ defines a stack-theoretic weighted blow-up, in the sense of \cite[Section~3.1]{ATW-weighted}. This interpretation provides a direct bridge between our torus-equivariant construction and the stack-theoretic framework used in earlier approaches.
\end{remark}

\subsubsection{Exceptional divisor}

\begin{lemma} \label{divisor}
The cobordant blow-up transforms the $\QQ$-ideal center $\cJ$ into the ideal of the exceptional divisor $D = V(t^{-1})$ on $B_+$:
\[
\cJ \cdot \cO_{B_+} = t^{-1} \cdot \cO_{B_+}.
\]
In particular, the inverse image of the center  $\cJ$ in \( B \) is given by
\[
V(\mathcal{J} \cdot \mathcal{O}_B) = \Ver(B) \cup D,
\]
where \( \Ver(B) \) is the vertex and \( D = V_B(t^{-1}) \) is the exceptional divisor.
\end{lemma}
\begin{proof}
We write $\cJ \cdot \cO_{B_+} = (\cJ \cdot t) \cdot t^{-1} \cdot \cO_{B_+}$. Since $(\cJ \cdot t) = ((x_1 t^{w_1})^{1/w_1}, \ldots, (x_k t^{w_k})^{1/w_k})$ is a trivial $\QQ$-ideal on $B_+$, it equals $\cO_{B_+}$. Hence, $\cJ \cdot \cO_{B_+} = t^{-1} \cdot \cO_{B_+}$.
\end{proof}

\begin{figure}[ht]
\centering
\begin{tikzpicture}[scale=1.2, every node/.style={font=\small}, >=Latex]

\begin{scope}[shift={(0,0)}]

\foreach \x in {-1.3, 1.3} {
  \draw[thick] (\x,-0.5) -- (\x,0.65);
  \draw[thick, ->] (\x,0.65) -- (\x,1.15); 
  \draw[thick] (\x,1.15) -- (\x,1.8);
}

\draw[blue, thick, ->] (0,-0.2) -- (0,0.78) node[below left, blue] {};
\node[blue,thick] at (0,0.4) {$V$ vertex};

\foreach \angle in {140, 120, 90, 60, 40}
  \draw[green!50!black, thick, ->] (0,0.8) -- ++(\angle:1.2);

\node[green!50!black] at (1.1,2) {$D$ exc. divisor};

\draw[very thick] (-1.5,-0.8) -- (1.5,-0.8);
\fill[red] (0,-0.8) circle (1.5pt);
\node[below, red] at (0,-0.8) {$V(\mathcal{J})$ center};

\node at (0,-2.0) {$B$ full cobord. blow-up};

\end{scope}

\begin{scope}[shift={(3.5,0)}]

\foreach \x in {-1.3, 1.3} {
  \draw[thick] (\x,-0.5) -- (\x,0.65);
  \draw[thick, ->] (\x,0.65) -- (\x,1.15);
  \draw[thick] (\x,1.15) -- (\x,1.8);
}

\foreach \angle in {140, 120, 90, 60, 40}
  \draw[green!50!black, thick, ->] (0,0.8) -- ++(\angle:1.2);

\node[green!50!black] at (1.1,2) {$D$ exc. divisor};

\node[green!50!black,very thick,draw, circle, inner sep=1.5pt] at (0,0.8) {};
\draw[very thick] (-1.5,-0.8) -- (1.5,-0.8);
\fill[red] (0,-0.8) circle (1.5pt);
\node[below, red] at (0,-0.8) {$V(\mathcal{J})$ center};

\node at (0,-2.0) {$B_+ = B \setminus V$ cob. blow-up};

\end{scope}

\begin{scope}[shift={(7,0)}]

\foreach \x in {-1.3, -0.9, -0.5, 0.5, 0.9, 1.3} {
  \draw[thick] (\x,-0.5) -- (\x,0.65);
  \draw[thick, ->] (\x,0.65) -- (\x,1.15);
  \draw[thick] (\x,1.15) -- (\x,1.8);
}

\draw[blue, thick, ->] (0,-0.2) -- (0,0.8) node[above right, blue] {$V$};

\draw[very thick] (-1.5,-0.8) -- (1.5,-0.8);
\fill[red] (0,-0.8) circle (1.5pt);
\node[below, red] at (0,-0.8) {$V(\mathcal{J})$ center};

\node at (0,-2.0) {$B \setminus D = X \times \mathbb{G}_m$};

\end{scope}

\end{tikzpicture}
\caption{Cobordant blow-up: the role of the vertex $V$ and the exceptional divisor $D$.}\label{F1}
\label{fig:cobordant-blowup}
\end{figure}

\subsubsection{Local description}

Let $x_1, \ldots, x_n$ be local parameters on $X$, and let $\cJ = (x_1^{1/w_1}, \ldots, x_k^{1/w_k})$. Then locally:
\[
B = \Spec\left( \cO_X[t^{-1}, x_1', \ldots, x_n'] / (x_i' t^{-w_i} - x_i) \right),
\]
with $x_i' = x_i \cdot t^{w_i}$ and $x_j' = x_j$ for $j > k$. The torus acts via $t$, and $B$ is a regular closed subscheme of $X \times \AA^{n+1}$.

\begin{remark}
Sometimes it is convenient to substitute $s = t^{-1}$ so that $x_i = x_i' s^{w_i}$.
\end{remark}

\subsubsection{Toric cobordant blow-ups}

In the toric setting, $X_\sigma = \Spec(k[x_1,\ldots,x_n])$ corresponds to a cone $\sigma \subset \RR^n$. The full cobordant blow-up of $(x_1^{1/w_1}, \ldots, x_n^{1/w_n})$ corresponds to the cone:
\[
\tau = \langle e_1, \ldots, e_n, v + e_{n+1} \rangle,
\]
where $v = w_1 e_1 + \ldots + w_n e_n$. The map $X_\tau \to X_\sigma$ projects along $e_{n+1}$ and represents a toric birational cobordism. The upper boundary of $\tau$ corresponds to the star subdivision of $\sigma$ at $v$, and $B_+ / T$ is the weighted blow-up of $X_\sigma$ at $\cJ$ (See Figure \ref{F2}).

\begin{figure}[ht]
\centering
\begin{tikzpicture}[scale=1, every node/.style={font=\small}]

\draw[thick] (0,0) -- (1.3,2) -- (3,0) -- cycle;
\node at (1.3,2.3) {\(\Delta\)};
\node at (1.3,3.5) {\textcolor{ForestGreen}{Full cobordant blow-up}};

\draw[->, thick] (1.2,-0.2) -- (1.2,-1.4);
\draw[->, thick] (1.8,-0.2) -- (1.8,-1.4);
\node at (2.5,-0.8) {$\pi$};
\draw[thick, magenta] (0,-2) -- (3,-2);
\node[magenta] at (1.3,-2.3) {\(\sigma\)};
\node[magenta] at (1.3,-2.8) {flat section of regular cone $\sigma$};
\node[magenta] at (-0.2,-2) {\(e_1\)};
\draw[thick, magenta] (0,-2) circle (1pt); 
\node[magenta] at (3.3,-2) {\(e_2\)};
\draw[thick, magenta] (3,-2) circle (1pt); 

\begin{scope}[shift={(4,0)}]
\draw[thick] (0,0) -- (1.3,2) -- (3,0);
\node at (1.3,2.3) {\(\Delta_+\)};

\node at (1.3,3.5) {\textcolor{ForestGreen}{Cobordant blow-up}};
\node at (1.3,3) {\textcolor{ForestGreen}{Upper boundary}};
\draw[->, thick] (1.2,-0.2) -- (1.2,-1.4);
\draw[->, thick] (1.8,-0.2) -- (1.8,-1.4);
\node at (2.5,-0.8) {$\pi$};
\draw[thick, magenta] (0,-2) -- (3,-2);
\draw[thick, magenta] (1.3,-2) circle (1pt); 
\node[magenta] at (1.3,-2.3) {\(v\)};
\node[magenta] at (-0.2,-2) {\(e_1\)};
\draw[thick, magenta] (0,-2) circle (1pt); 
\node[magenta] at (3.3,-2) {\(e_2\)};
\draw[thick, magenta] (3,-2) circle (1pt); 
\node[magenta] at (0.7,-1.7) {\(\pi(\Delta_+)\)};
\node[magenta] at (1.3,-2.8) {star subdivision of \(\sigma\) at $v$};

\end{scope}

\begin{scope}[shift={(8,0)}]
\draw[thick] (0,0) -- (3,0);
\node at (2.5,0.3) {\(\Delta_-\)};
\node at (1.3,3) {\textcolor{ForestGreen}{Lower boundary}};
\node at (1.3,3.5) {\textcolor{ForestGreen}{Trivial cob. blow-up}};

\draw[->, thick] (1.2,-0.2) -- (1.2,-1.4);
\draw[->, thick] (1.8,-0.2) -- (1.8,-1.4);
\node at (2.5,-0.8) {$\pi$};
\draw[thick, magenta] (0,-2) -- (3,-2);
\node[magenta] at (-0.2,-2) {\(e_1\)};
\draw[thick, magenta] (0,-2) circle (1pt); 
\node[magenta] at (3.3,-2) {\(e_2\)};
\draw[thick, magenta] (3,-2) circle (1pt); 
\node[magenta] at (1.3,-2.3) {\(\sigma\)};
\end{scope}


\end{tikzpicture}
\caption{Toric cobordant blow-up as the lifting of the star subdivision} \label{F2}
\end{figure}

\subsection{Resolution Invariant on Smooth Schemes with SNC Divisors}

\subsubsection{Compatibility with SNC Divisors}
\begin{definition}\label{compa}
A \emph{coordinate system} (or \emph{system of local parameters}) \( x_1, \ldots, x_n \) on an open subset \( U \) of a regular scheme \( X \) is said to be \emph{adapted to a simple normal crossings (SNC) divisor} \( E \) if every irreducible component \( E_i \) of \( E \) that intersects \( U \) is locally defined by \( V(x_i) \). 

At a point \( p \in U \), a coordinate \( x_i \) is called \emph{divisorial} if \( p \in V(x_i) = E_i \) for some component \( E_i \subseteq E \), and \emph{free} otherwise. For simplicity, we often identify \( E \) and its components \( E_i \) with the corresponding coordinates \( x_i \). A Rees center \( \cA = \cO_X[x_1t^{1/a_1}, \ldots, x_kt^{1/a_k}]^\inte \) or its extension is \emph{ adapted  to $E$} if the coordinates $x_1, \ldots, x_k$ are  adapted  to $E$.
\end{definition}

\begin{remark}
This matches the logarithmic viewpoint:  divisorial  coordinates define a chart $U \to \Spec(\ZZ[P])$ where $P$ is generated by the divisorial parameters. The free coordinates define coordinates on the stratum $V(P \setminus 0)$.
\end{remark}

\subsubsection{Centers adapted to an SNC Divisor}
Let \( \QQ_+ := \QQ \amalg \{a_+ \mid a \in \QQ\} \) with the order $a_+ > a$, and $a_+ < b$ if $a < b$. Define addition, scalar multiplication, and subtraction on \( \QQ_+ \) from \( \QQ \) in the natural way.

Given a center \( \cA = \cO_X[x_1t^{1/a_1}, \ldots, x_kt^{1/a_k}]^\inte \)  adapted  to $E$, associate to each coordinate $x_i$ the symbol $b_i := a_{i+}$ if divisorial and $b_i := a_i$ otherwise. Assume $b_1 \leq \ldots \leq b_k$, and define:
\[ \inv(\cA) := (b_1, \ldots, b_k). \]

\begin{remark}
Unlike the logarithmic order from \cite[Section 3.6]{ATW-principalization}, which gives divisorials weight $\infty$, we assign $1_+$---the minimal strictly greater value than 1. This is compatible with standard derivations and yields SNC resolutions.
\end{remark}

\subsubsection{Canonical Invariant}
\begin{definition}\label{inva3}
Let $X$ be a regular scheme with SNC divisor $E$. For an ideal $\cI$, define the \emph{canonical invariant} at $p$ as:
\[ \inv_p(\cI) := \max \{ (b_1, \ldots, b_k) \mid \cI t \subseteq \cO_X[x_1t^{1/a_1}, \ldots, x_kt^{1/a_k}]^\inte \}, \]
where the $x_i$ are  adapted  to $E$ at $p$. The corresponding center $\cA$ is \emph{maximal admissible} at $p$.
\end{definition}

\begin{remark}
We prove in Section \ref{algo} that such a maximum is attained by a unique maximal admissible center.

In the case where no boundary divisor is present, this definition agrees with the invariant defined in \cite{ATW-weighted}.

\end{remark}

\subsubsection{Presentation of Centers}
Centers can be presented more compactly as:
\[ \cA = \cO_X[\overline{x}_1t^{1/a_1}, \ldots, \overline{x}_kt^{1/a_k}]^\inte, \]
with blocks \( \overline{x}_i = (x_{i1}, \ldots, x_{ik_i}) \), where $a_1<\ldots<a_k$  and associated weights \( \overline{b}_i =(\underbrace{a_i,\ldots,a_i}_{\text{free}}, \underbrace{a_{i+},\ldots,a_{i+}}_{\text{divisorial}})
 \).
\[ \inv_p(\cI) = \max \{ (\overline{b}_1, \ldots, \overline{b}_k) \mid \cI t \subseteq \cA \}. \]

\subsubsection{Order of an Ideal}
\begin{lemma}\label{order3}
\[ \ord_p(\cI) = \max \{ a_1 \in \QQ_{>0} \mid \cI t \subseteq \cO_X[m_pt^{1/a_1}] \}. \]
\end{lemma}
\begin{proof}
By definition, $\ord_p(\cI) = \max \{ a_1 \mid \cI_p \subseteq m_p^{\lceil a_1 \rceil} \}$. Lemma \ref{comp3} ensures $\cI t \subseteq \cO_X[m_p t^{1/a_1}]$ implies $\cI_p \subseteq m_p^{\lceil a_1 \rceil}$.
\end{proof}

\subsubsection{Admissibility for Ideals}
\begin{lemma}\label{admi}
For an ideal $\cI$, the following are equivalent near $p$:
\begin{enumerate}
  \item $\cI t \subseteq \cA = \cO_X[x_1t^{1/a_1}, \ldots, x_kt^{1/a_k}]^\inte$,
  \item $\cI t \subseteq \cA^\ext := \cO_X[t^{-1/w_A}, x_1t^{1/a_1}, \ldots, x_kt^{1/a_k}]$,
  \item $\cI t^{w_A} \subseteq \cA^{w_A} = \cO_X[x_1t^{w_1}, \ldots, x_kt^{w_k}]^\inte$,
  \item $\cI t^{w_A} \subseteq (\cA^\ext)^{w_A} = \cO_X[t^{-1}, x_1t^{w_1}, \ldots, x_kt^{w_k}]$.
\end{enumerate}
These equivalences also hold if $\cI t$ is replaced by a Rees algebra $R$.
\end{lemma}
\begin{proof} 
 
The condition (3) and (4) are obtained by  rescaling the conditions in (1) and (2). On the other hand, by Lemma \ref{comp}, $$(\cA^{\ext})_{\geq 0}=\cO_X[t^{-1}, x_1t^{w_1},\ldots,x_kt^{w_k}]_{\geq 0}=\cO_X[x_1t^{w_1},\ldots,x_kt^{w_k}]^\inte=\cA.$$
This implies that condition (1) and   (2) are equivalent.
\end{proof}

\begin{definition}
Such a center  $\cA$ is called \emph{admissible} for $\cI$ at $p$.
\end{definition}

\subsubsection{Admissibility for Rees Algebras}
\begin{definition}
A Rees center $\cA = \cO_X[x_1t^{1/a_1}, \ldots, x_kt^{1/a_k}]^\Int$ is \emph{admissible} for a Rees algebra $R$ at $p\in X$ if any of the equivalent conditions in Lemma~\ref{admi} holds with $\cI t$ replaced by  $R$.
\end{definition}

\subsubsection{Resolution Invariant for Rees Algebras}
\begin{definition}
For Rees algebra $R$, define the invariant:
\[ \inv_p(R) := \max \{ (\overline{b}_1, \ldots, \overline{b}_k) \mid R \subseteq \cA \}. \]
\end{definition}

\begin{remark}
The unique center $\cA$ attaining this maximum exists (see Section~\ref{algo}).
\end{remark}

\subsubsection{Order of Rees Algebras}
\begin{definition}
Let \( R = \bigoplus R_a t^a \) be a Rees algebra. For any \( f \in R_b \), define the order at a point \( p \in X \) by
\[
\ord_p(ft^b) := \frac{\ord_p(f)}{b}.
\]
Then the order of the Rees algebra \( R \) at \( p \) is given by
\[
\ord_p(R) := \min_{a \in \Gamma_{>0}} \left\{ \frac{\ord_p(R_a)}{a} \right\},
\]
where \( \Gamma_{>0} \) denotes the set of degrees \( a > 0 \) such that \( R_a \neq 0 \).
\end{definition}

\begin{lemma}
The order of an ideal $\cI$ at a point $p$ coincides with the order of its associated Rees algebra:
\[
\ord_p(\cI) = \ord_p\left(\cO_X[\cI t]\right). \qed
\]

\end{lemma}
\begin{lemma}\label{order}
If \( R \subseteq \mathcal{A} = \mathcal{O}_X[\overline{x}_1 t^{1/a_1}, \ldots, \overline{x}_k t^{1/a_k}]^{\mathrm{int}} \), then \( a_1 \leq \ord_p(R) \), and the maximum value \( a_1 = \ord_p(R) \) is achieved for some admissible centers.
\end{lemma}
\begin{proof}
The inclusion $R \subseteq \cA \subseteq \cO_X[m_pt^{1/a_1}]^\Int$ implies $R_a \subseteq m_p^{\lceil a a_1 \rceil}$, and thus $\ord_p(R) \geq a_1$. Conversely, if $\ord_p(R) = a_1$, then for all $a$, $R_a \subseteq m_p^{\lceil a a_1 \rceil}$, so $R \subseteq \cO_X[m_pt^{1/a_1}]^\Int$.
\end{proof}

\subsubsection{Associated Invariants and Gradation}
For a Rees center $\cA$, define:
\[ \inv(\cA) = (\overline{b}_1, \ldots, \overline{b}_k), \quad \inv^1(\cA) = \overline{b}_1 \quad \ord(\cA)=a_1\]
Define:
\[ \inv^1_p(R) := \max \{ \inv^1(\cA) \mid R \subseteq \cA \}. \]

Note that in the above definition, it suffices to consider centers \( \mathcal{A} \) of order \( a_1 = \operatorname{ord}_p(R) \). In this case, the invariant \( \inv^1(\mathcal{A}) \) is a \( k \)-tuple consisting of \( a_1 \) and  \( a_{1+} \), which takes on only finitely many possible values. 

Therefore, there exists a center \( \mathcal{A} \) such that
\[
\inv^1(\mathcal{A}) = \inv^1_p(R),
\]

\subsubsection{Graded Rees Algebra and Initial Forms}

\begin{lemma} \label{g2}
Let \( R = \bigoplus R_b \) be a Rees algebra of order \( a_1 \) at a point \( p \in X \), so that \( R_b \subseteq m_p^{b a_1} \) for all \( b \). Then the associated graded algebra of \( R \) at \( p \) is generated by the initial forms
\[
\inn_p(f t^b) := (f + m_p^{b a_1 + 1}) t^b \quad \text{for } f \in R_b,
\]
and is given by
\[
\gr_p(R) := \bigoplus \frac{R_b + m_p^{b a_1 + 1}}{m_p^{b a_1 + 1}} t^b 
\subseteq \mathcal{O}_X\left[\left( \frac{m_p}{m_p^2} \right) t^{1/a_1} \right].
\]

This graded algebra can be viewed as a subalgebra
\[
\gr_p(R) \subseteq K_p[x_1, \ldots, x_k][t^{1/a_1}],
\]
where \( x_1, \ldots, x_k \) form a local system of parameters at \( p \), and \( K_p = \mathcal{O}_{X, p} / m_p \) is the residue field. The algebra \( \gr_p(R) \) is generated by elements of the form
\[
\inn_p(f t^b) := \inn_p(f)\, t^b,
\]
where \( \inn_p(f) \) is the initial form of \( f \in R_b \) of degree \( b a_1 = \ord_p(f) \) in the expansion
\[
f(x) \in R_b \subseteq \widehat{\mathcal{O}}_{X, p} \cdot R_b \subseteq \widehat{\mathcal{O}}_{X, p} = K_p[[x_1, \ldots, x_k]],
\]
in the gradation $t^b=(t^{1/a_1})^{ba_1}$ or zero if \( \ord_p(f) > b a_1 \).

\end{lemma}
\begin{lemma}\label{g3}
If $R$ has order $a_1$ at $p$, then:
\begin{itemize}
  \item $\ord_p(R) = \ord_p(\gr_p(R))$,
  \item $\inv_p^1(R) = \inv_p(\gr_p(R)) = \inv_p^1(\gr_p(R))$.
\end{itemize}
\end{lemma}
\begin{proof}
If \( R \subseteq \mathcal{A} \) and \( \inv^1_p(R) = \inv^1(\mathcal{A}) \), then it follows that
\[
\gr_p(R) \subseteq \gr_p(\mathcal{A})=K_p[\overline{x}_1 t^{1/a_1}],
\]
and hence,
\[
\inv^1(\gr_p(R)) \geq \inv^1(\gr_p(\mathcal{A})) = \inv^1(\mathcal{A})= \inv^1(R).
\]

Conversely, suppose \( \gr_p(R) \subseteq \mathcal{A} = \mathcal{O}_X[\overline{x}_1 t^{1/a_1}]=\gr_p(\mathcal{A} )\), and assume
\[
\inv_p(\gr_p(R)) = \inv^1_p(\gr_p(R)) = \inv^1(\mathcal{A}) = \inv(\mathcal{A}).
\]
Then, for some \( a_2 > a_1 \), there exists an extension
\[
R \subseteq \mathcal{A}' := \left( \mathcal{O}_X[\overline{x}_1 t^{1/a_1}, \overline{x}_2 t^{1/a_2}] \right)^{\mathrm{Int}},
\]
together with a coordinate system \( (\overline{x}_1, \overline{x}_2) \), such that
\[
\inv^1_p(R) \geq \inv^1(\mathcal{A}') = \inv^1(\mathcal{A}) = \inv^1_p(\gr_p(R))= \inv_p(\gr_p(R)).
\]\end{proof}

\subsubsection{Uniqueness of presentation of the invariant of  centers}
The following Lemma shows that $\inv(\cA)$ is well-defined and  independent upon presentation.
We shall need the following result:
\subsubsection{Replacement Lemma}
\begin{lemma} \label{val} (see also \cite{ATW-weighted} for non-divisorial case, $i=1$).

Let $\cA=\cO_X[\overline{x}_1 t^{1/a_1},\ldots\overline{x}_kt^{1/a_k}]^{\inte}$ be a center and $p\in V(\cA)$ be a point.
Let $\overline{x}'$ be a system of local parameters  adapted  to $E$ at a point $p$, such that $$\overline{x}' t^{1/a_1} \subset \cA=\cO_X[\overline{x}_1 t^{1/a_1},\ldots\overline{x}_kt^{1/a_k}]^\inte$$  in a neighborhood  of $p$ then 
one can find  the coordinates $\overline{x}'_1,\ldots, \overline{x}_k$ such that  $\overline{x}\subseteq \overline{x}_1'$ and   $$\cA=\cO_X[\overline{x_1}' t^{1/a_1},\ldots,,\ldots, \overline{x}_kt^{1/a_k}]^\inte$$ in a neighborhood of $p\in X$.
	\end{lemma}
	\begin{proof}
By Lemma~\ref{comp3} and the assumption \( a_1 \leq \cdots \leq a_k \), we have
\[
\mathcal{A}_{1/a_1} \subset (\overline{x}_1) + m_p^2.
\]Thus upon the coordinate change of $\overline{x}_1$ adapted  to $E$ in image in $m_p/m_p^2$, of the set of coordinates $\overline{x}$ is a subset of the image of coordinates  in $\overline{x}_1$. So one can extend $\overline{x}$, and assume that $\overline{x}$ and $\overline{x}_1$ define  the same images in  $m_p/m_p^2$.
Then  in the completion $\widehat{\cO_{X,p}}$ we can write
equality of the vectors of the coordinates
$$\overline{x}'_1=\overline{x}_1+\overline{g},\quad  \overline{x}'_j=\overline{x}_j,\quad j\geq 2, $$ 
where the coordinates of vector $\overline{g}$ are in $\cA_{a_1}\cap \,\,m_p^2$.
This determines an automorphism of $\widehat{\cO_{X,p}}$  which takes $\widehat{\cO_{X,p}}\cdot\cA$ into $\widehat{\cO_{X,p}}\cdot \cA$, and determines the desired coordinate change.
\end{proof}

\begin{corollary} \cite{ATW-weighted}  Assume that a  center $\cA$ has two different presentations at a point $p\in X$:
$$\cA=\cO_X[\overline{x}_1 t^{1/a_1},\ldots\overline{x}_kt^{1/a_k}]^{\inte}=\cO_X[\overline{x}'_1 t^{1/a'_1},\ldots\overline{x}'_{k'}t^{1/a_{k'}}]^{\inte} $$ 
then the associated invariants $(\overline{b}_1,\ldots,\overline{b}_k)=(\overline{b}'_1,\ldots,\overline{b}'_{k'})$ are the same.
	\end{corollary}
\begin{proof}  By Lemma \ref{order}, 
$a_1=a'_1=\ord_p(\cA)$, which can be verifed for generators.
By Lemma \ref{val}, applied to both  presentations, we can assume that $\overline{x_1}=\overline{x'_1}$.
Restricting  both algebras to $V(\overline{x}_1)=V(\overline{x'}_1)$ we get the equality by the inductive  assumption.
	\end{proof}
\subsubsection{Filtered presentation of the center} \label{filter}
Given a sequence of coordinates adapted to $E$,
\[
\overline{x}_1 \in \mathcal{O}_X, \quad \overline{x}_i \in \mathcal{O}_X / (\overline{x}_1, \ldots, \overline{x}_{i-1})=\cO_{H_{i-1}} \quad \text{for } i = 2, \ldots, k,
\]
which determines a filtration of smooth subvarieties
\[
X \supset H_1 := V(\overline{x}_1) \supset H_2 := V(\overline{x}_1, \overline{x}_2) \supset \cdots \supset H_{k-1} := V(\overline{x}_1, \ldots, \overline{x}_{k-1}),
\]
adapted to $E$, and  a sequence of rational numbers $a_1 > a_2 > \cdots > a_k > 0$, one can associate a center in {\it filtered form} as
\[
\mathcal{A} = \mathcal{O}_X[\overline{x}_1 t^{1/a_1}, \ldots, \overline{x}_k t^{1/a_k}]^{\mathrm{int}}.
\]

Taking representatives $\overline{y}_i\subset \cO_X$ one can write the center as
\[
\mathcal{A} = \mathcal{O}_X[\overline{y}_1 t^{1/a_1}, \ldots, \overline{y}_k t^{1/a_k}]^{\mathrm{int}},
\]
where $\overline{x}_1 = \overline{y}_1$ and $\overline{x}_i = \overline{y}_i \mod (\overline{y}_1, \ldots, \overline{y}_{i-1})$ for $i > 1$. 
 For any two choices of representatives \( \overline{y}_i \) and \( \overline{y}_i' \), the difference
$
\overline{y}_i t^{1/a_i} - \overline{y}_i' t^{1/a_i}
$
belongs, by induction, to the subalgebra 
$$
\mathcal{O}_X[\overline{y}_1 t^{1/a_1}, \ldots, \overline{y}_{i-1} t^{1/a_{i-1}}]^{\mathrm{int}}=\mathcal{O}_X[\overline{y}'_1 t^{1/a_1}, \ldots, \overline{y}'_{i-1} t^{1/a_{i-1}}]^{\mathrm{int}},
$$
since \( a_1 \leq \cdots \leq a_k \).  
This shows that \( \mathcal{A} \) is independent of the particular choice of representatives \( \overline{y}_i \).
\section{Resolution Algorithm via Cobordant Blow-ups}

\subsection{Coefficient ideal of Rees algebra}
\subsubsection{Differential operators preserving centers}
A  consequence of the condition $a_1\leq\ldots\leq a_n$ is that the center $A$ is preserved by the action of $D_{x_it^{1/a_1}}=\partial_{x_i} t^{1/a_1}$ for $i=1,\ldots,k$.
\begin{lemma} \label{obvious2} Let $A=\cO_X[{x}_1t^{1/a_1},\ldots,{x_k}t^{1/a_k}]^\inte$, and $\overline{x}=({x}_1,\ldots,{x}_n)$, where $n\geq k$, be a complete coordinate system on $X$ extending $({x}_1,\ldots,{x}_k)$.
If $ft^a\in A_a t^a$, and $|\alpha|/a_1< a$ then
$$D_{\overline{x}^{\alpha} t^{-|\alpha|/a_1}}(ft^a):=\frac{\partial^{|\alpha|}}{\partial x_1^{\alpha_1}\ldots \partial x_n^{\alpha_n}}ft^{a-|\alpha|} \in A_{a-(|\alpha|/a_1)} t^{a-(|\alpha|/a_1)}, $$	

\end{lemma}
 \begin{proof}
 
 The property can be verified on monomials $(x_1^{b_1}\cdot\ldots\cdot x_k^{b_k})t^a\in A_{a}t^a$ and on the differential operators $D_{x_it^{1/a_1}}$ for $i=1,\ldots,k$.
 If $x_1^{b_1}\cdot\ldots\cdot x_k^{b_k}\in A_{a}$ then, by Lemma \ref{comp2-short}, 
 $b_1/a_1+\ldots b_k/a_k\geq a$. 
 So $b_1/a_1+\ldots +(b_i-1)/a_i+\ldots +b_k/a_k\geq a-(1/a_1)$ and again by Lemma \ref{comp2-short}, 
 $$D_{x_it^{1/a_1}}(x_1^{b_1}\cdot\ldots\cdot x_k^{b_k} \cdot t^a)\sim (x_1^{b_1}\cdot\ldots\cdot x_i^{b_i-1}\cdot \ldots  \cdot x_k^{b_k} \cdot t^{a-(1/a_1)})\in A_{a-(1/a_1)}$$
 	
 \end{proof}

\subsubsection{Splitting of Derivations and Compatibility}




\begin{lemma}
Let $X$ be smooth over a field $K$ with coordinates $x_1,\ldots,x_n$ at $p \in X$. Let  $H = V(x_1,\ldots,x_k)$ be a smooth subvariety   in  $X$ with the inclusion $$i_D: \cO_H\hookrightarrow \widehat{\cO}_{X,H}=\lim \cO_X/\cI^i_H \simeq \cO_H[[x_1,\ldots,x_k]],$$ where $D:=\{D_1=\frac{\partial}{\partial x_1},\ldots,D_k=\frac{\partial}{\partial x_k}\}$  is a set of derivations
 which vanish on $$i_D(\cO_H)=: \widehat{\cO}_{X,H}^D=\{f\in \widehat{\cO}_{X,H} \quad D_i(F)\equiv 0, \quad D_i\in D \}\subset \widehat{\cO}_{X,H}\}.$$ 
 
 The  inclusion $i_D$ admits the natural right inverse $$\widehat{\cO}_{X,H}\to \cO_H, \quad \quad f\mapsto f_{|H}\in \widehat{\cO}_{X,H}/\cI_H=\cO_H.$$
Thus $f\to f_{|H}$ and $i_D$ determine  the identification of $\widehat{\cO}_{X,H}^D$ with $\cO_H$.
\end{lemma}

\begin{lemma} \label{spli2}Given a partial coordinate system \( \overline{x}_1 \) , consider an extension \( (\overline{x}_1, \overline{y}) \) and a set of derivations \( D^1 = \{D_1, \ldots, D_k\} \), where \( D_i = \frac{\partial}{\partial x_i} \) for \( x_i \in \overline{x}_1 \). Then for any center with presentation
\[
\mathcal{A} = \widehat{\mathcal{O}}_{X,H}\left[\overline{x}_1 t^{1/a_1},\ \ldots,\ \overline{x}_k t^{1/a_k}\right]^{\mathrm{int}},
\]
we have the compatibility:
\[
\widehat{\mathcal{O}}_{X,H} \cdot \mathcal{A} 
= \widehat{\mathcal{O}}_{X,H}\left[\overline{x}_1 t^{1/a_1},\ i_{D^1}(\overline{x}_{2|H}) t^{1/a_2},\ \ldots,\ i_{D^1}(\overline{x}_{k|H}) t^{1/a_k}\right]^{\mathrm{int}}.
\]

Equivalently, this can be written as:
\[
\widehat{\mathcal{O}}_{X,H} \cdot \mathcal{A} 
= \widehat{\mathcal{O}}_{X,H}\left[\overline{x}_1 t^{1/a_1},\ i_{D^1}(\mathcal{A}_{|H})\right]^\inte.
\]

\end{lemma}

\begin{proof}
Let $\overline{x}_i' := i_D(\overline{x}_{i|H_1})$. Then $\overline{x}_i' - \overline{x}_i \in (\overline{x}_1)$, so $\overline{x}_i' t^{1/a_i} - \overline{x}_i t^{1/a_i} \in (\overline{x}_1 t^{1/a_i}) \subset \cA$. Thus, replacing $\overline{x}_i$ with $\overline{x}_i'$ gives an equivalent presentation.
\end{proof}
This   can be translated into  the filtered form:
\begin{lemma} $\cA=\cO_X[\overline{x}_1t^{1/a_1},\cA_{|V(\overline{x}_1)}]$ \qed
\end{lemma}
\begin{proof} Follows from the fact that $$\widehat{\cO}_{X,H}\cdot \cA=\widehat{\cO}_{X,H}[\overline{x}_1 t^{1/a_1},i_{D^1} (\cA_{|H})]=\widehat{\cO}_{X,H}[\overline{x}_1 t^{1/a_1},\cA_{|H}].$$
\end{proof}

\subsubsection{Coefficient ideal of Rees algebra} The origin of the coefficient ideals concept can be traced back to the work of Abhyankhar and Hironaka, as seen in \cite{Hironaka}. Various definitions of this notion have been explored in multiple studies, including \cite{Villamayor}, \cite{Bierstone-Milman-simple}, \cite{Bierstone-Milman}, \cite{Wlodarczyk}, and \cite{Kollar}, among others. The approach adopted in this work is closely aligned with the definitions presented in \cite{Bierstone-Milman-simple} and \cite{Bierstone-Milman} within the framework of resolution by smooth centers, but it does not require factorization by monomials associated with the exceptional divisors.

Let $R=\cO_X[f_jt^{b_j}]_{j=1,\ldots,s}$ be a Rees algebra  generated by $f_jt^{b_j}$  for $j=1,\ldots,s$ , $a_1=\ord_p(R)$ and let $\overline{x}$  be any partial system of local coordinates adapted  to $E$ at $p$, and  set $H:=V(\overline{x})$ and 

For any element  $f_jt^{b_j} \in R_{b_j}$, one can write $f_j$ in $\widehat{\cO}_{X,H}$  as $$f_j\equiv \sum c_{j\alpha} \overline{x}^\alpha =\sum_{  |\alpha|<  b_j a_1} c_{j\alpha} \overline{x}^\alpha + \sum_{  |\alpha|\geq  b_j a_1} c_{j\alpha} \overline{x}^\alpha \in \widehat{\cO}_{X,H}$$ 
 where $ c_{j\alpha}\in i_{D^1}({\cO}_{H})=\widehat{\cO}_{X,H}^{D^1} \subset   \widehat{\cO}_{X,H}$ and $c_{j\alpha|H}\in \cO_H$ for $H:=V(\overline{x})$.
 .

\begin{definition}
The \emph{coefficient ideal} of the Rees algebra 
\[
R = \mathcal{O}_X[f_j t^{b_j}]
\]
with respect to the element \(\overline{x} t^{1/a_1}\) and the set $D$ of derivations $D_i=\frac{\partial}{\partial x_i}$, $x_i\in \overline{x}$ in a certain system of coordinates extending $\overline{x}$, is the Rees algebra on the hypersurface \(H := V(\overline{x})\) defined by:
\[
C_{\overline{x} t^{1/a_1}}(R) := \mathcal{O}_H\left[c_{j\alpha|H} \, t^{b_j - |\alpha|/a_1} \,\big|\, |\alpha| < b_j a_1 \right],
\]
where each generator \(f_j t^{b_j} \in R\) admits the expansion:
\[
f_j t^{b_j} = \left( \sum_{\alpha} c_{j\alpha} \, \overline{x}^\alpha \right) t^{b_j} 
= \sum_{|\alpha| < a_1 b_j} c_{j\alpha} \, t^{b_j - |\alpha|/a_1} \cdot \overline{x}^\alpha t^{|\alpha|/a_1} \mod (\overline{x}t^{1/a_1})^{a_1 b_j}
,\] where $c_{j\alpha} \in i_{D^1}(\cO_H)$.
\end{definition}

\subsubsection{Coefficient ideals and admissibility}
\begin{lemma} \label{Coef1}  Let $R=\cO_X[f_jt^{b_j}]$ be the Rees algebra on $X$. The following conditions are equivalent for $(\overline{x}_1,D)$ in a neighborhood of  $p\in V(\cA)\subset H:=V(\overline{x}_1)\subset X$.
\begin{enumerate}
\item $R\subseteq \cA^\Int=\cO_X[\overline{x}_1t^{1/a_1},\ldots,\overline{x_k}t^{1/a_k}]^{\Int}$.
\item $\cC_{\overline{x}_1t^{1/a_1}}((R)\subseteq \cA_{|H}^\Int=\cO_H[\overline{x}_{2|H}t^{1/a_2},\ldots,\overline{x}_{k|H}t^{1/a_k}]^\Int$.
 \end{enumerate}
\end{lemma}

\begin{proof} We can pass to the completion $\widehat{\cO}_{X,p}$, and replace $\overline{x}_i$ with $\overline{x}'_i :=i_{D^1}(\overline{x}_{i|H})$, so that $D^1(\overline{x}'_i)=0$ for $i>1$. 

$(1)\Rightarrow (2)$ by Lemma \ref{obvious2}, the  operator $D_{\overline{x}_1t^{-1/a_1}}=\frac{\partial}{\partial \overline{x}_1}t^{-1/a_1}$ preserves the center $\cA$ so if $f_jt^{b_j}\in \cA$ then $$\frac{1}{\alpha!}D_{\overline{x}^\alpha t^{-|\alpha|/a_1}}(f_j t^{b_j})_{|H}=c_{j\alpha|H}t^{b_j-|\alpha|/a_1}\in \cA_{|H}.$$

 $(2)\Rightarrow (1)$ 
  Conversely suppose $c_{j\alpha{|H}}t^{b_j-|\alpha|/a_1}\in \cA_{|H})$. Note that  $i_{D^1}(\cA_{|H})\subset \widehat{\cO}_{X,p}\cdot \cA$  since $i_{D^1}(\overline{x}_{i|H})t^{1/a_i}\subset \widehat{\cO}_{X,p}\cdot \cA$. Thus  $$c_{j\alpha}t^{b_j-|\alpha|/a_1}=i_{D^1}(c_{j\alpha{|H}}t^{b_j-|\alpha|/a_1})\in i_{D^1}(\cA_{|H})\subset \widehat{\cO}_{X,p}\cdot \cA$$
  
  Then since  $\overline{x}_1t^{1/a_1}$  and  $c_{j\alpha}t^{b_j-|\alpha|/a_1} \in \widehat{\cO}_{X,p}\cdot\cA$ each $f_jt^{b_j}\in  \widehat{\cO}_{X,p}\cdot \cA$.
 \end{proof}



\subsection{Maximal contact}

\subsubsection{Maximal contact of Rees algebra} \label{maxims}

\begin{definition}\label{cota}
Let $R = \cO_X[f_j t^{b_j}]_{j=1,\ldots,s}$ be a Rees algebra on an open affine $U \subset X$, and let $a_1 > 0$ be rational. The \emph{cotangent ideal} of $R$ with respect to $a_1$ is
\[
T^{1/a_1}(R) := \sum_{|\alpha| = b_j a_1 - 1} \cO_X D_{\overline{x}^\alpha}(f_j).
\]
In graded form:
\[
T^{1/a_1}(R)t^{1/a_1} := \sum_{|\alpha| = b_j a_1 - 1} D_{\overline{x}^\alpha}(f_j) t^{1/a_1}.
\]
\end{definition}
From Lemma \ref{obvious2} we obtain that the cotangent ideal consists of  elements of  $t^{1/a_1}$-gradations $\cA_{1/a_1}$ of  $R$ admissible centers $\cA$ of order $a_1$.
\begin{lemma}  \label{easy} If $R\subset \cA=\cO_X[\overline{x}_1t^{1/a_1},\ldots,\overline{x_k}t^{1/a_k}]^{\Int}$ then $T^{1/a_1}(R)\subset \cA_{1/a_1}$
\qed
\end{lemma}

Since  derivations commute with initial forms we obtain

\begin{lemma}
If $R$ is generated by $f_j t^{b_j}$, then $\gr_p(R)$ is generated by $\inn_p(f_j)t^{b_j}$ and
\[
\inn_p(T^{1/a_1}(R)) = T^{1/a_1}(\gr_p(R)). \qed
\]
\end{lemma}
\begin{lemma}
\[
\inv_p^1(T^{1/a_1}(R)t^{1/a_1}) = \inv_p(\inn_p(T^{1/a_1}(R))t^{1/a_1}) = \inv_p(\gr_p(R)) = \inv_p^1(R).
\]
\end{lemma}
\begin{proof}
Choose coordinates $x_1, \ldots, x_k \in T^{1/a_1}(R)$ such that $\inn_p(x_i)$ span $\inn_p(T^{1/a_1}(R))$. Then $\gr_p(R)\subseteq K_p[\inn_p(x_i)t^{1/a_1}]$, hence the result via Lemmas \ref{easy} and  \ref{g3}.
\end{proof}

The linear part of the cotangent ideal is uniquely determined, as is the linear part of the $t^{1/a_1}$-gradation of  a maximal admissible center:
\begin{lemma} \label{max}
Let $E_p$ denote the set of divisor components through $p \in X$. Consider the induced logarithmic structure $\overline E_p := \inn(E_p)$ on $\Spec(\gr_p(\mathcal O_{X,p}))$ and  the graded algebra $\gr_p(R)$ associated with $R$ of order  $a_1$ at $p$. Then:
\begin{enumerate}
	\item	$\inv_p(\gr_p(R)) = a_1\inv(\overline{x}_1)$ is uniquely determined by the minimal subspace $\overline{x}_1 \subset m_p / m_p^2$ adapted to $\overline E_p$ and containing $\inn_p(T^{1/a_1}(R))$, where $\overline{\cA}: =K_p[\overline{x}_1t^{1/a_1}]$ is a maximal admissible center for $\gr_p(R)$.
	\item The divisorial part $V_p \subset \overline{x}_1$ is the smallest subspace of $\spa(\overline E_p)$ generated by the divisorial coordinates $\overline y_i \in \overline E_p$ and containing $\inn_p(T^{1/a_1}(R)) \cap \spa(\overline E_p)$.
	\item	The free part of $\overline{x}_1$ is a maximal set of linearly independent vectors in $\inn_p(T^{1/a_1}(R))$ that lie outside this divisorial span.
	\end{enumerate}
\end{lemma}

\begin{proof}   $\overline{x}_1t^{1/a_1}t^{1/a_1}$  generates  a maximal admissible center for $\inn_p(T^{1/a_1}(R))$ 
if and only if  $\overline{x}_1$ is the smallest subspace adapted to $\overline{E}_p$ and containing $\inn_p(T^{1/a_1}(R))$. 
Moreover, we have
$
{\overline{x}_1}/{V_p} 
= \frac{\inn_p\left(T^{1/a_1}(R)\right)}{V_p \cap \inn_p\left(T^{1/a_1}(R)\right)} 
= \frac{T^{1/a_1}(R) + \mathrm{span}(E_p)}{m_p^2 + \mathrm{span}(\overline{E}_p)}
$ and thus \\${\overline{x}_1}=\inn_p(\overline{x}_1)$ is uniquely determined. The remaining part follows.
\end{proof}

As a consequence of the construction we have

\begin{lemma}
Suppose \( \inv_p^1(R) = (a_1, \ldots, a_1, a_{1+}, \ldots, a_{1+}) \), with \( s \) free and \( (r - s) \) divisorial components. Then there exists a partial coordinate system \( \overline{x} = (x_1, \ldots, x_r) \) at \( p \), adapted to \( E \), which extends to a full coordinate system \( (x_1, \ldots, x_n) \) at $p$ also adapted to \( E \), such that:
\begin{enumerate}
    \item \( x_i \in T^{1/a_1}(R) \), $x_i\notin E$, for \( i \leq s \) (free components),
    \item \( x_i \in E \) and \( \partial_{x_i}(T^{1/a_1} R) = \mathcal{O}_U \) for \( s < i \leq r \) (divisorial components),
    \item \( \partial_{x_j}(T^{1/a_1} R) \subseteq m_p \) for \( j > r \).
\end{enumerate}
Moreover, the invariant satisfies
\[
\inv_p^1(R) = a_1 \cdot \inv(\overline{x}). \qed
\]
\end{lemma}

\begin{definition} \label{maxim}
Under the above conditions, $\overline{x}$ is called a \emph{maximal contact} of $R$ at $p$. If the conditions (1) and (2)  hold locally on open $U$ for a certain $a_1$, we say that $\overline{x}$ is a \emph{partial maximal contact} of $(R,a_1)$.
\end{definition}
\begin{lemma} \label{obvious}
If $\overline{x}$ is a partial maximal contact on $U$ for $(R,a_1)$, then $\ord_q(R) \leq a_1$ for all $q \in U$.
\end{lemma}
\begin{proof}
There exists a generator $f_jt^{b_j}$ of $R$ and multi-index $|\alpha| = a_1 b_j$, such that $D_{\overline{x}^\alpha}(f_j)$ invertible at $q$, hence  $\ord_q(f_j) \leq a_1b_j$ and 
$\ord_q(f_jt^{b_j}) \leq  a_1$.
\end{proof}

Immediately we see
\begin{lemma}\label{extend} Any partial maximal contact $\overline{x}$ of $R$ for $a_1=\ord_p(R)$  centered at $p$ extends to a maximal contact at $p$ so that $\inv^1_p(R)\leq a_1\inv(\overline{x})$.\qed
\end{lemma}
\begin{example}
Let $R = \cO_X[(x_1+x_2+x_3^2)^3 + y^4 + z^5)t]$ with $a_1 = 3 = \ord_0(R)$, and $x_i$ are divisorial. Then
\[
T^{1/3}(R)t^{1/3} = \cO_X[(x_1+x_2+x_3^2, y^2, z^3)t^{1/3}],
\]
and the maximal contact at $0$ is $(x_1,x_2) t^{1/3}$.Here $\inv_0^1(R)=(3+,3+)=3\cdot\inv(x_1,x_2)$.
\end{example}
\begin{remark}
If no exceptional divisors are present, i.e., \( E = 0 \), one can consider a \emph{single maximal contact} of the form \( x t^{1/a} \), where \( x \) is a local parameter in \( T^{1/a_1}(R) \) at a point \( p \in X \).

Alternatively, one may a choose a local parameter \( x \in D_{\overline{x}^\alpha}(f_j) \), where \( f_j t^{b_j} \in R \) satisfies
\[
\ord(f_j t^{b_j}) = \frac{\ord(f_j)}{b_j} = a_1, \quad \text{and} \quad |\alpha| = b_j a_1 - 1.
\]
Although this approach is conceptually straightforward, it is often less convenient and slower in practice for computing the invariant. Note that in this case, the {\it maximal contact} is determined by a maximal partial system of local coordinates, each of which defines a {\it single maximal contact}.
\end{remark}
\begin{corollary} \label{max66}  
Let \( R \) be a Rees algebra with a maximal contact element \( \overline{x} \) at a point \( p \).  
Suppose \( \mathcal{A} = \mathcal{O}_X[\overline{x}_1 t^{1/a_1}, \ldots, \overline{x}_k t^{1/a_k}]^{\mathrm{int}} \) is an admissible center for \( R \) at \( p \), with \( a_1 = \operatorname{ord}_p(R) \).

\begin{itemize}
    \item Then, after a change of presentation of \( \mathcal{A} \), we have:
$
    \overline{x} \subseteq \overline{x}_1\subset \cA_{1/a_1}.
   $

    \item Moreover, if \( \operatorname{inv}^1_p(R) = \operatorname{inv}^1(\cA) \), then, after a change of presentation of \( \mathcal{A} \), we also have:
$
    \overline{x} = \overline{x}_1.
 $
\end{itemize}

\end{corollary}
\begin{proof} 
 According to Lemmas  \ref{easy} and 
 \ref{val}, the free coordinates in $\overline{x}$ are a part of $\overline{x}_{1}$ after a coordinate change of presentation of $\cA$. Moreover,
 by Lemma \ref{max}(2)  the divisorial coordinates of $\overline{x}_{1}$ contain those of  $\overline{x}$. Consequently, we obtain the  inclusion $\overline{x}\subseteq \overline{x}_1$ which is the equality if $\cA$ if $\inv^1(\cA)=\inv^1_p(R)$  then $$\inv^1_p(R)=a_1\inv(\overline{x})=\inv^1(\cA)=a_1(\inv(\overline{x}_1)).$$ 
 
 \end{proof}

\begin{proposition} \label{max666}    Given a Rees algebra $R$ with maximal contact $\overline{x}_1$ at $p$ and of order  $a_1$. Then if $\cA$ is a maximal admissible center for $R$ then $\cA_{|V(\overline{x}_1)}$ is a maximal admissible center for $C_{\overline{x}_1t^{1/a_1}}(R)$. Conversely if  $\cA$ is a center such that $\ord(A)=a_1$, $ \overline{x}_1$  is a maximal contact for $\cA$ at $p$ , and
$\cA_{|V(\overline{x}_1)}$ is a maximal admissible center for $C_{\overline{x}_1t^{1/a_1}}(R)$ then  $\cA$ is a maximal admissible center for $R$, and $\cA=\cO_X[\overline{x}_1t^{1/a_1},\cA_{|V(\overline{x}_1)}]$ in the filtered form.


\end{proposition}
\begin{proof}  By Corollary \ref{max66},  we can assume that  $\cA$ has a form $\cO_X[\overline{x}_1 t^{1/a_1}, \ldots, \overline{x}_k t^{1/a_k}]^\inte$, when we can apply Lemma \ref{Coef1} which implies that in such a case $\cA$ is maximal admissible for $R$ iff $\cA_{|V(\overline{x}_1)}$  is maximal admissible for $C_{\overline{x}_1t^{1/a_1}}(R)$.

\end{proof}

\begin{lemma} \label{partial}
If $\overline{x}_1 t^{1/a_1}$ is a maximal contact at $p$ with $\ord_p(R) = a_1$, then it is a partial maximal contact on a neighborhood of $p$. 
\end{lemma}
\begin{proof} The conditions extend to a neighborhood. Moreover if $x_i$ is divisorial at $p$ then it is considered divisorial in a neighborhood of $p$.
\end{proof}

\begin{lemma} \label{max4}
Let $\overline{x}$ be a partial maximal contact of $(R, a_1)$ on $U$, and set $b_1 := a_1 \inv(\overline{x})$. Then
the set of the points $\supp(\inv^1(R)\geq \overline{b}_1)$  is contained  in 
$V(\overline{x})$.
\end{lemma}
\begin{proof} Let $\overline{x} = (x_1, \ldots, x_r, x_{r+1}, \ldots, x_s)$, where $x_1, \ldots, x_r$ are free coordinates on $U$, and $x_{r+1}, \ldots, x_s$ are divisorial. 

Let $q \in U \setminus V(x_i)$, where $x_i \in \overline{x}$ is a free coordinate. Then $x_i \in T^{1/a_1}(R) = \cO_X$, and there exists a generator $f_j t^{b_j} \in R_{b_j}$ such that $\cD^{a_1 b_j - 1}(f_j)$ is invertible. It follows that:
\[
\ord_q(f_j t^{b_j}) \leq \frac{a_1 b_j - 1}{b_j} < a_1,
\]
so $\ord_q(R) < a_1$, and hence
$
\inv_q^1(R) < \overline{b}_1.
$

Now, consider a point $q \in V(x_1, \dots, x_r) \cap U \setminus V(x_i)$, where $x_i$ (for $i > r$) is a divisorial coordinate and $D_{x_i}(T^{1/a_1}(R)) = \mathcal O_X$. If $T^{1/a_1}(R)$ is invertible at $q$, then $\ord_q(R) < a_1$. Otherwise, there exists a local parameter $ u \in T^{1/a_1}(R) $ at $q$ such that $ D_{x_i}( u ) $ is invertible. Since $q \notin V(x_i)$, the coordinate $ x_i $ is free at $q$. Hence, $ u $ is also a free variable, linearly independent from both the free coordinates $ x_1, \dots, x_r $ and the other divisorial coordinates at $q$. 
This yields a partial maximal contact $(x_1, \ldots, x_r, u)$ with at least $r+1$ free coordinates which can be extended to  a maximal contact $\overline{x}'$ we have:
\[
\inv_q^1(R)= a_1\inv\overline{x}'\leq  a_1 \inv_1(x_1, \ldots, x_r, u) < \overline{b}_1 = a_1 \inv(\overline{x}).
\]

Note that, by construction, the $(r+1)$-th component of $a_1 \inv_1(x_1, \dots, x_r, u)$ is $a_1$, 
whereas in $\overline b_1 = a_1 \inv(\overline x)$ the corresponding component is $a_1^+$.
\end{proof}
\begin{lemma} \label{Giraud3}
$\ord_p(R)$ is upper semicontinuous. Thus, $\supp (\ord(R) \geq a_1)$ is closed.
\end{lemma}
\begin{proof} Let $R=\cO_X[f_jt^{b_j}]$. Then 
$\ord_q(R) \geq a_1$ iff all $D_{\overline{x}^\alpha}(f_j)$ vanish for $|\alpha| < b_j a_1$.
\end{proof}
This well known fact generalizes  the classical result which goes back to Hironaka:
\[
\supp\left(\ord(\cI) \geq a\right) = V\left(\cD^{\leq a-1}(\cI)\right),
\]
where $\cD^{\leq a-1}(\cI)$ denotes the ideal generated by $\cI$ and all the derivatives of order at most $a-1$ or the functions in $\cI$.
\begin{lemma} \label{Giraud2}
Let $\overline{x}=(x_1,\ldots,x_r)$ be a partial maximal contact on open subset $U$ which extends to coordinate system $(x_1,\ldots,x_n)$ adapted to $E$. Then on $U$  we have $$\inv^1(R)\leq \overline{b}_1:=a_1\inv(\overline{x}),$$ and 
\[
\supp (\inv^1(R)= \overline{b}_1) = \supp(\ord(R)= a_1) \cap V(\overline{x})\cap V(\sum^n_{i=r+1}\frac{\partial}{\partial x_i}(T^{1/a_1}R)))\]
is closed. In particular $\inv^1(R)$ is upper semicontinuous.
\end{lemma}
\begin{proof} Follows immediately from the definition of the maximal contact and Lemmas \ref{extend}, \ref{obvious} and  \ref{max4}.
\end{proof}

\begin{example}
Let $\cI = (x_1^2 + x_2)$ in $K[x_1, x_2]$, with $x_2$ divisorial. Then:
\[
T^1(\cI t) = (x_1^2 + x_2)t, \quad \inv^1_0(\cI t) = (1_+), \quad \supp (\inv^1(\cI t)=(1_+)) \subseteq V(x_2).
\]
At a point $p \neq 0$, the coordinate $x_1' := x_2 + x_1^2$ is free, which implies upper semicontinuity of the invariant:
\[
\inv^1_p(\cI t) = (1) < (1_+).
\]
\end{example}

 \subsection{Effective Algorithm  and Resolution Principle}\label{algo1}

\subsubsection{Algorithm and  Uniqueness of the Center}\label{algo}

Let \(R\) be a Rees algebra on a smooth variety \(X\). We construct a maximal admissible center \(\mathcal A\) for \(R\) at a point \(p \in X\) via a recursive procedure. As a reference, assume that a maximal admissible center 
\[
\mathcal A = \mathcal O_X\big[\overline x_1 t^{1/a_1}, \dots, \overline x_k t^{1/a_k}\big]^{\rm int}
\]
exists and can be expressed in the filtered form
\[
\mathcal A = \mathcal O_X\Big[ \overline x_1 t^{1/a_1}, \overline x_{2|V(\overline x_1)} t^{1/a_2}, \dots, \overline x_{k|V(\overline x_1, \dots, \overline x_{k-1})} t^{1/a_k} \Big]^{\rm int}.
\]
The center \(\mathcal A\) itself does not change throughout the process; only its presentation is adapted. This does not affect the construction and serves exclusively to clarify the procedure, thereby proving its uniqueness.

Set \( R_1 := R \subseteq \mathcal{A}_1 := \mathcal{A} \). By Lemma~\ref{order}, we have \( \mathrm{ord}_p(R_1) = a_1 \). Let \( \overline{x}_1' \subseteq T^{1/a_1}(R_1) \) be a maximal contact at \( p \), and define \( H_1 := V(\overline{x}_1') \). By Lemma~\ref{max66}, we may assume that \( \overline{x}_1 = \overline{x}_1' \) appears in the presentation of \( \mathcal{A} \), so that:
\[
\mathcal{A} = \mathcal{O}_X\left[\overline{x}_1 t^{1/a_1}, \mathcal{A}_{1|V(\overline{x}_1})\right]^{\mathrm{int}}.
\]

By Lemma~\ref{max666}, the restriction \( \mathcal{A}_2 := \mathcal{A}_{1|H_1}= \mathcal{O}_{H_1}\left[\overline{x}_{2|H_1} t^{1/a_2}, \ldots, \overline{x}_{k|H_2} t^{1/a_k}\right]^{\mathrm{int}} \) is a maximal admissible center for the coefficient algebra on $H_2=V(\overline{x}_1)$.
\[
R_2 := C_{\overline{x}_1 t^{1/a_1}}(R_1)
\]

\paragraph{\bf Recursive Step:}
For \( i \geq 2 \), proceed inductively:
\begin{itemize}
    \item $\cA_i=\cA_{i-1|H_{i-1}}=\cA_{|H_{i-1}}=\mathcal{O}_{H_{i-1}}\left[\overline{x}_{i|H_{i-1}}t^{1/a_{i}}, \ldots, \overline{x}_{k|H_{i-1}} t^{1/a_k}\right]^{\mathrm{int}} \)  is a maximal admissible center for $R_i$ at $p$, where $H_{i-1}=V(\overline{x}_1,\ldots,\overline{x}_{i-1})$  and $\cA=\mathcal{O}_X\left[\overline{x}_1 t^{1/a_1}, \ldots, \overline{x}_{i-1} t^{1/a_{i-1}}, \mathcal{A}_{i}\right]^{\mathrm{int}}$ in the filtered from.
    \item By Lemma~\ref{order}, \( \mathrm{ord}_p(R_i) = a_i \). Let \( \overline{x}_i' \) be a maximal contact for \( R_i \), and assume that \( \overline{x}_{i|H_{i-1}} = \overline{x}_i' \) appears in the presentation of the maximal admissible center \( \mathcal{A}_i= \mathcal{A}_{|H_{i-1}}\) for $R_i$ on $H_{i-1}$.
   
   \item Define the next Rees algebra on $H_i=V(\overline{x}_1,\ldots,\overline{x}_{i})$ by
\[
R_{i+1} := C_{\overline{x}_i t^{1/a_i}}(R_i):=C_{\overline{x}_i|H_{i-1} t^{1/a_i}}(R_i),
\]
with maximal admissible center
\[
\mathcal{A}_{i+1} = \mathcal{A}_{i|H_i} = \mathcal{O}_{H_i}\left[\overline{x}_{i+1|H_i} t^{1/a_{i+1}},\ \ldots,\ \overline{x}_{k|H_i} t^{1/a_k}\right]^{\mathrm{int}},
\]
where \( H_i := V(\overline{x}_1, \ldots, \overline{x}_i) \), and see Lemma~\ref{Coef1} for justification. The full center \( \mathcal{A} \), by Lemma~\ref{max666} has the filtered presentation:
\[
\mathcal{A} = \mathcal{O}_X\left[\overline{x}_1 t^{1/a_1},\ \ldots,\ \overline{x}_i t^{1/a_i},\ \mathcal{A}_{i+1}\right]^{\mathrm{int}}.
\]

\end{itemize}

\paragraph{\bf Termination:}
Eventually, the process terminates at \( R_{k+1} = 0 \), with \( \mathcal{A}_{k+1} = 0 \). This gives a filtered presentation of the maximal admissible center \( \mathcal{A} \) for \( R \):
\[
\mathcal{A} = \mathcal{O}_X\left[\overline{x}_1 t^{1/a_1}, \ldots, \overline{x}_k t^{1/a_k}\right]^{\mathrm{int}}.
\]
  The procedure, which is independent of $\cA$, results in a unique, predetermined center $\cA$ that is maximal admissible for $R$ at $p$.

The associated extended Rees algebra is given by:
\[
\mathcal{A}^{\mathrm{ext}} = \mathcal{O}_X\left[t^{-1/w_\cA}, \overline{x}_1 t^{1/a_1}, \ldots, \overline{x}_k t^{1/a_k}\right], \quad \text{where } w_\cA = \mathrm{lcm}(a_1, \ldots, a_k).
\]

\subsubsection{Existence}
\begin{proposition} \label{exi}  \label{ind}(see also \cite[Theorem 5.3.1]{ATW-principalization} (in the language of $\QQ$-ideals)), ). For any Rees algebra $R$ on a smooth variety $X$ over $K$ with a SNC divisor $E$, and for any point $p\in X$ there exists a uniquely  determined maximal admissible Rees center  $$\cA=\cO_X[\overline{x}_1t^{1/a_1},\ldots, \overline{x}_k t^{1}]^\inte,$$  where $\inv^1_p(R_i)=a_i(\inv_p(\overline{x}_i))$, and  with the inductive formula
\[
\inv_p(R) = \left( \inv^1_p(R),\, \inv_p\big(C_{\overline{x}_1}(R)\big) \right)=\left( \inv^1_p(R_1),\ldots,\inv^1_p(R_k) \right)
\]
 given by the recursion in Section \ref{algo}

\end{proposition}
\begin{proof}

The construction does not rely on the Rees center $\cA$ on the right side of the admissibility condition at any step. Moreover, the inductive process leads to a Rees center $\cA'$ that is admissible for $R$ at $p$. 
We can assume by induction that  the construction for $\cR_2=C_{\overline{x}_{i}t^{1/a_i}}(R)$ gives a unique admissible center $\cA_2$. Then, by Proposition \ref{max666},  $\cA'=\cA=\cO_X[\overline{x}_{1}t^{1/a_1},\cA_2]^\inte$ is maximal admissible for $R$ at $p$.

\end{proof}
\subsubsection{Comparison to the   coefficient ideals in the split form}
\begin{remark} \label{previous}
In an earlier version of this paper \cite{W23}, we used the coefficient ideal  associated with $(\overline{x}_1,D_1)$ in the form
\[
\mathcal{C}_{\overline{x}_1 t^{1/a_1}}(R) = \mathcal{O}_X\left[\overline{x}_1 t^{1/a_1}, i_{D_1}\left(C_{\overline{x}_1 t^{1/a_1}}(R)\right)\right]
\]
when passing to the completion, with the property that
\[
\mathcal{C}_{\overline{x}_1 t^{1/a_1}}(R)_{|V(\overline{x}_1)} = C_{\overline{x}_1 t^{1/a_1}}(R).
\]
This led to the recursive process:
\[
\mathcal{R}_1 = \mathcal{O}_X[\mathcal{I} t], \quad \mathcal{R}_{i+1} = \mathcal{C}_{\overline{x}_1 t^{1/a_1}}(R_i), \quad \mathcal{R}_{k+1} = \mathcal{A},
\]
where $\mathcal{A}$ is a maximal admissible center. This formulation enables control over derivations $D^i$ associated with the center and plays a critical role in the resolution of foliations, as discussed in \cite{ABTW25}.
\end{remark}

\subsubsection{Comparison to the Invariant in \cite{ATW-weighted}}

In the absence of exceptional divisors, one may work with \emph{single maximal contacts}, consisting of a local parameter \( x \in T^{1/a_1}(R) \), or equivalently \( x \in D_{\overline{x}^\alpha}(f)\), where \( f  \in R_b \) ,  ${|\alpha| = b a_1 - 1}$ and  $\ord_p(f)=a_1b$, with $a_1=\ord_p(R)$.

This gives rise to a sequence of Rees algebras:
\[
R = R_1 = \mathcal{O}_X[\mathcal{I} t], \quad R_{i+1} = C_{x_i t^{1/a_i}}(R_i),
\]
with each  \( x_i \in D_{\overline{x}^{\alpha_i}}(f_i)\), where \( f_i  \in (R_i)_{b_i} \) ,  ${|\alpha_i| = b_i a_i - 1}$ and  $\ord_p(f_i)=b_ia_i$, with $a_i=\ord_p(R_i)$


Then, our invariant can be expressed as:
\[
\inv_p(\mathcal{I}) = (\operatorname{ord}_p(R_1), \ldots, \operatorname{ord}_p(R_k)),
\]
which may be viewed as a simplified and more easily computable version of the invariant introduced in \cite{ATW-weighted}. That construction involved normalized orders of coefficient ideals, defined using Villamayor's approach.
\subsubsection{Semicontinuity of canonical invariant. Local admissibility} \label{semi1}

The inductive formula in Section \ref{ind} ensures the semicontinuity of the invariant \( \inv_p \), as established in \cite{ATW-weighted}. We prove this by induction on the dimension \( n = \dim(X) \).

Assume that the function \( p \mapsto \inv_p(R') \) is upper semicontinuous for any Rees algebra \( R' \) on a smooth variety \( H \) of dimension \( n - 1 \). Let \( R \) be a Rees algebra on a smooth \( X \) of dimension \( n \).

Suppose
\[
\inv_p(R) \geq (\overline{b}_1, \ldots, \overline{b}_k).
\]
This occurs if and only if
\[
\inv^1_p(R) > \overline{b}_1,
\quad \text{or} \quad
\inv^1_p(R) = \overline{b}_1 \ \text{and} \ 
\inv_p(C_{\overline{x}_1}(R)_{|H_1}) \geq (\overline{b}_2, \ldots, \overline{b}_k).
\]
This condition defines a closed subset of \( X \), since:
\begin{itemize}
    \item \( \inv^1_p(R) \) is upper semicontinuous (Lemma~\ref{Giraud2}),
    \item \( \inv_p(C_{\overline{x}_1}(R)_{|H_1}) \) is upper semicontinuous by the inductive hypothesis.
\end{itemize}

Hence, \begin{lemma} \label{max88}The canonical invariant \( \inv_p(R) \) is upper semicontinuous. \qed
\end{lemma}
Moreover, by induction and Proposition \ref{max666}, it follows that
\begin{lemma}[see also \cite{ATW-weighted}] \label{max8}
Let $R$ be a Rees algebra, and let $\mathcal{A}$ be a maximal admissible center for $R$ at a point $p \in X$. Then there exists an open neighborhood $U$ of $p$ such that 
\[
\max_U \inv(R) = \inv_p(R) = \inv(\mathcal{A})
\]
is attained precisely at $V(\mathcal{A})$.\qed 
\end{lemma}

\subsubsection{Duality of Rees centers}

The Rees centers \( \mathcal{A} \) admit a dual interpretation:

\begin{itemize}
    \item On one hand, they appear as admissible Rees algebras expressed using a formal (dummy) variable \( t \), for instance:
    \[
    \mathcal{A} = \mathcal{O}_X[\overline{x}_1 t^{1/a_1}, \ldots, \overline{x}_k t^{1/a_k}]^{\mathrm{int}}.
    \]
    
    \item On the other hand, they correspond to the algebras defining full cobordant blow-ups, where the variable \( t^{-1} \) becomes an actual coordinate on the blow-up space \( B \).
\end{itemize}

These two notions are deeply intertwined, as they share the same algebraic structure up to rescaling. However, to avoid notational ambiguity when both gradations (in \( t \) and in the geometry of \( B \)) appear simultaneously, we will adopt the convention of using a separate variable \( t_B \) for the cobordant blow-up space \( B \).

Thus, on \( B \), the blow-up algebra will be written as:
\[
\mathcal{O}_B[t_B^{-1}, \overline{x}_1 t_B^{w_1}, \ldots, \overline{x}_k t_B^{w_k}],
\]
while the admissibility conditions remain expressed in the variable \( t \). This careful distinction allows us to track the behavior of both the Rees algebra and its geometric realization through cobordant blow-up in a coherent and unified framework.
\subsubsection{Controlled transforms of the coordinates and ideals} \label{control3} \cite{ATW-weighted}

Let $\mathcal{I}$ be an ideal on a smooth (regular) variety $X$. The condition that $\mathcal{I}$ is admissible with respect to a center $\mathcal{A}$ can be expressed as:
\[
\mathcal{I} t \subset \mathcal{A}^\ext = \mathcal{O}_X\left[t^{-1/w_A},\, \overline{x}_1 t^{1/a_1},\, \ldots,\, \overline{x}_k t^{1/a_k}\right].
\]

Introducing a rescaled variable \( t_B \), we can rewrite this inclusion in the form:
\[
\mathcal{I} \cdot t_B^{w_A} \subset \mathcal{O}_B := \mathcal{O}_X\left[t_B^{-1},\, \overline{x}_1 t_B^{w_1},\, \ldots,\, \overline{x}_k t_B^{w_k} \right] = \mathcal{O}_X\left[t_B^{-1},\, \overline{x}_1',\, \ldots,\, \overline{x}_k'\right],
\]
where \( w_i = \frac{w_A}{a_i} \), and 
 \( \sigma^c(x_i):= \overline{x}_i' := \overline{x}_i t_B^{w_i} \)  are the {\it controlled transforms of the coordinates.}
According to Lemma~\ref{divisor}, the exceptional divisor on \( B_+ \) is defined by \( t_B^{-1} \), which is a local parameter on \( B \). As a result, the full transform \( \mathcal{O}_B \cdot \mathcal{I} \) is divisible by \( t_B^{-w_A} \), since:
\[
t^{w_A} \cdot \mathcal{O}_B \cdot \mathcal{I} = \mathcal{O}_B \cdot t_B^{w_A} \cdot \mathcal{I} \subset \mathcal{O}_B.
\]

\medskip

\noindent
We define the \emph{controlled transform} of the ideal \( \mathcal{I} \) as:
\[
\sigma^c(\mathcal{I}) := \mathcal{O}_B \cdot t_B^{w_A} \cdot \mathcal{I} \subset \mathcal{O}_B.
\]
\subsubsection{Strict transform of ideals}
\label{decrease1}

Recall that
\[
B_- = B \setminus V(t^{-1}) = X \times \mathbb{G}_m \to X
\]
is the trivial family over \( X \).

The \emph{strict transform} \( \sigma^s(\mathcal{I}) \) of an ideal \( \mathcal{I} \) on \( X \), under a full cobordant blow-up \( \sigma: B \to X \) along a center \( \mathcal{A} \), is defined as the schematic closure of 
\[
(\mathcal{O}_B \cdot \mathcal{I})|_{B_-} = \mathcal{O}_{B_-} \cdot \mathcal{I}.
\]
Equivalently,
\[
\sigma^s(\mathcal{I}) := \left\{ t^a f \in \mathcal{O}_B \,\middle|\, f \in \mathcal{O}_B \cdot \mathcal{I},\ a \geq 0 \right\}.
\]
This definition can be interpreted geometrically.
The strict transform of a closed subscheme \( Y \subset X \) is the schematic closure \( Y^s \) of \( Y \times \mathbb{G}_m \subset B_- =X \times \mathbb{G}_m\) in \( B \). It is defined by
\[
\mathcal{I}_{Y^s} := \sigma^s(\mathcal{I}_Y).
\]

This implies the inclusion:
\[
\sigma^c(\mathcal{I}) \subseteq \sigma^s(\mathcal{I}),
\quad \text{and hence} \quad
\inv(\sigma^c(\mathcal{I})) \geq \inv(\sigma^s(\mathcal{I}))
\quad \text{for any } p \in B.
\]

\begin{remark}
For the cobordant blow-up \( \sigma_+ : B_+ \to X \), the induced strict transform agrees with the classical definition. It is the schematic closure on \( B_+ \) of
\[
\mathcal{O}_{\sigma_+^{-1}(X \setminus V(\mathcal{J}))} \cdot \mathcal{I}_{|X \setminus V(\mathcal{J})} = (\mathcal{O}_B \cdot \mathcal{I})|_{B_+ \cap B_-}.
\]
\end{remark}

\subsubsection{Controlled transforms of Rees algebras and double gradation}

Let \( R = \bigoplus R_a t^a \) be a Rees algebra, and assume \( \mathcal{A} \) is an \( R \)-admissible center so that:
\[
R \subset \mathcal{A}^\ext = \mathcal{O}_X[t^{-1/w}, \overline{x}_1 t^{1/a_1}, \ldots, \overline{x}_k t^{1/a_k}],
\]
with \( w \) a common multiple of the \( a_i \) and $w_R$. By substituting \( t_B \mapsto t_B^w \) and setting \( w_i = w/a_i \), we construct the full cobordant blow-up:
\[
B = \Spec_X \mathcal{O}_B = \Spec_X \mathcal{O}_X[t_B^{-1}, \overline{x}_1 t_B^{w_1}, \ldots, \overline{x}_k t_B^{w_k}]
= \Spec_X \mathcal{O}_X[t_B^{-1}, \overline{x}_1', \ldots, \overline{x}_k'],
\]
where \( \overline{x}_i' := \overline{x}_i t_B^{w_i} \).

Using the double grading in both \( t \) and \( t_B \), we express admissibility on \( B \) as:
\[
\bigoplus R_a t_B^{aw} t^a
\subset \mathcal{O}_B[t_B^{-1} t^{-1/w}, \overline{x}_1 t_B^{w_1} t^{1/a_1}, \ldots, \overline{x}_k t_B^{w_k} t^{1/a_k}]
\subset \mathcal{O}_B[t^{-1/w}, \overline{x}_1' t^{1/a_1}, \ldots, \overline{x}_k' t^{1/a_k}].
\]

\noindent
Hence, the controlled transform of \( R \) on \( B \) is:
\[
\sigma^c(R) := \bigoplus \left(\mathcal{O}_B \cdot R_a \cdot t_B^{aw} \right)t^a,
\]
and is referred to as the \emph{controlled transform of \( R \)}. Consequently if $ft^b\in R$ then $$\sigma^c(f):=ft_B^{wb}t^b\in \sigma^c(R)$$ is called  the {\it controlled transform} of $ft^b$.

\subsubsection{Cobordant blow-ups and admissibility} \label{control}

\begin{lemma} \label{algebra}
Let \( \mathcal{A}^\ext = \mathcal{O}_X[t^{-1/w}, \overline{x}_1 t^{1/a_1}, \ldots, \overline{x}_k t^{1/a_k}]\) be admissible for \( R \), and let
\[
\sigma : B = \Spec_X \left(\mathcal{O}_B[t_B^{-1}, \overline{x}_1 t_B^{w_1}, \ldots, \overline{x}_k t_B^{w_k}] \right)
\]
be the full cobordant blow-up of \( \mathcal{A}^\ext \). Then the Rees center on \( B \)  given by
\begin{align*}
\mathcal{A}^\ext_B := \sigma^c(\mathcal{A}^\ext) 
&= \mathcal{O}_B\left[t^{-1/w},\ \overline{x}_1 t_B^{w_1} t^{1/a_1},\ \ldots,\ \overline{x}_k t_B^{w_k} t^{1/a_k}\right] \\
&= \mathcal{O}_B\left[t^{-1/w},\ \overline{x}_1' t^{1/a_1},\ \ldots,\ \overline{x}_k' t^{1/a_k}\right],
\end{align*}
is admissible for 
$
\sigma^c(R) = \bigoplus \mathcal{O}_B R_a\, t_B^{aw} t^a
$
along the vertex \( V_B(\overline{x}_1', \ldots, \overline{x}_k') \). \qed
\end{lemma}

\subsubsection{Derivations on cobordant blow-up} 
\label{derivations}

Consider the full cobordant blow-up
\[
B = \Spec_X\left( \mathcal{O}_X[t_B^{-1}, x_1 t_B^{w_1}, \ldots, x_k t_B^{w_k}] \right) \to X
\]
of a center 
\[
\mathcal{A}^\ext = \mathcal{O}_X[t^{-1/w}, x_1 t^{1/a_1}, \ldots, x_k t^{1/a_k}].
\]

The sheaf of derivations \( \mathcal{D}_X \) on \( X \) is a coherent \( \mathcal{O}_X \)-module, locally generated by the derivations \( D_{x_i}=\partial_{x_i} \). Using the chain rule, we express the derivations \( D_{x_i'}=\partial_{x_i'} \) on \( B \) as:
\[
\sigma^c(D_{x_i}):=D_{x_i'} = t_B^{-w_i} D_{x_i}, \quad \text{and} \quad \sigma^c(D_{x_i})=D_{x_j'} = D_{x_j}.
\]

We now extend this using the principle of double gradation. For each graded derivation \( t^{-1/a_i} D_{x_i} \), we define its \emph{controlled transform} as:
\[
\sigma^c(t^{-1/a_i} D_{x_i}) = t^{-1/a_i} t_B^{-w_i} D_{x_i} = t^{-1/a_i} D_{x_i'}.
\]

In general, the {\it controlled transform of the sheaf} \( t^{-1/a_1} \mathcal{D}_X \) is defined as the subsheaf
\[
\sigma^c(t^{-1/a_1} \mathcal{D}_X) := t^{-1/a_1} \mathcal{O}_B \cdot t_B^{-w_1} \mathcal{D}_X \subseteq t^{-1/a_1} \mathcal{D}_B,
\]
where \( \mathcal{D}_B \) is the sheaf of derivations on \( B \), graded by \( t^{-1/a_1} \).

The sheaf \( \mathcal{O}_B \cdot t_B^{-w_1} \mathcal{D}_X \) is generated by:
\[
t_B^{-(w_1 - w_i)} t_B^{-w_i} D_{x_i} = t_B^{-(w_1 - w_i)} D_{x_i'}
\quad \text{for } i = 1, \ldots, k,
\]
and
\[
t_B^{-w_1} D_{x_i} = t_B^{-w_1} D_{x_i'}
\quad \text{for } i = k+1, \ldots, n.
\]
The action of derivations commute with controlled transforms:
\begin{lemma} \label{DER}
The action of derivations commutes with controlled transforms. More precisely, we have for $ft^b\in R_bt^b$:
\[
\sigma^c\left(t^{-1/a_i} D_{x_i}\right)\left(\sigma^c\left(f t^{b}\right)\right)
= \sigma^c\left(\left(t^{-1/a_i} D_{x_i}\right)\left(f t^{b}\right)\right),
\]
and similarly for differential operators:
\[
\sigma^c\left(t^{-1/a_1} \mathcal{D}_X\right)\left(\sigma^c\left(R_b t^{b}\right)\right)
= \sigma^c\left(\left(t^{-1/a_1} \mathcal{D}_X\right)\left(R_b t^{b}\right)\right)\subset   t^{-1/a_1}\cD_B \left(R_b t^{b}\right).
\]
\end{lemma}

\subsubsection{The order of the controlled transforms}

\begin{lemma}[\cite{ATW-weighted}]
Let $\sigma: B \to X$ be a cobordant blow-up of an $R$-admissible center 
\[
\mathcal{A}^\ext = \mathcal{O}_X[t^{1/w}, \overline{x}_1 t^{1/a_1}, \ldots, \overline{x}_k t^{1/a_k}],
\]
where $\ord_p(R) \leq a_1$ for some point $p \in X$. Then 
\[
\ord_{p'}(\sigma^c(R)) \leq a_1 \quad \text{for any } p' \in B.
\]
\end{lemma}

\begin{proof}
Write the Rees algebra $R$ locally as 
\[
R = \mathcal{O}_X[f_j t^{b_j}]_{j = 1, \ldots, s}.
\]
If $\ord_p(R) = a_1$, then for some $j$, we have $\ord_p(f_j t^{b_j}) = b_j a_1$, which means there exists a multi-index $\alpha$ with $|\alpha| = b_j a_1$ such that the derivative $\mathcal{D}_{x^\alpha} f_j$ is invertible at $p$.

Now, using Lemma \ref{DER}, we get
\begin{align*}
\cO_B&=\sigma^c(\mathcal{D}_X^{b_j a_1}(f_j ))=
\sigma^c( t^{-b_j}\mathcal{D}_X^{b_j a_1})(\sigma^c(f_j t^{b_j})) \subset \\ \subset\,\, &  t^{-b_j} \mathcal{D}_B^{b_j a_1}(\sigma^c(R_{b_j} t^{b_j})) 
= \mathcal{D}_B^{b_j a_1}(\sigma^c(R))_{b_j}) = \mathcal{O}_B,
\end{align*}
which shows that the order of $\sigma^c(R)$ at any $p' \in B$ satisfies 
\[
\ord_{p'}(\sigma^c(R)) \leq a_1.
\]
\end{proof}


\subsubsection{Controlled transforms of cotangent ideal}
Similarly we have 
\begin{lemma}\label{tangent} \cite{ATW-weighted},   Let 
$\sigma: B\to X$ be a cobordant blow-up of $R=\cO_X[f_jt^{b_j}]$-admissible center $\cA^\ext=\cO_X[ t^{1/w},\overline{x}_1t^{1/a_1},\ldots, \overline{x}_kt^{1/a_k}]$. If $T^{1/a_1}(R)$ is the cotangent ideal for $R$ 
then $$\sigma^c(T^{1/a_1}(R)t^{1/a_1})\subseteq (T^{1/a_1}(\sigma^c(R))t^{1/a_1}.$$ \end{lemma}
\begin{proof} Let $f_jt^{b_j}\in R_{b_j}$, and $|\alpha|=b_ja_1-1$. Then, by Lemma \ref{DER}, 
\begin{align*} & \sigma^c((\cD_X^{ab_j-1}f_j)t^{1/a_1})= \sigma^c((\cD_X^{ab_j-1}t^{b_j-1/a_1})(\sigma^c( f_jt^{b_j}))\\
& \subset (\cD_B^{ab_j-1}t^{b_j-1/a_1})(\sigma^c(R_{b_j}t^{b_j}))=(\cD_B^{ab_j-1}(\sigma^c(R)_{b_j})t^{1/a_1}\subseteq  (T^{1/a_1}(\sigma^c(R)))t^{1/a_1}.
\end{align*}
\end{proof}
\subsubsection{Controlled transforms of a partial maximal contact} \label{control2}

\begin{lemma}[\cite{ATW-weighted}] \label{contact}
Let $\sigma: B \to X$ be a cobordant blow-up of the $R$-admissible center 
\[
\mathcal{A}^\ext = \mathcal{O}_X[t^{1/w}, \overline{x}_1 t^{1/a_1}, \ldots, \overline{x}_k t^{1/a_k}],
\]
where $R = \mathcal{O}_X[f_j t^{b_j}]$ and $\overline{x}_1 = (x_1, \ldots, x_r)$ is a partial maximal contact for $(R, a_1)$ on an open affine subset $U \subseteq X$. Then the controlled transform $\sigma^c(\overline{x}_1) = (x_1', \ldots, x_r')$ is a partial maximal contact for $\sigma^c(R)$ on $\sigma^{-1}(U)$.
\end{lemma}

\begin{proof}
The controlled transform $\sigma^c(\overline{x}_1) = (x_1', \ldots, x_r')$ forms a local system of coordinates. 

If $x_i \in T^{1/a_1}(R) \cap \overline{x}_1$ is free, then $x_i' = \sigma^c(x_i) \in T^{1/a_1}(\sigma^c(R)) \cap \sigma^c(\overline{x}_1)$ is also free. Thus, condition (1) of Definition~\ref{maxim} is satisfied.

If $x_i \in \overline{x}_1$ is divisorial, then by Lemma~\ref{tangent}, we compute:
\begin{align*}
\mathcal{O}_B &= \mathcal{O}_B \cdot D_{x_i}(T^{1/a_1}(R))
= \sigma^c\left( {t^{-1/a_1}}D_{x_i})(T^{1/a_1}(R) t^{1/a_1})\right) \\
&\subseteq  D_{x'_i}\sigma^c(T^{1/a_1}(R))\subseteq  D_{x'_i}(T^{1/a_1}(\sigma^c(R)))=\cO_B,
\end{align*}
which verifies condition (2) of Definition~\ref{maxim}.
\end{proof}



\subsubsection{Restriction of cobordant blow-up to a maximal contact}
\begin{lemma} \cite{ATW-weighted} \label{rest} If $\cA^{\ext}$ is admissible for $R$ then $\cA^{\ext}_{|H}$ is admissible for $R_{|H}$ and $\sigma^c(R)_{|H}=\sigma_H^c(R_{|H})$, where $H:=V(\overline{x}_1)$. 
The restriction of the blow-up  $\sigma_X: B\to X$ of $\cA^{\ext}$ to the strict transform $H_B=V(\overline{x}'_1)$ of $H=V(\overline{x}_1)$ is the cobordant blow-up  $\sigma_H: H_B\to H$ of the restriction  $\cA^{\ext}_{|H}$. \qed
	
\end{lemma}

\subsubsection{Controlled transform of the coefficient ideal}

\begin{lemma}[see also \cite{ATW-weighted}] \label{ccc}
Let $\sigma: B \to X$ be a cobordant blow-up of an $R$-admissible center 
\[
\mathcal{A} = \mathcal{O}_X[\overline{x}_1 t^{1/a_1}, \ldots, \overline{x}_k t^{1/a_k}]^{\inte},
\]
where $R = \mathcal{O}_X[f_j t^{b_j}]$. Then the controlled transform is given by
\[
\sigma^c(R) = \mathcal{O}_B[\sigma^c(f_j t^{b_j})],
\]
and we have the following commutativity:
\[
\sigma^c\left(C_{\overline{x}_1 t^{1/a_1}}(R)\right) = C_{\overline{x}_1' t^{1/a_1}}(\sigma^c(R)).
\]
\end{lemma}

\begin{proof} Write the coefficient ideal of \( R \) at \( \overline{x}_1 t^{1/a_1} \) as
\[
C_{\overline{x}_1 t^{1/a_1}}(R) := \mathcal{O}_{V(\overline{x}_1)}\left[ c_{j\alpha|V(\overline{x}_1)} \, t^{b_j - |\alpha|/a_1} \,\big|\, |\alpha| < b_j a_1 \right],
\]
where
\[
f_j t^{b_j} = \left( \sum_{\alpha} c_{j\alpha} \, \overline{x}_1^\alpha \right) t^{b_j} 
= \sum_{|\alpha| < a_1 b_j} c_{j\alpha} \cdot t^{b_j - |\alpha|/a_1} \cdot \overline{x}_1^\alpha t^{|\alpha|/a_1} 
\mod \left( \overline{x}_1 t^{1/a_1} \right)^{a_1 b_j},
\]
and \( c_{j\alpha} \in i_{D^1}(\mathcal{O}_H) \), where \( D^1 \) consists of the derivations \( \frac{\partial}{\partial x_i} \) for \( x_i \in \overline{x}_1 \). Then
\[
\sigma^c(f_j t^{b_j}) = 
 \sum_{|\alpha| < a_1 b_j} \sigma^c(c_{j\alpha} \, t^{b_j - |\alpha|/a_1}) \cdot \overline{x'}_1^\alpha t^{|\alpha|/a_1} \mod (\overline{x'}_1t^{1/a_1})^{a_1 b_j}
,\] 

where $\sigma^c(c_{j\alpha})\in  i_{\sigma^c(D^1)}(\mathcal{O}_{V(\overline{x'}_1)})$, and
 \[
C_{\overline{x'}_1 t^{1/a_1}}(\sigma^c(R)) = \mathcal{O}_{V(\overline{x}_1)}\left[ \sigma^c(c_{j\alpha} \, t^{b_j - |\alpha|/a_1})_{|V(\overline{x'}_1)}, |\alpha| < b_j a_1 \right]=
\sigma^c(C_{\overline{x}_1 t^{1/a_1}}(R)).\] 
\end{proof}
 
	

\subsubsection{The centers with maximal invariant}

\begin{proposition}\label{lower0}
Let $R = \bigoplus R_a$ be a Rees algebra on a smooth variety $X$ over a field $K$, such that $R_a \neq \mathcal{O}_X$ for all $a \in A$. Then there exists a unique Rees center $\mathcal{A}(R) = \mathcal{A}$ satisfying the following:
\begin{enumerate}
    \item The maximum value $\max \inv(R) = (\overline{b}_1, \ldots, \overline{b}_k)$ of the invariant $\inv_p(R)$ (for $p \in X$) is attained precisely on the closed subset $V(\mathcal{A})$;
    
    \item $\mathcal{A} = \mathcal{O}_X[\overline{x}_1 t^{1/a_1}, \ldots, \overline{x}_k t^{1/a_k}]^\inte$ is a maximal admissible center for $R$, with
    \[
    \inv(\mathcal{A}) = (\overline{b}_1, \ldots, \overline{b}_k), \quad \text{where} \quad \overline{b}_i = a_i \cdot \inv(\overline{x}_i).
    \]
\end{enumerate}
\end{proposition}

\begin{proof}
By Lemma~\ref{max88}, the function $p \mapsto \inv_p(R)$ is upper semicontinuous and attains only finitely many values. Let 
$
S := \supp (\inv(R)= (\overline{b}_1, \ldots, \overline{b}_k))
$  
be the closed subset where the invariant achieves its maximal value.

For any point $p \in S$, let $\mathcal{A}_p$ denote a maximal admissible center for $R$ at $p$. By Lemma~\ref{max8}, there exists a neighborhood $U_p$ of $p$ such that $V(\mathcal{A}_p) \cap U_p = S \cap U_p$ and $\mathcal{A}_p$ is uniquely determined on $U_p$. Since the centers $\mathcal{A}_p$ agree on overlaps $U_p \cap U_q$, they glue together to define a global center $\mathcal{A}$ on $V(\mathcal{A}) = S$.

This defines a unique maximal admissible center $\mathcal{A}$ for $R$ associated with $\max \inv(R)$, as required.
\end{proof}

\subsubsection{Cobordant blow-ups of the centers with maximal invariant}

\begin{proposition}\label{lower}
(see also \cite{ATW-weighted})  
Let $R = \bigoplus R_a$ be a Rees algebra on a smooth variety $X$ over a field $K$, with $R_a \neq \mathcal{O}_X$ for all $a \in A$. 

Let $\mathcal{A} = \mathcal{A}(R)$ be the maximal admissible center associated to the maximum value of the invariant $\inv_p(R) = (\overline{b}_1, \ldots, \overline{b}_k)$, and let $\sigma: B \to X$ be the full cobordant blow-up of $\mathcal{A}^{\ext}$. Then:

\begin{enumerate}
    \item The maximum of $\inv(\sigma^c(R))$ remains equal to $(\overline{b}_1, \ldots, \overline{b}_k)$ and is attained precisely on $V(\sigma^c(\mathcal{A}))$. Moreover,
    \[
    \sigma^c(\mathcal{A}^{\ext}) = \mathcal{O}_B[t^{-1/w}, x'_1 t^{1/a_1}, \ldots, x'_k t^{1/a_k}]
    \]
    is a maximal admissible center for $\sigma^c(R)$.

    \item On the open subset $B_+ := B \setminus V(\sigma^c(\mathcal{A}))$, the invariant strictly decreases:
    \[
    \inv(\sigma^c(R)) < (\overline{b}_1, \ldots, \overline{b}_k).
    \]
\end{enumerate}
\end{proposition}
\begin{proof}
We proceed by induction on $\dim(X)$. If $\dim(X) = 0$, then $\mathcal{A} = R = \mathcal{O}_X = \mathcal{O}_X[0]$, $p = X$, and $\inv_p(R) = ()$ (the empty tuple), which corresponds to an infinite sequence of $\infty$.

Let $p \in V(\mathcal{A})$, and set $\ord_p(\mathcal{A}) = a_1$. Choose a neighborhood $U$ of $p$ such that $\overline{x}_1$ is a maximal contact for $R$ at $p$ and a partial maximal contact on $U$. Then, by Lemma~\ref{Giraud2}, the function $\inv^1$ attains its maximum $\overline{b}_1$ on
\[
\supp( \inv^1(R)= \overline{b}_1) \subset V(\overline{x}_1),
\]
and $\overline{x}_1$ remains a maximal contact along this subset.

For all $p \in \supp (\inv^1(R)=\overline{b}_1)$, we have:
\[
\inv_p(R) = (\inv^1_p(R), \inv_p(C_{\overline{x}_1 t^{1/a_1}}(R)).
\]

Now consider the cobordant blow-up $\sigma: B \to X$, and observe that, by Lemma~\ref{contact}, the controlled transform $\sigma^c(\overline{x}_1) = \overline{x}'_1$ is a partial maximal contact for $\sigma^c(R)$ on $B_U = \sigma^{-1}(U)$. Then, by Lemma~\ref{Giraud2}, $\inv^1(\sigma^c(R))$ attains its maximum $\overline{b}_1$ on
\[
\supp (\inv^1(\sigma^c(R))= \overline{b}_1) \subset H'_1 := V(\overline{x}'_1),
\]
and $\overline{x}'_1$ is a maximal contact on this locus.

By Lemma~\ref{ccc} and Section~\ref{ind}, we have:
\[
\inv_{p'}(\sigma^c(R)) = \left( \inv^1_{p'}(\sigma^c(R)), \inv_{p'}(C_{\overline{x}'_1 t^{1/a_1}}(\sigma^c(R))) \right).
\]

Now, by Lemma~\ref{rest}, the restriction of $\sigma: B \to X$ to $H'_1 = V_B(\overline{x}'_1)$ is the cobordant blow-up $\sigma_{H'_1}: H'_1 \to H_1$ of the restriction $\mathcal{A}^{\ext}_{|H_1}$. By the inductive hypothesis applied to $H_1$ (of dimension $\dim(X) - 1$), both conditions (1) and (2) of the proposition hold for $C_{\overline{x}_1 t^{1/a_1}}(R)_{|H_1}$ and the blow-up $\sigma_{H'_1}$. 

Moreover, $\mathcal{A}^{\ext}_{B|H'_1}$ is a maximal admissible center for the transformed algebra:
\[
\sigma^c_{H'_1}(C_{\overline{x}_1 t^{1/a_1}}(R))
= C_{\overline{x}'_1 t^{1/a_1}}(\sigma^c(R)),
\]
and the maximal value of $\inv(C_{\overline{x}'_1 t^{1/a_1}}(\sigma^c(R)))$ is equal to $\inv(\mathcal{A}^{\ext}_{B|H'_1})$ and is attained at $V(\mathcal{A}^{\ext}_{B|H'_1})$.

Consequently, by Proposition~\ref{max666} and Lemmas~\ref{algebra}, \ref{contact}, we deduce that
\[
\sigma^c(\mathcal{A})
= \mathcal{A}_B^{\ext}
= \mathcal{O}_B[\overline{x}'_1 t^{1/a_1}, \mathcal{A}^{\ext}_{B|H'_1}]
\]
is maximal admissible for $\sigma^c(R)$ at the vertex $V(\mathcal{A}) = V(\mathcal{A}_{|H'_1})$.

Furthermore, the maximal value of
\[
\inv_{p'}(\sigma^c(R))
= \left( \inv^1_{p'}(\sigma^c(R)), \inv_{p'}(C_{\overline{x}'_1 t^{1/a_1}}(\sigma^c(R))) \right)
\]
is equal to
\[
\inv(\mathcal{A}^{\ext}_B)
= \inv(\overline{b}_1, \inv(\mathcal{A}^{\ext}_{B|H'_1})),
\]
and is attained at $V(\mathcal{A}^{\ext}_B) = V(\mathcal{A}^{\ext}_{B|H'_1})$.
\end{proof}

\subsubsection{Resolution Principle}

\noindent
The resolution process for a rational Rees algebra $R$ on a smooth scheme $X$ with an SNC divisor $E$ over a field of characteristic zero follows these key steps:

\begin{center}
\begin{tikzcd}[column sep=huge, row sep=large]
(X, R) \arrow[d, "\inv_p(R)"', dashed] 
  \arrow[r, "\text{Blow-up at } \cA^{\ext}(R)"] 
  & (B, \sigma^c(R)) 
    \arrow[d, "\inv_p(\sigma^c(R))", dashed] 
    \arrow[r, "\text{Remove vertex } \Ver(B)"] 
    & \left(B_+, \sigma^c(R)|_{B_+}\right) \arrow[d, "\inv_p(\cdot)", dashed] \\
\maxinv_X(R) 
  & = \maxinv_B(\sigma^c(R)) 
  & > \maxinv_{B_+}(\sigma^c(R))
\end{tikzcd}
\end{center}

\vspace{1em}

\noindent
\textbf{Step-by-step summary:}
\begin{enumerate}
    \item Compute the invariant $\inv_p(R)$ on $X$. It is upper semicontinuous.
    \item The maximum $\maxinv_X(R)$ is achieved at a unique center $\cA^{\ext}(R) \subset X$.
    \item Blow up $X$ at $\cA^{\ext}(R)$ to obtain the cobordant space $B$, equipped with the transformed Rees algebra $\sigma^c(R)$.
    \item The invariant $\inv_p(\sigma^c(R))$ reaches the same maximum on $B$ at the vertex $\Ver(B) = V(\sigma^c(\cA^{\ext}))$.
    \item Removing the vertex $\Ver(B)$ results in a strict drop of the invariant on $B_+$:
    \[
    \maxinv_X(R) = \maxinv_B(\sigma^c(R)) > \maxinv_{B_+}(\sigma^c(R)).
    \]
\end{enumerate}

\subsection{Properties of the Invariant} \label{prop}
\subsubsection{The invariant $\inv$ at the smooth points}\label{invv}
Assume that $Y$ is a smooth subvariety of codimension  $k$  on a smooth variety $X$ and is described at  $p\in Y$ by a partial set of free local parameters $Y=V(u_1,\ldots,u_k)$  adapted  to an SNC divisor $E$.
Then  $$\cA:=\cO_X[(u_1,\ldots,u_k)t]$$  is a maximal $\cI_Y$- admissible center at $p$, with $$\inv_p(\cI)=(1,\ldots,1),$$ with $k$ entries equal $1$. Conversely, if
$\inv_p(\cI_Y)=(1,\ldots,1)$ is as above then there exists a partial system  of free local parameters $u_1,\ldots,u_k\in \cI$ adapted to  $E$, such that $$\cO_X[\cI_Yt]\subseteq \cO_X[(u_1,\ldots,u_k)t]^\inte=\cO_X[(u_1,\ldots,u_k)t].$$  So $\cI_Y=(u_1,\ldots,u_k)$ is smooth  generated by free coordinates and  adapted  to  $E$ for $Y$ having SNC with $E$ at $p\in X$.

\subsubsection{Torus action} \label{tor}

Suppose \( X \) admits a torus \( T \)-action and the Rees algebra \( R \) is \( T \)-invariant. Then the maximal admissible centers are canonical, hence \( T \)-stable. The algorithm from Section~\ref{algo} can be carried out using semiinvariant maximal contacts and derivations, ensuring that all intermediate Rees algebras \( R_i \) remain \( T \)-stable. As a result, one can inductively choose semiinvariant coordinates for the centers, corresponding to  maximal contacts lying in the \( T \)-stable cotangent ideals \( T^{1/a_i}(R_i) \).

If, in addition, \( X \) admits a geometric quotient \( X/T \) with all orbits of dimension \( \dim(T) \), then any \( T \)-stable center is a smooth \( T \)-invariant subvariety of codimension at most \( \dim(X) - \dim(T) = \dim(X/T) \). Hence, the number of semiinvariant coordinates in such a center is at most \( \dim(X/T) \).

Consequently, the values of the invariant \( \inv_p(\cI) \) throughout the resolution process
\[
X_0 \leftarrow X_1 \leftarrow X_2 \leftarrow \cdots
\]
lie in \( ((\mathbb{Q}_+)_{\geq 0})^k \), where \( k = \dim(X_i) - \dim(\mathbb{G}_m^i) = \dim(X/T) \) is preserved under equivariant blow-ups.
\subsubsection{The descending chain condition}

\begin{lemma}
Let \( \alpha := (a^1, \ldots, a^s) \) be a finite set of positive rational numbers. Let \( \Gamma^k_{\alpha} \subset (\mathbb{Q}_+)^k \) denote the set of possible values of \( \inv_p(R) \), where \( R \) ranges over all \( T \)-stable Rees algebras on smooth varieties \( X \) with a torus action, such that:
\begin{itemize}
    \item a geometric quotient \( X/T \) exists with \( \dim(X/T) = k \), and
    \item \( R \) is generated by homogeneous components \( R_{a^i} t^{a^i} \) for \( i = 1, \ldots, s \).
\end{itemize}
Then the set \( \Gamma_{\alpha} \) satisfies the descending chain condition (DCC).
\end{lemma}

\begin{proof}
Let  \(a := {\rm lcm}(a^1, \ldots, a^s) \). We proceed by induction on \( k = \dim(X/T) \).

For \( k = 1 \), the values of \( \inv_p(R) = \inv^1_p(R) \) lie in \( \frac{1}{a} \cdot \mathbb{N}_+ \), which clearly satisfies dcc.

For general \( k \), note that \( \mathcal{A} \) is \( R \)-admissible if and only if \( R_a t^a \subset \mathcal{A} \), and hence
\[
\inv_p(R) = \frac{1}{a} \cdot \inv_p(R_a),
\]
reducing the problem to the study of ideals \( \mathcal{I} = R_a \subset \mathcal{O}_X \).

For any such ideal \( \mathcal{I} \), the invariant satisfies
\[
\inv_p(\mathcal{I}) = \left( \inv^1_p(\mathcal{I}), \inv_p\left( C_{\overline{x}_1 t^{1/a_1}}(\mathcal{O}_X[\mathcal{I} t])) \right) \right),
\]
where \( C := C_{\overline{x}_1 t^{1/a_1}}(\mathcal{O}_X[\mathcal{I} t])= \bigoplus C_{a/a_1} \cdot t^{a/a_1} \) is generated in degrees \( \frac{1}{a_1}, \ldots, \frac{a_1 - 1}{a_1} \), and \( \dim(H_1/T) < k \).

By the inductive hypothesis on $k$, the set of values \( \inv_p(C) \) satisfies dcc. Since \( \inv^1_p(\mathcal{I}) \in  \bigcup_{i\leq k}(\mathbb{(Z_+)}_{\geq 0})^{ i} \) also satisfies dcc, we conclude that the full invariant \( \inv_p(\mathcal{I}) \) satisfies dcc.

It follows that the set of values of \( \inv_p(R) \) across any resolution sequence also satisfies dcc.
\end{proof}

\subsubsection{Functoriality of the invariant} \cite{ATW-weighted}

\begin{lemma} \label{funct}
The invariant \( \inv_p(R) \) and the associated maximal admissible center \( \mathcal{A} \) are functorial under smooth morphisms, field extensions, and group actions.
\end{lemma}

\begin{proof}
The functoriality of \( \inv_p(R) \) and the center \( \mathcal{A} \) follows from the functoriality of the resolution algorithm described in Section~\ref{algo}. Each step of the algorithm-construction of Rees algebras, derivations, coefficient ideals, and maximal contacts-is compatible with base change, smooth morphisms, and equivariant structures. Therefore, the resulting invariant and center are preserved under these operations.
\end{proof}
\subsection{Motivating Examples}\label{ex}
\subsubsection{Brieskorn Singularities in Characteristic Zero} \label{11}
\begin{example}
Let \( X = \mathbb{A}^n_k = \Spec k[x_1, \ldots, x_n] \) and consider the hypersurface singularity at the origin given by
\[
f = \alpha_1 x_1^{c_1} + \alpha_2 x_2^{c_2} + \cdots + \alpha_n x_n^{c_n}, \quad \alpha_i \in k^\times,
\]
Note that in the case of hypersurface the singularity locus is described by the set of points where the multiplicity $\ord_p(f)\geq 1$, which by Lemma \ref{Giraud3} coincides with the vanishinig locus $V(f, \frac{\partial{f}}{\partial{x}_i})_{i=1}^n=V(x_1,\ldots x_n)$.

We group the monomials by common exponent and write:
\[
f = \overline{\alpha}_1\overline{x}_1^{a_1} + \cdots + \overline{\alpha}_k\overline{x}_k^{a_k},
\]
with
\[
\overline{\alpha}_j \overline{x}_j^{a_j} := \alpha_{i_{j-1}+1}x_{i_{j-1}+1}^{a_j} + \cdots + \alpha_{i_j}x_{i_j}^{a_j}.
\]
We aim to construct the maximal admissible Rees algebra
\[
\mathcal{A} = \mathcal{O}_X[\overline{x}_1 t^{1/a_1}, \ldots, \overline{x}_k t^{1/a_k}]^{\mathrm{int}},
\]
containing \( R := \mathcal{O}_X[f t] \), by iteratively adjoining maximal contact variables and computing coefficient ideals.

\paragraph{\bf Step 1.} Set \( R_1 := R= \mathcal{O}_X[f t]  \), and observe that \( \ord_p(f) = a_1 \). The cotangent ideal is generated of the derivatives of degrees $a_1-1$ of $f$  and is equal to: 
\[
T^{1/a_1}(R_1) = \cD^{a_1-1}(f) = (\overline{x}_1, \overline{x}_2^{a_2 - a_1}, \ldots, \overline{x}_k^{a_k - a_1}),
\]
which includes a maximal contact \( \overline{x}_1\in T^{1/a_1}(R_1) \). Let \( H_1 = V(\overline{x}_1) \).  Write 
$$ft=\overline{\alpha}_1(\overline{x}_1t^{1/a_1})^{a_1}+(\overline{\alpha}_2\overline{x}_2^{a_2}+\ldots +\overline{\alpha}_k x_k^{a_k})t=(\overline{\alpha}_1\overline{x}_1t^{1/a_1})^{a_1}+f_{|H_1} t.$$ in the graded coefficient form with respect to
 the graded  coordinate  $\overline{x}_1t^{1/a_1}$. 
Then the  coefficient ideal  is generated by  the only coefficient $f_{|H_1} \cdot t $ in the presentation of $f  t$.
\[
R_2 := C_{\overline{x}_1 t^{1/a_1}}(R_1) = \mathcal{O}_{H_1}[ f_{|H_1} t],\quad \text{with } f_{|H_1} = \overline{\alpha}_2\overline{x}_2^{{a}_2}+\ldots +\overline{\alpha}_k\overline{x}_k^{{a}_k}
\]

\paragraph{\bf Step 2.} Repeat the process and compute
$T^{1/a_2}_{H_1}(R_2)= (\overline{x}_2, \overline{x}_3^{a_3 - a_2}, \ldots, \overline{x}_k^{a_k - a_2}),$
and maximal contact $\overline{x}_2t^{1/a_2}$. Set  \( H_2 := V(\overline{x}_1,\overline{x}_2) \). $$ f_{|H_1}\cdot  t=\overline{\alpha}_2(\overline{x}_2t^{1/a_2})^{a_2}+(\overline{\alpha}_3\overline{x}_3^{a_3}+\ldots +\overline{\alpha}_kx_k^{a_k})t=\overline{\alpha}_2(\overline{x}_2t^{1/a_2})^{a_2}+ f_{|H_2} t$$
\[
R_3 := C_{\overline{x}_2 t^{1/a_2}}(R_2) = \mathcal{O}_{H_2}[ f_{|H_2} t],
\quad \text{with } f_{|H_2} = \overline{\alpha}_3\overline{x}_3^{{a}_3}+\ldots +\overline{\alpha}_k\overline{x}_k^{{a}_k}.
\]

\paragraph{\bf Inductive Step.}  Set $H_i:=V(\overline{x}_1,\ldots,\overline{x}_i)$.
 $$R_{i+1} := C_{\overline{x}_i t^{1/a_i}}(R_i) = \mathcal{O}_{H_i}[ f_{|H_i} t],\quad \text{with } f_{|H_i} = \overline{\alpha}_{i+1}\overline{x}_{i+1}^{{a}_{i+1}}+\ldots +\overline{\alpha}_k\overline{x}_k^{{a}_k}$$

The process terminates at $R_{k+1}=0$.

The maximal admissible center is 
\[
\mathcal{A} = \mathcal{O}_X[\overline{x}_1 t^{1/a_1}, \ldots, \overline{x}_k t^{1/a_k}]^{\text{int}} .
\]
Thus,  
\( \inv_p(f) =  (\underbrace{a_1, \dots, a_1}_{\#\overline{x}_1}, \dots, \underbrace{a_k, \dots, a_k}_{\#\overline{x}_k})=(c_1, \ldots, c_n) \).

\paragraph{\bf Cobordant Blow-Up.} Consider the extended algebra
\[
\mathcal{A}^{\ext} = \mathcal{O}_X[t^{-1/w}, \overline{x}_1 t^{1/a_1}, \ldots, \overline{x}_k t^{1/a_k}],
\quad \text{with } w = \lcm(a_1, \ldots, a_k).
\]
The full cobordant blow-up is
\[
B = \Spec(\mathcal{O}_X[t^{-1}, \overline{x}_1 t^{w_1}, \ldots, \overline{x}_k t^{w_k}]),
\quad \text{where } w_i = w / a_i.
\]
Going back to the original notation the controlled transform of $f = \sum_{i=1}^n\alpha_i x_i^{c_i} $
becomes
\[
\sigma^c(f) = t^{a_1 w_1} f = \sum_{i=1}^n \alpha_i (x_i')^{c_i},
\quad \text{with } x_i' = x_i t^{w_i}.
\]
The singularity reappears only at the vertex \( V(x_1', \ldots, x_n') \), while the complement \( V(\sigma^c(f))\cap B_+ = V(\sigma^c(f)) \setminus V(x') \) is smooth. Hence, resolution is achieved in a single cobordant blow-up.
\end{example}\subsubsection{Generalizations}

The previous example extends naturally to more general forms:
\begin{example}
Let $X = \Spec k[x_1, \dots, x_n]$, and assume a decomposition of the coordinates into disjoint subsets:
\[
\overline{x}_1, \overline{x}_2, \dots, \overline{x}_k,
\]
such that each $\overline{x}_i$ corresponds to a homogeneous polynomial $F_i(\overline{x}_i)$ of degree $a_i$, with $a_1 < a_2 < \dots < a_k$, and
the ideal $D^{a_i - 1}_{\overline{x}_i}(F_i)$,
generated by the derivatives of $F_i$ in $\overline{x}_i$ of order $a_i-1$ is equal to $
D^{a_i - 1}_{\overline{x}_i}(F_i) = (\overline{x}_i).
$
Let $f = F_1 + \dots + F_k$ and define $\cI = (f)$ and $R_1 = \cO_X[f t]$.

\paragraph{\bf Recursive Construction.} 
We recursively construct Rees algebras by taking the coefficient ideal:
\[
R_{i+1} = C_{\overline{x}_i t^{1/a_i}}(R_i) = \cO_{H_i}\left[ f_{|H_i} \cdot t \right] ,
\]
where $H_i = V_{H_i}(\overline{x}_i)$ is defined by the  maximal contact, and $f_{|H_i} = F_{i+1} + \dots + F_k$ is the restriction of $f_{|H_{i-1}}$ to $H_i$, which appears as the only coefficient in the graded presentation of the generator $f_{|H_{i-1}} t$ of $R_i$:
\[
f_{|H_{i-1}} t = F_i(\overline{x}_i t^{1/a_i}) + f_{|H_i} t,
\]
and this decomposition leads naturally to the next step of the process.

\paragraph{\bf Final Step.} Eventually, we obtain $R_{k+1}=0$ and
\[
  \cA = \cO_X\left[ \overline{x}_1 t^{1/a_1}, \dots, \overline{x}_k t^{1/a_k} \right]^{\mathrm{int}},\quad \text{with the invariant:}
\]
\[
\inv_p(f) = (\underbrace{a_1, \dots, a_1}_{\#\overline{x}_1}, \dots, \underbrace{a_k, \dots, a_k}_{\#\overline{x}_k}).
\]

\paragraph{\bf Cobordant Blow-Up.} The full cobordant blow-up $B$ along the extended algebra:
\[
\cA^{\ext} = \cO_X\left[t^{-1/w}, \overline{x}_1 t^{1/a_1}, \dots, \overline{x}_k t^{1/a_k} \right], \quad w = \mathrm{lcm}(a_1, \dots, a_k),
\]
is given by:
\[
\cO_B = \left( \cA^{\ext} \right)^w = \cO_X\left[t^{-1}, \overline{x}_1 t^{w_1}, \dots, \overline{x}_k t^{w_k} \right], \quad w_i = w/a_i.
\]

The controlled transform of $f$ is:
\[
\sigma^c(f) = t^w f = F_1(\overline{x}_1') + \dots + F_k(\overline{x}_k'), \quad \text{where } \overline{x}_i' = t^{w_i} \overline{x}_i.
\]
which is improved on  $B_+ = B \setminus V(\overline{x}_1', \dots, \overline{x}_k')$ due to the semicontinuity of the invariant.

\end{example}

\subsubsection{Varieties with a divisorial SNC structure}

\begin{example} \label{div2}
Consider the hypersurface singularity:
\[
f = (x_1 + x_2)^2 + x_3^7,
\]
where $x_1, x_2$ are divisorial coordinates at the origin and $x_3$ is free. The resolution proceeds similarly to the free case, but the construction of maximal contact reflects the divisorial structure.

Set $R_1 := \cO_X[ft]$, with $\ord_0(f) = 2$. The cotangent ideal is
\[
T^{1/2}(R_1) = D(f) = ((x_1 + x_2), x_3^6).
\]

As $T^{1/2}(R_1)$ contains no free coordinates the maximal contact for $R_1$ in gradation $t^{1/2}$ is given by the divisorial coordinates $(x_1,x_2)$ containing   $\inn_0(T^{1/2}(R_1))= \left((x_1 + x_2)\right)$. Set $H_1=V(x_1,x_2)$.
Expand $ft$ with respect to $(x_1,x_2)t^{1/2}$:
\[
ft = (x_1 + x_2)^2 t + x_3^7 t = (x_1 t^{1/2}+x_2 t^{1/2})^2 + x_3^7t.
\]
So the coefficient algebra becomes
$
R_2 := C_{x_1t^{1/2}}(R_1) = \cO_{H_1}[x_3^7t],
$
and  has order $7$, and $x_3$ gives maximal contact in degree $t^{1/7}$, with
$
R_3 := C_{x_3t^{1/7}}(R_3) = 0$
giving  
 the maximal admissible center $ \cA=\cO_X[(x_1,x_2)t^{1/2}, x_3t^{1/7})^{\inte}]$ with $\lcm(2,7) = 14$, with  the extended Rees algebra:
\[
\cA^\ext = \cO_X[t^{-1/14}, (x_1,x_2) t^{1/2}, x_3 t^{1/7}],
\]
and the full cobordant blow-up obtained by the rescaling $t\mapsto t^{14}$:
\[
B = \cO_X[t^{-1}, x_1 t^{7}, x_2 t^7, x_3 t^2].
\]
Since $x_1, x_2$ are divisorial, and $x_3$ is free, we conclude:
\[
\inv_0(f) = (2_+, 2_+, 7).
\]

On $B_+ := B \setminus V(x_1', x_2', x_3')$, define $f' := t^{14} f = (x_1' + x_2')^2 + (x_3')^7$. Setting $u := x_1' + x_2'$, we find the new maximal center:
we get
\[
f' = u^2 + (x_3')^7,\quad \inv_{p'}(f') = (2,7) < \inv_0(f).
\]
with
$
\cA' = \cO_{X'}[(u^{1/2}, x_3'^{1/7})^{\inte}].
$
A second cobordant blow-up at $\cA'$ resolves the singularity .

\end{example}
\subsection{Final Conclusions} \label{Final}

\subsubsection{Functorial Principalization in the SNC Setting}
\smallskip
A functorial principalization in the non-SNC setting was previously considered in \cite{ATW-weighted}.

Let \( \mathcal{I} \) be an ideal on a smooth variety \( X \). We initiate the SNC resolution algorithm via a sequence of cobordant blow-ups
\[
\sigma_i: X_{i+1} \to X_i,
\]
performed at the maximal \( \mathcal{I}_i \)-admissible centers \( \mathcal{A}(\mathcal{I}_i) \subset X_i \), where each \( \mathcal{I}_i \) is the controlled transform of \( \mathcal{I}_{i-1} \). This yields a sequence:
\[
X = X_0 \xleftarrow{\sigma_0} X_1 \xleftarrow{\sigma_1} \cdots \xleftarrow{\sigma_{k-1}} X_k = X',
\]
with \( \mathcal{I}_0 := \mathcal{I} \) and \( \mathcal{I}_{i+1} := \sigma_i^c(\mathcal{I}_i) \).

At each step, the invariant strictly decreases:
\[
\max \inv_{X_i}(\mathcal{I}_i) > \max \inv_{X_{i+1}}(\mathcal{I}_{i+1}).
\]
By the descending chain condition (dcc), this process terminates in finitely many steps, at which point:
\[
\max \inv_{X_k}(\mathcal{I}_k) = 0.
\]

Hence, the final controlled transform
\[
\sigma^c(\mathcal{I}) := \sigma_{k-1}^c \circ \cdots \circ \sigma_0^c(\mathcal{I}) = \mathcal{O}_{X_k}
\]
is trivial, and the full transform \( \mathcal{O}_{X_k} \cdot \mathcal{I} \) is locally monomial, generated by the $T$-invariant product $\epsilon\cdot x^\alpha$ of the semiinvariant equations of the components of the exceptional divisor and a semiinvariant unit $\epsilon$.

Moreover, each cobordant blow-up \( X_{i+1} \to X_i \) naturally carries an induced torus action \( T_{i+1} = T_i \times \mathbb{G}_m \). Since all constructions-centers, ideals, and blow-ups-are canonical and functorial under smooth morphisms, they remain \( T_i \)-stable. This proves Theorem~\ref{principalization}.

\subsubsection{Embedded Desingularization}
\smallskip
See also \cite{ATW-weighted} for the non-SNC case.

To resolve the singularities of an irreducible subvariety \( Y \subset X \) of codimension \( k \), we consider the principalization of its ideal \( \mathcal{I} = \mathcal{I}_Y \subset \mathcal{O}_X \). The algorithm proceeds via cobordant blow-ups at maximal \( \mathcal{I} \)-admissible centers \( \mathcal{A} = \mathcal{A}(\mathcal{I}) \).

In the embedded setting, we apply strict transforms of the ideal \( \mathcal{I} = \mathcal{I}_Y \) at each step. The algorithm stops once the invariant reaches the value
\[
\max \inv_X(\mathcal{I}) = (1, \ldots, 1),
\]
with \( k \) entries equal to 1. In this case, the corresponding maximal admissible center satisfies \( \mathcal{A} = \mathcal{O}_X[\mathcal{I}_{Y'} t] \), where \( Y' \) is the strict transform of \( Y \). By Section~\ref{invv}, it follows that \( Y' \) is smooth and has simple normal crossings (SNC) with the exceptional divisor, which is also SNC.

Alternatively, one may run the algorithm using controlled transforms of ideals instead of strict transforms. This version of the process is valid but typically slower, and some centers may not lie entirely within the strict transform of \( Y \).

Both approaches confirm the existence of functorial embedded desingularization by smooth cobordant blow-ups, as stated in Theorem~\ref{embedded}. The outcome is a smooth subvariety \( Y' \subset X' \) having SNC with the SNC exceptional divisor \( E' \), all compatible with an induced torus action.

At each step, the cobordant blow-up \( X_{i+1} \to X_i \) extends the torus action via \( T_{i+1} = T_i \times \mathbb{G}_m \). Since the algorithm is canonical and functorial for smooth morphisms, all centers and strict transforms \( Y_i \subset X_i \) are \( T_i \)-stable.

Passing to geometric quotients yields a sequence of weighted blow-ups,
\[
X = X_0 \xleftarrow{\sigma_0} X_1 / T_1 \xleftarrow{\sigma_{1/T_1}} \cdots \xleftarrow{\sigma_{k-1/T_{k-1}}} X_k / T_k = X' / T,
\]
where \( Y' / T \subset X' / T \) has only quotient singularities. Considering the stack-theoretic quotients \( [X_i / T_i] \), we obtain smooth stacks with a smooth substack \( [Y' / T] \subset [X' / T] \) having SNC with the exceptional divisor. A related approach without requiring SNC exceptional divisors was studied in \cite{ATW-weighted}.

\subsubsection{Nonembedded SNC Resolution}

The nonembedded resolution is derived from the embedded resolution by using local embeddings and functoriality. It produces a resolution with an SNC exceptional divisor.

We define a modified invariant \( \widetilde{\inv}_p(Y) \) for a variety \( Y \) over \( K \), initially without any divisor, as follows. For a closed point \( p \in Y \), embed \( Y \) locally into a smooth variety \( X \). Any two such embeddings \( Y \subset X_1 \) and \( Y \subset X_2 \) into smooth varieties of the same dimension are \'etale equivalent. If \( \dim(X_1) + m = \dim(X_2) \) for some \( m \geq 0 \), then the induced embeddings \( Y \subset X_1 \subset \mathbb{A}^m_{X_1} \) and \( Y \subset X_2 \) are \'etale equivalent. Here, the embedding
\[
X_1 \subset \mathbb{A}^m_{X_1} = \Spec(\mathcal{O}_{X_1}[x_1, \ldots, x_m])
\]
is defined by \( V(x_1, \ldots, x_m) \).

For embeddings \( Y \subset X_1 = V(x_1, \ldots, x_m) \subset \mathbb{A}^m_{X_1} \) and \( Y \subset X_2 \)  with \( \dim(X_1) + m = \dim(X_2) \), we have that  $(x_1, \ldots, x_m)\subset \cI_Y$ and the tuple \( (x_1, \ldots, x_m) \) forms a partial maximal contact of $\cO_{\mathbb{A}^m_{X_1}}\cdot \cI_Y\cdot  t$ in gradation \( t \). Passing to the completion
 the maximal admissibility condition
\[
\widehat{\mathcal{O}}_{\mathbb{A}^m_{X_1},p} \cdot \cI_Y\cdot t = \widehat{\mathcal{O}}_{\mathbb{A}^m_{X_1},p}[(x_1, \ldots, x_m)+ \cI_{Y|X_1}]\cdot t \subset \widehat{\cA}= \widehat{\mathcal{O}}_{\mathbb{A}^m_{X_1},p} [(x_1, \ldots, x_m)t+\cA_{|X_1}],
\]
 is equivalent to the maximal admissibility condition \( \widehat{\mathcal{O}}_{X,p}\cI_{Y|X_1}t \subset \mathcal{A}_{|X_1} \). Therefore,
\[
\inv_p(\cI_{Y,\mathbb{A}^m_{X_1}}) = (1, \ldots, 1, \inv_p(\cI_{Y|X_1})) = (1, \ldots, 1, \inv_p(\mathcal{I}_{Y,X_1})).
\]
By functoriality,
\[
\inv_p(\mathcal{I}_{Y,X_2}) = \inv_p(\mathcal{I}_{Y,\mathbb{A}^m_{X_1}}) = (1, \ldots, 1, \inv_p(\mathcal{I}_{Y,X_1})).
\]

For a fixed embedding \( Y \subset X \), let \( \inv_p(\mathcal{I}_Y) = (b_1, \ldots, b_k) \) and \( \dim(X) = n \). We define the invariant \( \widetilde{\inv}_p(Y) \) to be the equivalence class of sequences \( (b_1, \ldots, b_k)_n \), indexed by $n$, where two sequences are equivalent if
\[
(b_1, \ldots, b_k)_n \sim (1, \ldots, 1, b_1, \ldots, b_k)_{n+m},
\]
with \( m \) additional 1s at the front. This invariant is functorial and independent of the choice of embedding. Comparison between equivalence classes is done lexicographically after fixing representatives with the same ambient dimension \( n \).

\begin{lemma}\label{compp}
Let \( Y \subset X_1 \), \( Y \subset X_2 \) be two embeddings. Let \( \mathcal{A}_1 \) and \( \mathcal{A}_2 \) be maximal admissible centers for \( \mathcal{I}_{Y,X_1} \) and \( \mathcal{I}_{Y,X_2} \), respectively. Let
\[
B_{1+} \to X_1, \quad B_{2+} \to X_2
\]
be the corresponding cobordant blow-ups. Then,
\[
\mathcal{A}_{1|Y} = \mathcal{A}_{2|Y},
\]
and the restrictions of \( B_{1+} \to X_1 \) and \( B_{2+} \to X_2 \) to the strict transform of \( Y \) coincide with the cobordant blow-up
\[
B_{Y+} \to Y
\]
of the common center \( \mathcal{A}_{1|Y} = \mathcal{A}_{2|Y} \).
\end{lemma}

\begin{proof}
Using an argument as in \cite[Lemma 2.5.3]{Wlodarczyk}, one can embed both \( X_1 \) and \( X_2 \) into an affine space \( \mathbb{A}^N \) such that the induced embeddings of \( Y \) into \( \mathbb{A}^N \) coincide. This reduces the situation to the case \( X_1 \subset X_2 = \mathbb{A}^N \), where
\[
\mathcal{I}_{Y,X_1} = \mathcal{I}_{Y,X_2|X_1}, \quad \text{and} \quad \mathcal{A}_1 = \mathcal{A}_{2|X_1}.
\]
Hence \( \mathcal{A}_{1|Y} = \mathcal{A}_{2|Y} \), and the induced morphisms on the strict transforms of \( Y \) agree, as they are both equal to the cobordant blow-up of the same center.
\end{proof}

To construct the resolution, we run the embedded algorithm for \( Y \) using centers associated with \( \max \widetilde{\inv} \). Consider an affine open cover \( \{Y^j\} \) of \( Y \), with closed embeddings \( Y^j \subset X^j \), where all \( X^j \) are smooth and of the same dimension. Define the disjoint unions:
\[
\overline{Y} := \coprod Y^j, \quad \overline{X} := \coprod X^j,
\]
with a closed embedding \( \overline{Y} \subset \overline{X} \), and \'etale projection \( \overline{Y} \to Y \).

Apply embedded desingularization with strict transforms to the pair \( \overline{Y} \subset \overline{X} \). This gives a sequence of cobordant blow-ups
\[
\overline{X} = \overline{X}_0 \xleftarrow{\sigma_0} \overline{X}_1 \xleftarrow{\sigma_1} \cdots \xleftarrow{\sigma_{k-1}} \overline{X}_k = \overline{X}',
\]
with corresponding strict transforms of the ideals \( \mathcal{I}_{\overline{Y}_i} \), and
\[
\overline{Y} = \overline{Y}_0 \xleftarrow{\sigma_{0|\overline{Y}}} \overline{Y}_1 \xleftarrow{\sigma_{1|\overline{Y}}} \cdots \xleftarrow{\sigma_{k-1|\overline{Y}}} \overline{Y}_k = \overline{Y}'.
\]
This sequence descends to a sequence of cobordant blow-ups on \( Y \):
\[
Y = Y_0 \xleftarrow{\sigma_{0Y}} Y_1 \xleftarrow{\sigma_{1Y}} \cdots \xleftarrow{\sigma_{k-1Y}} Y_k = Y'.
\]

The invariant \( \max \widetilde{\inv} \) drops to the minimal value \( (1, \ldots, 1)_n \) on \( \overline{Y}' \) and \( Y' \), where \( n = \dim(\overline{X}) \), and the number of 1's equals the codimension of \( \overline{Y} \subset \overline{X} \). Since the exceptional divisor on \( \overline{X}' \) is SNC and transverse to \( \overline{Y}' \), its restriction defines an SNC divisor on \( \overline{Y}' \). The result descends to \( Y' \), completing a functorial nonembedded SNC resolution of \( Y \), as stated in Theorem~\ref{nonembedded}.

A related non-SNC version of the resolution was studied in \cite{ATW-weighted}.

\section{Resolution of Almost Homogeneous Singularities in Arbitrary Characteristic} \label{free}

 \subsection{The Weighted Normal Bundles to the Centers}

\subsubsection{Weighted Normal Bundle} \label{cone}

Let \( X \) be a regular scheme, and let \( \cJ = (u_1^{1/w_1}, \ldots, u_k^{1/w_k}) \) be a local presentation of a \( \QQ \)-ideal. Define the associated Rees algebra:
\[
\cA_{\cJ} = \cO_X[t^{-1}, u_1 t^{w_1}, \ldots, u_k t^{w_k}]^\inte.
\]
This induces a filtration \( \{ \cA_{\cJ,a} \}_{a \in \ZZ_{\geq 0}} \) on \( \cO_X \), with \( \cA_{\cJ,a} = (\cJ^a)_X \).

The associated graded algebra is the sheaf:
\[
\gr_{\cJ}(\cO_X) = \bigoplus_{a \geq 0} \frac{(\cJ^a)_X}{(\cJ^{a+1})_X} t^a,
\]
which can also be realized as:
\[
\gr_{\cJ}(\cO_X) = \cA_{\cJ}^{\ext}/(t^{-1} \cdot \cA_{\cJ}^{\ext}) = \cO_X[t^{-1}, u_1 t^{w_1}, \ldots, u_k t^{w_k}]/(t^{-1})=\cO_{V(\cJ)}[u_1t^{w_1},\ldots,u_kt^{w_k}].
\]
We define the {\it weighted normal bundle} of \( \cJ \) as:
\[
N_{\cJ}(X) := \Spec(\gr_{\cJ}(\cO_X)).
\]

\begin{remark}
A similar construction appears independently in \cite{Rydh-proj}.
\end{remark}

\subsubsection{Ideal of Initial Forms}

Given \( f \in \cO_{X,p} \), with \( f \in (\cJ^a)_X \setminus (\cJ^{a+1})_X \), we define the {\it initial form}:
\[
\inn(f) \in \left( \frac{(\cJ^a)_X}{(\cJ^{a+1})_X} \right) t^a \subset \gr_{\cJ}(\cO_X).
\]
For an ideal sheaf \( \cI \subset \cO_X \), define the filtration \( \cI_a := \cI \cap (\cJ^a)_X \), and the ideal of initial forms:
\[
\inn(\cI) := \bigoplus_{a \geq 0} \frac{\cI_a + (\cJ^{a+1})_X}{(\cJ^{a+1})_X} t^a \subset \gr_{\cJ}(\cO_X).
\]

\subsubsection{Exceptional Divisor and Weighted Normal Bundle}

\begin{lemma}\label{grad1} (Figure \ref{F3})
Let \( \sigma: B \to X \) be the full cobordant blow-up of the center \( \cJ \). Then the exceptional divisor is:
\[
V_B(t^{-1}) = \Spec_X(\cO_B/(t^{-1})) \cong N_{\cJ}(X).
\]
\qed
\end{lemma}

\begin{figure}[h]
\centering
\begin{tikzpicture}[scale=1.5, yscale=-1, every node/.style={font=\small}, >=Stealth]

  \draw[thick,magenta] (-0.2, 5) -- (4.2, 5) node[above left, magenta] {\Large $X$};
  \draw[very thick, blue]  (0, 0) -- (0, 4) ;
\node[blue] at (0,4.2) {$V$-vertex};
\node[blue] at (-0.2,3.6) {$t^{-1}$};
  \draw[->, thick, blue] (0, 2.1) -- (0, 1.9);

  \filldraw[magenta ] (0,5) circle (1.2pt) node[below right,magenta] {$J$-center};

  \draw[very thick, ->, green!50!black] (0,0) -- (4,0);
  \node[green!50!black] at (2,-0.3) {$\pi_A^{-1}(0)=D = N_{J}(X)$-exc. divisor-Normal bundle};
 \node[green!50!black] at (4.2,0){x'};

 \draw[very thick, green!50!black,dashed] (0,3) -- (4,3);
  \node[green!50!black] at (2,3.3){$\pi_A^{-1}(p)\simeq X$};
   \filldraw[black] (-1.25,3) circle (1.2pt) node[below right,black] {$p$};
  \draw[very thick] (-1.25,0) -- (-1.25,4);
  \filldraw[black] (-1.25,0) circle (1pt) node[above right] {\small 0};
  \node[black] at (-1,4.5) {$\mathbb{A}^1=\Spec(K[t^{-1}])$};
\node[black] at (-1,3.5) {$\mathbb{A}^1$};
  \foreach \a in {0.3, 0.6, 1.0} {
    \draw[thick, domain=0.3:3.8, smooth, variable=\x, postaction={decorate},
      decoration={markings, mark=at position 0.55 with {\arrow{>}}}] 
      plot ({\x}, {\a/\x});
  }

  \draw[->, thick, magenta] (1.5, 2.6) -- (0.8, 4.6);
  \node[magenta] at (2.5,4.4) {$\pi: B \to X = B //\! \mathbb{G}_m$};
  \node[magenta] at (2.5,4.1) {\scriptsize morphism along orbits};
  \node[magenta] at (2.5,3.9) {\scriptsize quotient};

 \draw[->, thick] (-0.3, 0) -- (-0.8, 0);
  \draw[->, thick] (-0.3, 1) -- (-0.8, 1);
  \node at (-0.6,1.5)  {$\pi_A:B \to \mathbb{A}^1$};
  \draw[->, thick] (-0.3, 3) -- (-0.8, 3);

  \node at (-0.6,2) {\scriptsize projection};

\node at (3, 2.2) {\Large $B$};

\end{tikzpicture}

\caption{The cobordant blow-up $\pi: B \to X = B /\ \mathbb{G}_m$ of the center $\mathcal{J}$, together with the secondary projection $\pi_A: B \to \mathbb{A}^1$, which describes $B$ as a deformation space with special fiber the normal bundle $D = N_{\mathcal{J}}(X)$.}
\label{F3}

\end{figure}
\subsubsection{Strict Transform and Initial Form Ideal}

\begin{lemma}\label{grad2}( see also Figure \ref{F4} and Lemma \ref{normal})
Let \( \cI \subset \cO_X \) be an ideal sheaf. Then under the identification \( \cO_B/(t^{-1}) \cong \gr_{\cJ}(\cO_X) \), the restriction \( \sigma^s(\cI)_{|V(t^{-1})} \) of the strict transform \( \sigma^s(\cI) \) to \( V(t^{-1}) \subset B \) corresponds to \( \inn(\cI) \subset \gr_{\cJ}(\cO_X) \).
\end{lemma}

\begin{proof}
For \( f \in \cI \cap (\cJ^a)_X \setminus (\cJ^{a+1})_X \), the strict transform is \( \sigma^s(f) =  ft^a \in \cO_B \). Its reduction modulo \( t^{-1} \) gives:
\[
\sigma^s(f)=ft^a +(\cO_Bt^{-1}\cap \cO_Xt^a)=ft^a+((\cJ^{a+1})_X)t^a=\inn(f) \in \left( \frac{(\cJ^a)_X}{(\cJ^{a+1})_X} \right) t^a ,
\]
hence the identification.
\end{proof}

\subsubsection{Weighted Normal Cone}

The notion of a weighted normal bundle extends to $\ZZ_{\geq 0}$-graded Rees algebras \( \cR = \bigoplus_{a \geq 0} \cR_a t^a \subset \cO_Y[t] \) on a scheme $Y$. Define:
\[
\gr_{\cR}(\cO_Y) = \bigoplus_{a \geq 0} \frac{\cR_a}{\cR_{a+1}} t^a,
\]
and the {\it weighted normal cone}  of \( X \) at \( \cR \) by:
\[
C_{\cR}(Y) := \Spec_{V(\cR)}(\gr_{\cR}(\cO_Y)).
\]

\begin{definition}
Let \( Y \subset X \) be a closed integral subscheme of a regular scheme with ideal sheaf \( \cI_Y \), and let \( \cJ \) be a center on \( X \) with \( V(\cJ) \subset Y \). Define  the {\it weighted normal cone of \( Y \) at \( \cJ \)} as:
\[
C_{\cJ}(Y) := \Spec_{V(\cJ)}(\gr_{\cO_Y \cdot \cA_{\cJ}}(\cO_Y)).
\]
\end{definition}

\begin{lemma}\label{normal}
With the above notation, \( C_{\cJ}(Y) \subset N_{\cJ}(X) \) is defined by the ideal \( \inn(\cI_Y) \subset \gr_{\cJ}(\cO_X) \).
\end{lemma}

\begin{proof}
The morphism \( \phi: \cA_{\cJ} \to \cO_Y \cdot \cA_{\cJ} \) has kernel generated by \( (\cI_Y \cdot \cO_X[t]) \cap \cA_{\cJ} \). Thus,
\[
\gr_{\cO_Y \cdot \cA_{\cJ}}(\cO_Y)) = \bigoplus_a \frac{(\cJ^a)_X}{(\cJ^{a+1})_X + (\cI_Y \cap (\cJ^a)_X)} t^a,
\]
whose kernel in \( \gr_{\cJ}(\cO_X) \) is exactly \( \inn(\cI_Y) \).
\end{proof}

\begin{figure}[ht]
\begin{tikzpicture}[scale=1.5, yscale=-1, every node/.style={font=\small}, >=Stealth]

  \draw[->, thick,magenta] (-0.2, 3.5) -- (3.2, 3.5);
  \draw[very thick,black] (0, 3.5) -- (0.3, 3.5) node[above ,black] {$Y$};

  \draw[very thick, blue]  (0, 0) -- (0, 3) node[below left, blue] {$V$};

 \fill[violet!30, pattern=north east lines] (0,0) rectangle (0.3,3);
 \draw[very thick, violet] (0,0) -- (0.3,0);

  \draw[->, thick, violet!80!black] (0.6, 1.5) -- (0.3, 1.5);
  \node[violet!90!black] at (1,1.5) {\large $\sigma^s({Y})$};


  \node[right] at (0.6,2.4) {$Y \times \mathbb{G}_m=\sigma^s({Y})\cap B_-$};
 \draw[->, thick, violet!80!black] (0.5, 2.4) -- (0.3, 2.4);
  \draw[->, thick, blue] (0, 2.1) -- (0, 1.9);

  \filldraw[magenta ] (0,3.5) circle (1.2pt) node[below left,magenta] {$J$};

  \draw[very thick,black,dashed] (0,0) -- (0.3,0);
  \draw[very thick, ->, green!50!black] (0,0) -- (3,0);
  \node[green!50!black] at (3,-0.2) {$D = N_{J}(X)$};
   \node[black] at (3.2,3.8) {$X$};
   \node[black] at (0.5,-0.5) {Normal Cone of $Y$ at $\cJ$};
   \node[black] at (0.5,-0.2) {$C_{\cJ}(Y)=\sigma^s({Y})\cap D$};
 \draw[->,thick ] (0.4,-0.1)-- (0.1,0);

     \node[black] at (0.2,3.1) {$Y$};

  \draw[very thick] (-1.25,0) -- (-1.25,3);
  \filldraw[black] (-1.25,0) circle (1pt) node[above right] {\small 0};
  \node[black] at (-1,3.4) {$\mathbb{A}^1$};

 \draw[->, thick] (-0.5, 1.5) -- (-0.9, 1.5);

\node at (2, 1.2) {\Large $B$};
\draw[very thick, ->, magenta] (2,2.8) -- (2,3.2);
\end{tikzpicture}
\caption{The strict transform $\sigma^s(Y)\subset B$ of  $Y \subset X$ under the full cobordant blow-up. The transform decomposes into the product $Y \times \mathbb{G}_m$ over $B_- = X \times \mathbb{G}_m$ and the weighted normal cone $C_{\cJ}(Y) = \sigma^s(Y) \cap D$ inside the exceptional divisor $D = N_{\cJ}(X)$.}
\label{F4}
\end{figure}

\begin{remark}\label{deformation}
(See also \cite{Rydh-proj}.)

Assume \( X \) is a variety over an algebraically closed field \( K \), and let \( B \to X \) be the cobordant blow-up along a center \( \mathcal{J} \). Then the secondary projection
\[
\pi: B \to \mathbb{A}^1 = \Spec K[t^{-1}]
\]
defines a flat degeneration of \( X \) to its weighted normal cone along \( \mathcal{J} \). Indeed, for \( a \neq 0 \),
$
\pi^{-1}(a) \cong X,
$
while the special fiber
$
\pi^{-1}(0) = V_B(t^{-1})
$
is the weighted normal cone \( N_{\mathcal{J}}(X) \).
(See Matsumura \cite[Exercise 5, p.~176]{Mat} for an algebraic construction of this degeneration via the Rees algebra.)

The normal cone can be viewed as an infinitely stretched version of the infinitesimal neighborhood of \( \mathcal{J} \), realized via a torus action (See Figures~\ref{F3},\ref{F4}). 

Similarly, for the strict transform \( C = V(\sigma^s(\mathcal{I})) \), the restriction
$
\pi_C: C \to \mathbb{A}^1
$
gives a deformation of \( Y \) to its weighted normal cone
$
C_{\mathcal{J}}(Y) = \pi_C^{-1}(0) 
$.
\end{remark}
\subsection{Almost Homogeneous Singularities and Their Resolution} \label{almost}

Let $X$ be a regular scheme and $\cI$ an ideal. Denote by
\[
\Sing(V(\cI)) := \Sing(\Spec_X(\cO_X/\cI))
\]
the singular locus of the subscheme $V(\cI) \subset X$.

\begin{definition}
Let $Y \subset X$ be an integral closed subscheme defined by $\cI_Y$. A regular subscheme $Z \subset Y$ is called an \emph{almost homogeneous singularity} of $Y$ if:
\begin{enumerate}
  \item $\Sing(Y) = Z$.
  \item There exists a center $\cJ$ on $X$ such that $V(\cJ) = Z$ and the weighted normal cone $C_{\cJ}(Y) \subset N_{\cJ}(X)$ satisfies
  \[
  \Sing(C_{\cJ}(Y)) = V(\inn(\cJ)) = Z.
  \]
\end{enumerate}
\end{definition}

\begin{theorem} \label{Homoge1}
Let $X$ be a regular scheme and $Y \subset X$ an integral closed subscheme with almost homogeneous singularity $Z \subset Y$ for a center $\cJ$. Let $C_{\cJ}(Y) \subset N_{\cJ}(X)$ be the weighted normal cone. If either:
\begin{itemize}
  \item $X$ is universally catenary, or
  \item every component of $C_{\cJ}(Y) \setminus Z$ has codimension equal to $\codim_X(Y)$,
\end{itemize}
then the cobordant blow-up $B_+ \to X$ at $\cJ$ resolves $Y$: the strict transform $Y' \subset B_+$ is regular of codimension equal to $\codim_X(Y)$.
\end{theorem}

\begin{proof}
We may work locally on $X$, so let $X = \Spec A$ be affine and regular, with the center $\cJ$ locally defined as
\[
\cJ = (u_1^{1/w_1}, \ldots, u_k^{1/w_k}) \subset A.
\]
Let $Y = V(\cI_Y) \subset X$ be defined by an ideal $\cI_Y \subset A$, with singular locus $Z = V(\cJ) = \Sing(Y)$. Let $B = \Spec A[t^{-1}, t^{w_1}u_1, \ldots, t^{w_k}u_k]$ be the full cobordant blow-up, and let $B_+ = B \setminus V(t^{-1})$.

Let $Y' = \overline{\sigma^{-1}(Y \times \mathbb{G}_m)}$ be the strict transform of $Y$ in $B$. Since $B_- := X \times \mathbb{G}_m$ is dense in $B$ and $\sigma^{-1}(Y \times \mathbb{G}_m) \subset B_-$ has codimension $d = \codim_X(Y)$, we have:
\[
\codim_B(Y') = d.
\]

Now consider $Y' \cap V(t^{-1}) \subset B$. This corresponds to the initial ideal $\inn_{\cJ}(\cI_Y)$ in the Rees algebra graded ring $\gr_{\cJ}(A) \cong A / \cJ[t^{w_1}u_1, \ldots, t^{w_k}u_k]$. Then:
\[
Y' \cap V(t^{-1}) = V(\inn_{\cJ}(\cI_Y)) \subset N_{\cJ}(X),
\]
the weighted normal cone. By assumption, $\Sing(C_{\cJ}(Y)) = Z = V(\cJ)$, and all other points of $C_{\cJ}(Y)$ are regular. Thus:
\[
\Sing(Y' \cap V(t^{-1})) \subseteq V(u_1t^{w_1}, \ldots, u_kt^{w_k}) = V_B(t^{-1}) \cap V_B(u_1', \ldots, u_k').
\]

Therefore, any point $p \in Y' \cap V(t^{-1}) \setminus V_B(u_1', \ldots, u_k')$ is regular. Since $t^{-1}$ is a non-zero divisor on $\cO_B/\cI_{Y'}$, and $Y' \cap V(t^{-1})$ has codimension $d+1$ in $B$, the regularity of $Y'$ at such $p$ follows.

Next, consider points in $Y' \setminus V(t^{-1}) = \sigma^{-1}(Y \setminus Z) \subset B_-$. Since $Y$ is regular outside $Z$ and the blow-up is an isomorphism over $X \setminus Z$, it follows that $Y'$ is regular at these points as well.

Thus, $\Sing(Y') \subseteq V_B(u_1', \ldots, u_k') \cap V(t^{-1})$, and $Y'$ is regular on $B_+ := B \setminus V_B(u_1', \ldots, u_k')$.

Hence, $Y' \subset B_+$ is regular of codimension $d = \codim_X(Y)$, completing the proof.
\end{proof}

\begin{example}
Let $Y \subset X = \AA^2_{\ZZ} \setminus V(k)$, where $k, p \in \ZZ$, $k > p$, $p$ is prime, and $p \nmid k$, be a scheme over $\Spec(\ZZ)$ defined by
\[
f = x^p + p^p + y^k \in \ZZ[1/k][x, y].
\]
Then
\[
\Sing(f) \subset V\left(f, \frac{\partial f}{\partial x}, \frac{\partial f}{\partial y}\right)
= V(x^p + p^p, y, px^{p-1})
= V(x, y, p).
\]

Make the coordinate change $x' := x + p$, so $x = x' - p$. Then
\[
f = {x'}^p - p \cdot p {x'}^{p-1} + \ldots + p \cdot p^{p-1} x' + y^k.
\]

Now choose weights satisfying $p w_1 = p w_2 + w_1 = k w_3$, and set
\[
\cJ = \left((x')^{1/w_1}, p^{1/w_2}, y^{1/w_3}\right).
\]
Then
\[
\inn(f) = \inn_{\cJ}(f)
= (x')^p - p^p x' + y^k \in \ZZ_p[x', y, p].
\]

Thus,
\[
\Sing(\inn(f)) \subset V\left(\inn(f), \frac{\partial (\inn(f))}{\partial x'}, \frac{\partial (\inn(f))}{\partial y}\right)
= V((x')^p - p^p, -p^p, ky^{k-1})
= V(\cJ).
\]

Therefore, $f$ defines an almost homogeneous singularity, and a single cobordant blow-up resolves it.

However, the naive choice $\cJ_1 = (x^{1/k}, p^{1/k}, y^{1/p})$ fails: here,
\[
\inn(f) = x^p + p^p + y^k
\]
has
\[
\Sing(\inn(f)) = V(x + p, y),
\]
which strictly contains $V(\cJ_1)$. This explains the need for the coordinate change and the appropriate choice of weights.
\end{example}

\subsection{Main Resolution Principle in Arbitrary Characteristic}

The method described generalizes resolution techniques known in characteristic zero and applies to invariants defined in any characteristic.

Let $\cI$ be an ideal on a regular scheme $X$, and let $p \mapsto \Inv_p(\cI)$ be a local invariant with values in a totally ordered set $\Gamma$, satisfying:
\begin{itemize}
  \item \textbf{(Restriction)}: For a regular subscheme $Y \subset X$,  
  $\Inv_p(\cI_{|Y}) \geq \Inv_p(\cI)$.
  \item \textbf{(Product)}: For a regular scheme $Z$ and $p' \in X \times Z$ lying over $p \in X$,  
  $\Inv_{p'}(\cO_{X \times Z} \cdot \cI) = \Inv_p(\cI)$.
\end{itemize}

\begin{theorem} \label{Homoge3}
Assume there exists a weighted center $\cJ$ and a value $\Phi \in \Gamma$ such that for a given ideal $\cI$ on $X$:
\begin{enumerate}
  \item $V(\cJ) \subset X$ lies in the locus where $\Inv(\cI) \geq \Phi$.
  \item $V(\cJ) \subset N_{\cJ}(X)$ lies in the locus where $\Inv(\inn_{\cJ}(\cI)) \geq \Phi$.
\end{enumerate}
Then the cobordant blow-up $\sigma_+: B_+ \to X$ of $\cJ$ satisfies:
\[
\Inv_{B_+}(\sigma^s(\cI)) < \Phi.
\]
\end{theorem}

\begin{proof}
Let $\cJ = (x_1^{1/w_1}, \ldots, x_k^{1/w_k})$. Consider $q \in B_+$. There are two cases:
\begin{itemize}
  \item If $q \in V(t^{-1}) \setminus V(x'_1, \ldots, x'_k)$, then:
  \[
  \Inv_q(\sigma^s(\cI)) \leq \ord_q(\sigma^s(\cI)_{|V(t^{-1})}) = \Inv_q(\inn_\cJ(\cI)) < \Phi.
  \]
  \item If $q \in B_- \setminus V(x'_1, \ldots, x'_k) = (X \setminus V(\cJ)) \times \mathbb{G}_m$, then $\sigma(q) \in X \setminus V(\cJ)$ implies:
  \[
  \Inv_q(\sigma^s(\cI)) < \Phi.
  \]
\end{itemize}
\end{proof}

This principle applies, for instance, to the order function in any characteristic and to the invariant $\inv_p(\cI)$ in characteristic zero.

\begin{example}
Let $k$ be a field of characteristic $p > 2$, and consider the hypersurface $Y \subset X = \Spec k[x,y,z]$ defined by:
\[
f = x^p + y^p z + z^k + x^{p-1} y^2 + x^{p+1} y z,
\]
where $p \nmid k(k-1)$ and $k \geq 2p+1$. Then:
\begin{itemize}
  \item The order of $f$ at the origin is $p$.
  \item Computing the ideal $\cD^{\leq 2}(f)$  generated by $f$ and all its derivatives of order $\leq 2$ , yields $\supp(\ord(f)\geq3) =  V(\cD^{\leq 2}(f))=V(x,y,z)=V(\cJ)$.
\end{itemize}

Solving the weight system:
\[
p w_1 = p w_2 + w_3 = k w_3,
\]
one obtains $\cJ = (x^{1/w_1}, y^{1/w_3}, z^{1/w_3})$ such that $\inn(f) = x^p + y^p z + z^k$ and :
\[
 \supp(\ord(\inn(f)\geq 3) = \cD^{\leq 2}(\inn(f))=V(y^p+kz^k,k(k-1)z^k,x^p)=V(x,y,z)=V(\cJ).
\]
Thus, by Theorem \ref{Homoge3} applied for the order function $\Phi=\ord$ the cobordant blow-up at $\cJ$ reduces the order to $2$, and the singularity can be resolved via further cobordant blow-ups.

\end{example}


\subsection{Homogeneous Subschemes} \label{homog}

Let $V$ be a regular scheme and 
\[
X := \mathbb{A}^n_V = \Spec_V(\cO_V[x_1, \ldots, x_n])
\]
be the affine $n$-space over $V$. Let $\cJ = (x_1^{1/w_1}, \ldots, x_k^{1/w_k})$ be a weighted center on $X$. Then $\cJ$ induces a grading on $\cO_V[x_1,\ldots,x_n]$ via the isomorphism
\[
x_i \mapsto t^{w_i}x_i \text{ for } i \leq k, \qquad x_j \mapsto x_j \text{ for } j > k,
\]
yielding a ring map
\[
\phi: \cO_V[x_1, \ldots, x_n] \to \cO_V[t^{w_1}x_1, \ldots, t^{w_k}x_k, x_{k+1}, \ldots, x_n].
\]
This map sends a monomial $x^\alpha = x_1^{a_1} \cdots x_n^{a_n} \in \cJ^d$ (with $d = \sum_{i=1}^k a_i w_i$) to $t^d x^\alpha$.

Let $f \in \cJ^d \setminus \cJ^{d+1}$ be a homogeneous element of degree $d$. Then
\[
\phi(f) = t^d f \in \cO_V[t^{w_1}x_1, \ldots, t^{w_k}x_k, x_{k+1}, \ldots, x_n],
\]
which corresponds to the strict transform of $f$ in the cobordant algebra
\[
\cO_B = \cO_X[t^{-1}, t^{w_1}x_1, \ldots, t^{w_k}x_k].
\]

\begin{definition}
A closed subscheme of $X$ defined by a homogeneous ideal $\cI \subset \cO_X=\cO_V[x_1, \ldots, x_n]$ (with respect to the grading induced by $\cJ$) is called a \emph{homogeneous subscheme} with respect to $\cJ$.
\end{definition}

This setting gives rise to a useful structural result:
\begin{lemma} \label{homo} (Figures \ref{F3}, \ref{F4})
Let $\cI \subset \cO_V[x_1, \ldots, x_n]$ be a homogeneous ideal with respect to
\[
\cJ = \left(x_1^{1/w_1}, \ldots, x_k^{1/w_k}\right).
\]
Let $\sigma: B \to X$ be the cobordant blow-up of $\cJ$, and let $\pi: X \times \mathbb{A}^1 \to X$ be the natural projection. Then:
\[
B
= \Spec_X\left(\cO_X[t^{-1}, t^{w_1} x_1, \ldots, t^{w_k} x_k]\right)
\]
\[
= \Spec_X\left(\cO_V[t^{-1}, t^{w_1} x_1, \ldots, t^{w_k} x_k, x_{k+1}, \ldots, x_n]\right)
\]
\[
\simeq X \times \mathbb{A}^1
= \Spec_V\left(\cO_V[x_1, \ldots, x_n, y]\right),
\]
via the map defined by:
\[
x_i \mapsto t^{w_i} x_i \quad (i \le k), \quad
x_j \mapsto x_j \quad (j > k), \quad
y \mapsto t^{-1}.
\]

Moreover:
\begin{itemize}
  \item $\cO_Y \cdot \cI$ is sent to the strict transform ideal $\sigma_B^s(\cI)$;
  \item $V(x_1, \ldots, x_n) \subset X \times \mathbb{A}^1$ is mapped to the vertex $V$ of $B$.
\end{itemize}
\qed
\end{lemma}
\begin{example} \label{111}
Let $f = \sum_{i=1}^n \alpha_i x_i^{c_i}$ define a Brieskorn hypersurface in $X = \AA^n_K$. 

Under a cobordant blow-up at \[
\cJ = (x_1^{1/w_1}, \ldots, x_n^{1/w_n}) \quad \text{with } w_ic_i = w_jc_j \text{ for all } i, j
,\] the function transforms to
\[
\sigma^c(f) = \sigma^s(f) = \alpha_1 (x_1')^{c_1} + \cdots + \alpha_k (x_k')^{c_k},
\]
preserving its weighted homogeneity in the new coordinates, with the center replaced by the vertex. In positive characteristic, if the singular locus of $V(f)$ is contained in \( V(x_1, \ldots, x_k) \), then by Theorem \ref{Homoge1} the single cobordant blow-up (after removing the vertex) suffices to resolve the singularity, just as in characteristic zero. However, over non-perfect fields, the general case remains a challenging and unresolved problem. \cite{CPS}
\end{example}
\subsection{The Narasimhan Example Revisited} 
\begin{example}[Narasimhan] \label{NS}
Let \( f = x^2 + y z^3 + z w^3 + y^7 w \in k[x,y,z,w] \), where \( \operatorname{char}(k) = 2 \). The singular locus is:
\[
\Sing(f) = V(f, Y, Z, W), \quad \text{where} \quad
\begin{cases}
Y = D_y(f) = z^3 + y^6 w, \\
Z = D_z(f) = y z^2 + w^3, \\
W = D_w(f) = z w^2 + y^7.
\end{cases}
\]
This is a 1-dimensional toric subvariety of \( \AA^4 \), given by the parametrization \( t \mapsto (t^{32}, t^7, t^{19}, t^{15}) \).

The subscheme \( V(f) \) is homogeneous with respect to the weighted center
\[
\cJ = (x^{1/32}, y^{1/7}, z^{1/19}, w^{1/15}),
\]
so the cobordant blow-up of \( \cJ \) transforms \( f \) without changing its equation. On the chart where \( z \neq 0 \), we observe that
\[
f = x^2 + \frac{1}{z^2} YZ + \frac{y^6 w^4}{z^2},
\]
and after the coordinate change \( X := x + \frac{y^3 w^2}{z} \), we get
\[
f = X^2 + \frac{1}{z^2} Y Z.
\]
Then \( (X, Y, Z) \) is a system of regular parameters which describes the singular locus on the chart where \( z \neq 0 \), and defines a smooth center for a final cobordant blow-up. Similarly on the charts  $w\neq 0$, and $y\neq 0$.

Thus, the singularity resolves in two cobordant blow-ups, while the approaches via characteristic-zero-style invariants $\inv_p(f)$ \ requires at least three steps.

\end{example}

\begin{remark}
This example illustrates that in positive characteristic, the equimultiple locus  may not be contained in any smooth hypersurface-it has embedding dimension 4, but codimension only 3. Hence, \emph{maximal contact fails to exist}, and characteristic-zero techniques based on hypersurface reduction do not apply. Cobordant blow-ups provide a viable alternative.
\end{remark}
\subsection{Cobordant Blow-Ups vs. Classical Blow-Ups at Smooth Centers}
In the cobordant resolution of the Narasinhman example, the second blow-up was at the smooth center \( (X, Y, Z) \) with all weights equal to 1. More generally:

\begin{lemma}
Let \( X \) be a regular scheme. Suppose \( B_+ \to X \) is a cobordant blow-up at a center \( \cJ = (u_1, \ldots, u_k) \) with all weights equal to 1. Let \( \Bl_{\cJ}(X) \to X \) be the classical blow-up at this smooth center. Then the induced quotient morphism
\[
B_+ \to B_+ / \GG_m \simeq \Bl_{\cJ}(X)
\]
is a locally trivial \( \GG_m \)-bundle.
\end{lemma}

\begin{proof}
The statement is local on \( X \), so take \( B = \Spec_X(\cO_X[t^{-1}, u_1 t, \ldots, u_k t]) \). The open cover of \( B_+ \) consists of charts \( B_{u_i t} = \Spec_X(\cO_X[t^{-1}, u_j t, (u_i t)^{-1}]) \). Then the quotient
\[
B_{u_i t} / \GG_m = \Spec_X\left(\cO_X\left[\frac{u_1}{u_i}, \ldots, \frac{u_k}{u_i}\right]\right)
\]
is an affine chart of the classical blow-up \( \Bl_{\cJ}(X) \). Moreover, since \( t^{-1} = (u_i t)^{-1} u_i \), we can express
\[
\cO_{B_{u_i t}} = \cO_X\left[\frac{u_1}{u_i}, \ldots, \frac{u_k}{u_i}\right][u_i t, (u_i t)^{-1}],
\]
so \( B_{u_i t} \simeq \Bl_{\cJ}(X) \times \GG_m \). Hence the morphism is locally a product with \( \GG_m \).
\end{proof}
\footnote{This presentation of \( B_+ / \GG_m \simeq \Bl_{\cJ}(X) \) also appears in \cite[Definition 5.1.5]{HS06}.}
\section{Appendix}
\subsection{Generalized Cobordant Blow-Ups}\label{general}



\subsubsection{ Cobordant Blow-Ups of Rees Algebras}

To treat nonembedded situations, we extend the definition of cobordant blow-ups from smooth schemes to arbitrary Noetherian schemes using Rees algebras instead of $\QQ$-ideals. This avoids requiring normality or integrality.

\begin{definition}
Let $\mathcal{R} = \bigoplus_{a \geq 0} \mathcal{R}_a t^a \subset \mathcal{O}_X[t]$ be a Rees algebra on a Noetherian scheme $X$. The \emph{full cobordant blow-up} of $\mathcal{R}$ is
\[
B := \Spec_X(\mathcal{R}[t^{-1}]) \to X.
\]
We define the \emph{cobordant blow-up} as the open subset $B_+ := B \setminus V(\mathcal{R}_1)$.
\end{definition}

\begin{definition}
A \emph{Rees center} (or generalized center) is any extended Rees algebra $\mathcal{A}^{\ext}$ on $X$, locally generated in the form
\[
\mathcal{A}^{\ext} = \mathcal{O}_X[t^{-w}, f_1 t^{1/a_1}, \ldots, f_k t^{1/a_k}],
\]
where $V(f_1, \ldots, f_k) \subset X$ is a regular subscheme.
\end{definition}

\begin{lemma}
Let $X \to Y$ be a morphism of Noetherian schemes, and let $\mathcal{R}_Y$ be a Rees algebra on $Y$. Then the pullback $\mathcal{R}_X := \mathcal{O}_X \cdot \mathcal{R}_Y$ defines a Rees algebra on $X$, and the corresponding cobordant blow-ups satisfy a natural morphism:
\[
B_X \to B_Y,
\]
compatible with the maps to $X$ and $Y$. If $X \to Y$ is a closed immersion, then so is $B_X \to B_Y$. \qed
\end{lemma}

\subsubsection{Restriction of Centers to Subschemes}

\begin{lemma}
Let $X$ be a smooth variety over a field of characteristic zero, and let $Y \subset X$ be a reduced subscheme not locally contained in an SNC divisor $E \subset X$. Suppose $\mathcal{A}^{\ext} = \mathcal{O}_X[t^{-w}, u_1 t^{1/a_1}, \ldots, u_k t^{1/a_k}]$ is a maximal $\mathcal{I}_Y$-admissible center on $X$. Then the restriction
\[
\mathcal{A}^{\ext}_{|Y} := \mathcal{O}_Y[t^{-w}, f_1 t^{1/a_1}, \ldots, f_k t^{1/a_k}], \quad f_i := u_i|_Y,
\]
is a Rees center on $Y$ whose support $V(f_1, \ldots, f_k)$ is regular.
\end{lemma}

\begin{proof}
Since $\mathcal{I}_Y t \subset \mathcal{A}^{\ext}$, we have $Y \supseteq V(u_1, \ldots, u_k)$. Thus the vanishing locus of the restrictions,
\[
V(f_1, \ldots, f_k) = V(u_1, \ldots, u_k) \cap Y,
\]
is regular, and the algebra $\mathcal{O}_Y \cdot \mathcal{A}^{\ext}$ has the required form.
\end{proof}
\subsection{Cobordant Blow-Ups in the Logarithmic Category}

Cobordant blow-ups extend naturally to the logarithmic category. Let \( X \) be a logarithmically regular scheme (e.g., a strict toroidal variety), and let the center have the form:
\[
\cJ = (u_1^{1/w_1}, \ldots, u_k^{1/w_k}, m_1^{1/w_{k+1}}, \ldots, m_r^{1/w_{k+r}}),
\]
where \( u_i \) are regular parameters and \( m_i \) are monomials from the logarithmic structure. Then the full cobordant blow-up yields:
\[
B = \Spec\left( \cO_X[t^{-1}, t^{w_1}u_1, \ldots, t^{w_k}u_k, t^{w_{k+1}}m_1, \ldots, t^{w_{k+r}}m_r] \right)^{\text{int}},
\]
which remains logarithmically regular. The construction is functorial and compatible with charts, hence applies in general to logarithmic schemes.

Such centers were originally studied in \cite{ATW-principalization, ATW-relative, ATW-weighted,Quek} in the context of Kummer blow-ups and stack-theoretic resolutions. 


\subsection{ Resolution of Foliations via Cobordant Blow-ups}\label{foliations}

Cobordant blow-ups have recently gained attention in the resolution of singular foliations due to their distinctive behavior and advantages over classical techniques (See \cite{ABTW25}).

\begin{itemize}
    \item \textbf{Smooth blow-ups} are often insufficient to resolve singularities of foliations.
    \item \textbf{Weighted blow-ups} are conjectured to reduce foliations to simpler normal forms, but may still retain singularities.
    \item \textbf{Cobordant blow-ups} yield significantly simpler transformation formulas for foliations, often leading to more structured normal forms.
    \item In some cases, cobordant blow-ups produce \emph{nonsingular} foliations - a phenomenon that does \emph{not} occur under weighted blow-ups. (Figure \ref{F5})
\end{itemize}







\begin{figure}[ht]
\centering
\begin{tikzpicture}[scale=1.2, >=stealth]
 \node at (2,-4.9) {\textcolor{black}{ $x\partial_x$}};
\node at (4,-4.8) {\textcolor{black}{\Large $X$}};
\node at (3,-3.3) {\textcolor{black}{\Large $B$}};
\node at (-0.8,-2.3) {\textcolor{blue}{Vertex \large $V$}};
\node at (-0.2,-4.9) {\textcolor{magenta}{\large  ${\mathcal J}$-center}};
 
  \draw[->] (-0.5,0) -- (4.5,0) node[right] {$x'$};
  \draw[->] (0,-4) -- (0,0.5) node[left] {$t^{-1}$};

  \draw[very thick, blue] (0,-4) -- (0,0);

  \foreach \n/\m in {2/2, 2/1, 2/0, 2/-1,
                     1/2, 1/1, 1/0, 1/-1,
                     0/2, 0/1, 0/0, 0/-1,
                    -1/2, -1/1, -1/0, -1/-1} {
    \pgfmathsetmacro\xstart{2^(\n-1)}
    \pgfmathsetmacro\xend{2^\n}
    \pgfmathsetmacro\yval{-2^\m}
    \fill[red] (0,\yval) circle (1.2pt);

    \draw[->, red, thick] (\xstart,\yval) -- (\xend,\yval);
  }
  
  \foreach \n in {-1, 0, 1, 2} {
    \pgfmathsetmacro\xstart{2^(\n-1)}
    \pgfmathsetmacro\xend{2^\n}
    \draw[->, red, thick] (\xstart,0) -- (\xend,0);
  }

  \foreach \n in {-1, 0, 1, 2} {
    \pgfmathsetmacro\xstart{2^(\n-1)}
    \pgfmathsetmacro\xend{2^\n}
    \draw[->, red, thick] (\xstart,-4.5) -- (\xend,-4.5);
     \fill[magenta] (0,-4.5) circle (2pt);
  }

      \foreach \k in {2} {
    \pgfmathsetmacro\a{-2^\k}
    \draw[domain=1:4, smooth, variable=\x, black]
      plot ({\x}, {\a/\x});

  }

      \foreach \k in {1} {
    \pgfmathsetmacro\a{-2^\k}
    \draw[domain=0.5:4, smooth, variable=\x, black]
      plot ({\x}, {\a/\x});

  }

      \foreach \k in {0,-1,-2} {
    \pgfmathsetmacro\a{-2^\k}
    \draw[domain=0.25:4, smooth, variable=\x, black]
      plot ({\x}, {\a/\x});

  }



\end{tikzpicture}
\caption{Cobordant resolution of the vector field $x\partial_x$ on $X$ via a single cobordant  blow-up. The foliation becomes nonsingular $x'\partial_{x'}$ on $B_+ = B \setminus V$.}
\label{F5}
\end{figure}

\providecommand{\bysame}{\leavevmode\hbox to3em{\hrulefill}\thinspace}
\providecommand{\MR}{\relax\ifhmode\unskip\space\fi MR }
\providecommand{\MRhref}[2]{%
  \href{http://www.ams.org/mathscinet-getitem?mr=#1}{#2}
}
\providecommand{\href}[2]{#2}


\end{document}